\documentclass[11pt]{amsart}
\usepackage{mathtools}
\usepackage{amssymb,latexsym,esint,cite,mathrsfs}
\usepackage{verbatim,wasysym}
\usepackage[left=2.6cm,right=2.6cm,top=2.9cm,bottom=2.9cm]{geometry}
\usepackage[T1]{fontenc}
\usepackage{lmodern}
\usepackage[colorlinks=true,urlcolor=blue, citecolor=red,linkcolor=blue,
linktocpage,pdfpagelabels, bookmarksnumbered,bookmarksopen]{hyperref}
\usepackage{tikz,enumitem,graphicx, subfig, microtype, color}
\usepackage{epic,eepic}
\usepackage[english]{babel}
\usepackage{xcolor}

\allowdisplaybreaks

\renewcommand{\epsilon}{\varepsilon}

\numberwithin{equation}{section}
\newtheorem{theorem}{Theorem}[section]
\newtheorem{proposition}[theorem]{Proposition}
\newtheorem{lemma}[theorem]{Lemma}

\theoremstyle{definition}

\newcommand{\R}{\mathbb{R}}

\DeclareMathOperator{\dist}{dist}

\begin{document}

\title[Sign changing bubble tower solutions]
{Sign changing bubble tower solutions to a slightly subcritical  elliptic problem  with non-power nonlinearity}

\author[S.\ Deng]{Shengbing Deng}
\author[F.\ Yu]{Fang Yu}

\address[S.\ Deng]{School of Mathematics and Statistics \newline\indent
Southwest University \newline\indent
Chongqing 400715, People's Republic of China.}
\email{shbdeng@swu.edu.cn}

\address[F.\ Yu]{School of Mathematics and Statistics \newline\indent
Southwest University \newline\indent
Chongqing 400715, People's Republic of China.}
\email{fangyumath@163.com}

\subjclass[2020]{35B33; 35B40; 35J15}

\keywords{non-power nonlinearity;  sign changing bubble tower solutions; Lyapunov-Schmidt reduction.}

\begin{abstract}
We study the following elliptic problem involving slightly subcritical non-power nonlinearity
\begin{eqnarray*}
\left\{ \arraycolsep=1.5pt
   \begin{array}{lll}
-\Delta u =\frac{|u|^{2^*-2}u}{[\ln(e+|u|)]^\epsilon}\ \   &{\rm in}\ \Omega, \\[2mm]
u= 0 \ \  & {\rm on}\ \partial\Omega,
\end{array}
\right.
\end{eqnarray*}
where $\Omega$ is a bounded smooth domain in   $\R^n$, $n\geq 3$,
$2^*=\frac{2n}{n-2}$ is the critical Sobolev exponent,
$\epsilon>0$ is a small parameter.
By  the finite dimensional Lyapunov-Schmidt reduction method,
we  construct a sign changing bubble tower solution with the shape of a tower of bubbles as $\epsilon$ goes to zero.
\end{abstract}

\maketitle

\section{Introduction}\label{intro}

In this paper, we consider the  following elliptic  problem involving slightly subcritical non-power nonlinearity
\begin{align}\label{eqa}
\left\{ \arraycolsep=1.5pt
   \begin{array}{lll}
-\Delta u =\frac{|u|^{2^*-2}u}{[\ln(e+|u|)]^\epsilon}\ \   &{\rm in}\ \Omega, \\[2mm]
u= 0 \ \  & {\rm on}\ \partial\Omega,
\end{array}
\right.
\end{align}
where $\Omega$ is a bounded smooth domain in   $\R^n$, $n\geq 2$,
$2^*=\frac{2n}{n-2}$ is the critical Sobolev exponent  for the embedding  $H^1_0(\Omega)\hookrightarrow L^{2^*}(\Omega)$,
$\epsilon>0$ is a small parameter.

The  main feature of problem (\ref{eqa}) is the  non-power type nonlinearity,
which  is first proposed by
Castro and Pardo \cite{cp1},
they proved the existence of a priori $L^\infty$
bounds for positive solutions of Laplacian
problem   involving the  nonlinearity
$f(u)=\frac{u^{\frac{n+2}{n-2}}}{\ln(2+u)^\alpha}$ with $\alpha>\frac{2}{n-2}$.
Then, Mavinga and Pardo   \cite{np} obtained a priori estimates for positive classical solutions to the following Hamiltonian elliptic system
\begin{eqnarray*}\label{stem}
\left\{ \arraycolsep=1.5pt
   \begin{array}{lll}
-\Delta u =\frac{v^p}{[\ln(e+v)]^\alpha}\ \   &{\rm in}\ \Omega, \\[2mm]
-\Delta v =\frac{u^q}{[\ln(e+u)]^\beta}\ \   &{\rm in}\ \Omega, \\[2mm]
u= v=0 \ \  & {\rm on}\ \partial\Omega,
\end{array}
\right.
\end{eqnarray*}
where $\Omega$  is a bounded convex domain with smooth boundary in $ \R^n$ for $n>2$,
$1<p$, $q<\infty$ and  $\alpha$, $\beta>0$, $\frac{1}{p+1}+\frac{1}{q+1}=\frac{n-2}{n}$.
For more results   of   non-power nonlinearity,
we refer to \cite{cp2,CUD,rp2} for slightly subcritical problem and
\cite{deng}  for supercritical  problem.

On the one hand,   problem (\ref{eqa}) is related to  the following slightly subcritical elliptic problem
\begin{eqnarray}\label{cil}
\left\{ \arraycolsep=1.5pt
   \begin{array}{lll}
-\Delta u =|u|^{2^*-2-\epsilon}u\  &{\rm in}\ \Omega, \\[2mm]
u= 0 \ \  & {\rm on}\  \partial\Omega.
\end{array}
\right.
\end{eqnarray}
When $\epsilon=0$,
Pohozaev  \cite{Poh} proved that  the non existence of  nontrivial solution  if $\Omega$ is a star-sharped domain.
When $\Omega$  is a annulus,
Kazdan and Warner \cite{KAW} obtained the existence of a positive radial solution.
Bahri and Coron \cite{bmc} studied a positive solution provided that $\Omega$  has nontrivial topology.
For the existence of sign changing solution,
there are few results.
When $\Omega=\R^n$, Ding \cite{ding} showed the infinitely many sign changing solutions by Ljusternik-Schnirlman category theory.
In specific case like torii,
Hebey and Vaugon \cite{hvu} investigated the existence and multiple sign changing solutions.
The existence and multiplicity of sign changing solutions are also treated in some contractible domains with an involution symmetry by
Clapp and Weth \cite{cmw}.

When $\epsilon$ is a positive parameter,
problem (\ref{cil}) has   a positive least energy solution $u_\epsilon$,
that is,  $u_\epsilon$ is a solution for the  variational problem
\[
\inf\Big\{\|u\|^2=\int_\Omega|\nabla u|^2dx: u\in H^1_0(\Omega),
\int_\Omega|u|^{2^*-\epsilon}dx=1 \Big\}.
\]
The blow-up phenomenon for positive and  sign changing solutions to   (\ref{cil})  has been studied extensively.
When $\epsilon$ goes to $0$,
Rey \cite{rey} and Han \cite{han}
studied that the  solution to (\ref{cil})
blows up and concentrates at a critical point of   Robin  function.
Moreover,
Flucher and Wei \cite{fwe} proved that the concentration point is the minimum point  of the  Robin  function.
Furthermore,
if $\xi^*$ is a stable critical point of   Robin  function,
then   (\ref{cil}) has a positive solution which blows up at $\xi^*$,
this result  is obtained in  \cite{rey,map}.
For the multiple concentration points,
Rey \cite{resy} showed that
the two  blow   up and concentration points $(\xi_1^*,\xi_2^*)$,
which  is a critical point of a function involving   Robin  function
and Green's function.
If the domain is convex,
Grossi and  Takahashi in \cite{gti} proved that (\ref{cil}) does not admit
any positive solution  blowing up more than two points.
The positive solution to (\ref{cil}) concentrate
simultaneously at different points $\xi_1,\cdots,\xi_k\in\Omega$, $k\geq 2$, has been established  in  \cite{map,BA}.
If any $\xi_i$, $i=1,\cdots, k$, is a simple blow up point,
Li \cite{liy} characterized the  form of solution $u_\epsilon$
near each blow up point $\xi_i$ as
$$
u_\epsilon(x)\sim
 \frac{\mu_i\sqrt{\epsilon}}
{(\mu_i^2\epsilon+|x-\xi_i|^2)^{\frac{n-2}{2}}},
\quad\mbox{with}\  \mu_i>0.
$$

On the other hand, the existence of one sign changing solution   to (\ref{cil}) is first proved in \cite{baw,cacn},
and multiple sign changing solutions with their nodal properties are treated in \cite{bwe,bas}
for $\epsilon\in(0,\frac{4}{n-2})$.
Moreover, they proved that   (\ref{cil}) has a least energy nodal solution  with two nodal domains.
Ben Ayed et al. \cite{bmp} obtained   that the low energy sign-changing solutions blow up at two points,
and the  energy converges to the value $2S^{\frac{n}{2}}$,
where  $S$ is the Sobolev constant for the embedding  $H_0^1(\Omega)$ into $L^{\frac{2n}{n-2}}(\Omega)$.
Bartsch et al. \cite{Bmp} considered that
  (\ref{cil}) has $k$ pairs of sign changing solutions $\pm u_\epsilon^{(i)}$,  $i=1,\cdots, k$,
which satisfies that $u_\epsilon^{(i)}$  blows up positively at a point $\xi_1^{(i)}\in\Omega$
and $-u_\epsilon^{(i)}$  blows up negatively at a point $\xi_2^{(i)}\in\Omega$
with $\xi_1^{(i)} \neq \xi_2^{(i)}$.
Bartsch et al. \cite{Bap} proved a sign changing four-bubble solution  with two positive and two negative blow-up points
provided that $\Omega$
is convex and satisfies some symmetry conditions.

In contrast  to the result of positive and sign changing solutions
to   (\ref{cil}),
there are some papers of  bubble tower.
If $\Omega$ is a smooth bounded domain in $\R^n$ symmetric with respect to $x_1,\cdots,x_n$ and contains the origin,
Pistoia and Weth \cite{paw} constructed a sign changing bubble tower solution $u_\epsilon$ concentrating at the center of symmetry of $\Omega$.
The same consequence in any bounded smooth domain is considered in \cite{gener},
and they removed
the assumption on non-degeneracy of  critical point of Robin's function.
If the domain has holes like $\Omega\setminus B(a,\epsilon)\cup B(b, \epsilon) $ with center at point $a$, $b$ and radius $\epsilon>0$,
Ge et al. \cite{gmp} constructed sign changing solutions  blowing up both at $a$ and $b$.
For any other bubble tower results of  elliptic problem, see \cite{CLY,CS,del,con,Cdj}  and references therein.
In particular,
we refer to the papers \cite{DAG,cog} for fractional and biharmonic operators involving almost critical Soblev exponent. 

Recently,
by Lyapunov-Schmidt reduction method, Clapp et al. \cite{mp,bay} constructed solutions to problem (\ref{eqa}).
Before stating the results, let us introduce some definitions and notations.
For $\xi \in \Omega$ and $\mu>0$,  let
\begin{equation}\label{bubble}
U(x)=\alpha_n \frac{1}{(1+|x|^2)^{\frac{n-2}{2}}},
\quad
U_{ \mu,\xi}(x)=\frac{\alpha_n \mu^{\frac{n-2}{2}}}{( \mu^2 +|x-\xi|^2)^{\frac{n-2}{2}}},
\quad  \mbox{with}\ \alpha_n=(n(n-2))^{\frac{n-2}{4}},
\end{equation}
which   are the only solutions of the equation
\begin{equation}\label{limiequ}
-\Delta u=u^{2^*-1},
\ \ u>0 \quad\mbox{in}\ \ \mathbb{R}^n.
\end{equation}
Let us denote by
$
G(x,y)
$
the Green's function of $-\Delta$ in $\Omega$ with  Dirichlet  boundary condition, 
and by $H(x,y)$  its regular part, so that
\[
H(x,y)=\frac{1}{(n-2)|\partial B|} \Big(\frac{1}{|x-y|^{n-2}}-G(x,y)\Big), \quad\mbox{for\ every } \ x,y\in\Omega,
\]
where $|\partial B|$  denotes the surface area of the unit sphere in  $\R^n$.
The Robin function is defined as $ \varphi(x)=H(x,x)$ for every $x\in\Omega$.
Let $\xi_*$ is a non-degenerate critical point of   Robin function,
Clapp et al. \cite{mp} constructed
a single bubble solution of the form
\[
u_\epsilon= U_{\mu_\epsilon, \xi_\epsilon}+\phi_\epsilon,
\]
with
$\mu_\epsilon\Big(\frac{|\ln\epsilon|}{\epsilon}\Big)^{\frac{1}{n-2}}\rightarrow d>0$,
$\xi_\epsilon\rightarrow\xi_*$,
$\phi_\epsilon\in H_0^1(\Omega)$ such that
 $\int_\Omega|\nabla \phi_\epsilon|^2dx= O\Big(\frac{\epsilon}{|\ln\epsilon|}\Big)$
as $\epsilon\rightarrow0$.
Liu et al. \cite{ZZW}   established a solution concentrating at the origin point for a critical H\'{e}non problem with non-power type.
Ben Ayed et al. \cite{bay} obtained  positive as well as sign changing solutions concentrating at several points,
which involving   Robin  function
and Green's function.

In present paper,
motivated by   several results \cite{mp,bay,paw,gener},
we construct  a solution with the shape of a tower of sign changing bubbles  to
problem (\ref{eqa}) by finite Lyapunov-Schmidt dimensional reduction procedure.

Our result can be stated as follows.

\begin{theorem}\label{kia}
Assume that $n\geq 3$, for any integer $k\geq 1$, there exists $\epsilon_0>0$ such that for every $\epsilon\in(0,\epsilon_0)$,
there are some points $\xi_{i_\epsilon}\in\Omega$
and positive constants $d_{i_\epsilon}$   for $i=1,\cdots k$,
problem (\ref{eqa}) has a  solution  $u_\epsilon$  of the  form,
$$
u_\epsilon(x)=\alpha_n\sum_{i=1}^k(-1)^i
\bigg(\frac{d_{i_\epsilon} (\frac{\epsilon}{|\ln\epsilon|^2} )^{\frac{2i-1}{n-2}}}
{ \Big(d_{i_\epsilon}(\frac{\epsilon}{|\ln\epsilon|^2})
^{\frac{2i-1}{n-2}}\Big)^2
+|x-\xi_{i_\epsilon}|^2}\bigg)^{\frac{n-2}{2}} +\Theta_\epsilon(x),
$$
where  $\|\Theta_\epsilon\| \rightarrow 0$ as $\epsilon\rightarrow 0$,
$\varphi(\xi_{i_\epsilon})\rightarrow \min\limits_{z\in\Omega}\varphi(z)$
and
$d_{i_\epsilon} \rightarrow d_i>0$ for $i=1,\cdots k$.
\end{theorem}

Observe that in the above construction,
the solutions behaves like a superposition of bubbles of different
blow up orders centered at around the minimum point of the Robin function,
thus, it is called bubble tower solutions.
It was first studied  by del Pino et  al.  \cite{pmf} for a slightly supercritical Brezis-Nirenberg problem in a ball,
and this type solutions  has been constructed in many problems, see
\cite{paw,gener,gmp,cog,DAG,Chx} and references therein.

The   paper is organized as follows.
In Section 2, we    give the scheme of proof for Theorem \ref{kia}.
We  show the finite dimensional reduction process  in Section 3.
Proposition \ref{leftside} is proved in Section 4.
Finally, there are some estimates in the Appendix.
We will use  $C>0$ to
denote various positive constants.

\section{Scheme of the proof}

In this section, let us  give the sketch proof of  Theorem \ref{kia}.
We first introduce some notations.
The Sobolev space $ H_0^1(\Omega)$ is endowed
with inner product $\langle \cdot,\cdot\rangle$ defined by
\[
\langle u,v\rangle=\int_\Omega \nabla u \nabla vdx,
\]
for all $u$, $v\in  H_0^1(\Omega)$,
and  $L^p(\Omega)$ is the Lebesgue space  with the norm
$
|u|_{q}=\Big(\int_\Omega |u|^qdx\Big)^{\frac{1}{q}},
$
for $1<q<\infty$.

Let $i^*:L^{\frac{2n}{n+2}}(\Omega)\hookrightarrow   H_0^1(\Omega)$
be the adjoint operator of the embedding  $i: H_0^1(\Omega)\hookrightarrow L^{\frac{2n}{n-2}}(\Omega)$,
that is, for  $v\in L^{\frac{2n}{n+2}}(\Omega)$,
$u=i^*(v)$ if and only if
\[
-\Delta u=v\quad \mbox{in}\ \Omega,\ \
u=0\quad \mbox{on}\ \partial\Omega.
\]
Then,  it holds
\begin{equation}\label{embed}
\|i^*(v)\|\leq c|u|_{\frac{2n}{n+2}},
\end{equation}
for some constant $c>0$  depending only on $\Omega$ and $n$.
Using these definitions and notations,
problem (\ref{eqa}) is equivalent
to the following equation
\[
u=i^*[f_\epsilon(u)],\quad
u\in H_0^1(\Omega),
\]
where $f_\epsilon(u)=\frac{|u|^{2^*-2}u}{[\ln(e+|u|)]^\epsilon}$.

In order to describe the shape of the solutions to problem (\ref{eqa}),
we give  an integer number $k$,
and define the positive parameters
$\mu_i$  as
\begin{equation}\label{delta}
\mu_i=\Big(\frac{\epsilon}{|\ln\epsilon|^2}\Big)^{\frac{2i-1}{n-2}}d_i,\quad\mbox{with}\quad d_i>0,\quad i=1,\cdots,k.
\end{equation}
Let $\xi$ be a point in $\Omega$,
$\xi_i\in \Omega$, $i=1,\cdots,k$, is given by
\begin{equation}\label{deltas}
\xi_i=\xi+\mu_i\sigma_i,\quad\mbox{for\ some\ points}\ \sigma_i \in\R^n,
\end{equation}
where  $\sigma_k=0$.
We will assume the following bounds on the
parameters and points appearing in (\ref{delta}) and (\ref{deltas}):
given $\eta>0$  small,
\begin{equation}\label{eta}
\dist(\xi,\partial\Omega)>\eta,\quad
\eta<d_i<\frac{1}{\eta},\quad
|\sigma_i|\leq\frac{1}{\eta},\quad i=1,\cdots,k.
\end{equation}
It is an immediate observation that
\[
\mu_1=\Big(\frac{\epsilon}{|\ln\epsilon|^2}\Big)^{\frac{1}{n-2}}d_1\quad\mbox{and}\quad
\frac{\mu_{i+1}}{\mu_i}
=\Big(\frac{\epsilon}{|\ln\epsilon|^2}\Big)^{\frac{2}{n-2}}\frac{d_{i+1}}{d_i}.
\]
We denote by  $PU_{ \mu,\xi}$ the projection onto  $H^1_0(\Omega)$ of $U_{ \mu,\xi}$, that is
\begin{eqnarray*}
\left\{ \arraycolsep=1.5pt
   \begin{array}{lll}
-\Delta  PU_{ \mu,\xi}=-\Delta  U_{ \mu,\xi}\ \   &{\rm in}\ \Omega, \\[2mm]
PU_{ \mu,\xi}=0\ \  & {\rm on}\ \partial\Omega.
\end{array}
\right.
\end{eqnarray*}
Let $k\geq 1$,
the approximate solutions  are given by
\begin{equation}\label{solution}
u(x)=V(x)+\phi(x),\quad V(x)=V_{\bar{d},\bar{\sigma},\xi}(x)=\sum_{i=1}^k(-1)^iPU_{\mu_i,\xi_i}(x),
\end{equation}
where 
\begin{equation}\label{com}
\bar{d}=(d_1,\cdots,d_k)\in \R^k_+, \
\bar{\sigma}=(\sigma_1,\cdots,\sigma_k)\in (\R_+^{n})^k.
\end{equation}
The term $\phi$ is small in some
sense.

Let us  describe  $\phi$.
As it is shown in \cite{bEG}, any solution  of
\begin{equation}\label{linear-equ}
-\Delta \psi=f_0^{'}(U_{ \mu,\xi})\psi \quad\mbox{in}\ \mathbb{R}^n,
\end{equation}
can be expressed as a linear combination of
\begin{equation}\label{psi0}
\psi^0(y)=\frac{(n-2)\alpha_n}{2}
\frac{|
y|^2-1}{(1+|y|^2)^{\frac{n }{2}}},\quad
\psi^h(y)= (n-2)\alpha_n
\frac{y_h}{(1+|y|^2)^{\frac{n }{2}}},
  \quad\mbox{for}\ h=1,\cdots,n.
\end{equation}
Moreover, we set
\begin{align}\label{psi}
& \psi^0_{ \mu,\xi}(x)= \mu^{-\frac{n-2}{2}}\psi^0(\frac{x-\xi}{ \mu})
=\frac{n-2}{2} \alpha_n \mu^{\frac{n-2}{2}}  \frac{|x-\xi|^2-\mu^2}{(\mu^2+|x-\xi|^2)^{\frac{n}{2}}},\nonumber\\
&  \psi^h_{\mu,\xi}(x)
=\mu^{-\frac{n-2}{2}} \psi^h(\frac{x-\xi}{\mu})
=(n-2)  \alpha_n  \mu^{\frac{n}{2}} \frac{x_h-\xi_h}{(\mu^2+|x-\xi|^2)^{\frac{n}{2}}}, \ \mbox{for} \ h=1,\cdots,n,
\end{align}
then
\begin{equation}\label{xy}
\psi^0_{ \mu,\xi}(x)=
 \mu\frac{\partial U_{ \mu,\xi}  }{ \partial \mu},\quad
\psi^h_{ \mu,\xi}(x)=
 \mu\frac{\partial U_{ \mu,\xi}  }{ \partial \xi_{h}}.
\end{equation}
We denote that $ P\psi^h_{ \mu,\xi}$ is the projection  of  $\psi^h_{ \mu,\xi}$,  $h=0,\cdots,n$,
and define the subspace of $H^1_0(\Omega)$
\begin{align*}
& E_{\mu,\xi} =
\mbox{span} \left\{P\psi^h_{ \mu,\xi}: h=0,1,\cdots,n,\ i=1,\cdots,k\right\},\\
& E_{\mu,\xi}^\perp =
\left\{\phi\in H^1_0(\Omega):  \langle \phi,P\psi^h_{ \mu,\xi} \rangle =0: h=0,1,\cdots,n,\ i=1,\cdots,k\right\}.
\end{align*}
Let
\[
\Pi_{ \mu,\xi}:  H_0^1(\Omega)\rightarrow  E_{ \mu,\xi}
\quad\mbox{and}\quad \Pi^{\bot}_{ \mu,\xi}:  H_0^1(\Omega)\rightarrow  E^{\bot}_{ \mu,\xi},
\]
be the  corresponding projections.
To solve (\ref{eqa}),  it is equivalent to solve the couple of following equations
\begin{equation}\label{ng2}
\Pi^{\bot}_{ \mu,\xi}
\Big(V+\phi
-i^*[ f_\epsilon(V+\phi ]\Big)=0,
\end{equation}
and
\begin{equation}\label{ng1}
\Pi_{  \mu,\xi}
\Big(V+\phi
-i^*[ f_\epsilon(V+\phi )]\Big)=0.
\end{equation}

We solve  equation (\ref{ng2}) in the following result,
whose proof can be found in Section 3.

\begin{proposition}\label{pro21}
There exists $\epsilon_0>0$ such that for any $\xi\in\Omega$, $\bar{d} \in \R^k_+$, $\bar{\sigma} \in (\R_+^{n})^k$ satisfying (\ref{eta}), for   $\epsilon\in(0,\epsilon_0)$,
there is a unique  function
$\phi \in E^{\bot}_{ \mu,\xi}$ which solves (\ref{ng2}).
Moreover
\begin{eqnarray}\label{phi}
\|\phi \|
=
\left\{ \arraycolsep=1.5pt
   \begin{array}{lll}
  O\bigg( \frac{\epsilon}{|\ln\epsilon|^2} \Big|\ln  \frac{\epsilon}{|\ln\epsilon|^2}\Big|\ln\Big|\ln \frac{\epsilon}{|\ln\epsilon|^2}\Big| \bigg)\ \   &{\rm if }\   3\leq n\leq 6 , \\[2mm]
 O\bigg(\Big(\frac{\epsilon}{|\ln\epsilon|^2}\Big)
 ^{\frac{n+2}{2(n-2)}}\bigg)\ \  & {\rm if}\   n\geq7.
\end{array}
\right.
\end{eqnarray}
\end{proposition}

From Proposition \ref{pro21}, there is a unique  $\phi \in E^{\bot}_{ \mu,\xi}$ such that (\ref{ng2}) holds,
it means that there are some constants $c_{il}s$ ($i =1,\cdots,k$, $ l=0,\cdots,n$) such that
\begin{equation}\label{ccf}
V+\phi -i^*[f_\epsilon(V+\phi )]
=\sum_{i=1}^k\sum_{l=0}^nc_{il} P\psi^l_{ \mu_i,\xi_i},
\end{equation}
which equals to  solve  equation (\ref{ng1}),
that is, the following result is valid,
 whose proof is postponed to Section 4.

\begin{proposition}\label{leftside}
The following facts hold. \\
Part a. If  $(\bar{d}_\epsilon,\bar{\sigma}_\epsilon,\xi_\epsilon)$ satisfies
\begin{align}\label{aosd}
   \Big\langle V+\phi -i^*[f_\epsilon(V+\phi )], P\psi^h_{ \mu_{j_\epsilon},\xi_{j_\epsilon}}\Big\rangle
  = 0, \quad\mbox{for}\ \ h=0,\cdots,n,
\end{align}
where $j =1,\cdots,k$.  Then $V+\phi $ is a solution of problem (\ref{eqa}).\\
Part b.
For $\xi\in\Omega$,
$\bar{d}=(d_1,\cdots,d_k)\in \R^k_+, \bar{\sigma}=(\sigma_1,\cdots,\sigma_k)\in (\R_+^{n})^k$,
there holds\\
\begin{eqnarray*}
 \quad\Big\langle V+\phi -i^*[f_\epsilon(V+\phi )], P\psi^h_{ \mu_j,\xi_j}\Big\rangle
=
\left\{ \arraycolsep=1.5pt
   \begin{array}{lll}
  \frac{ \epsilon}{|\ln\epsilon|^2} G_0^\epsilon(\bar{d},\bar{\sigma},\xi)
 - \frac{2k^2}{(n-2)^2} a_4 \epsilon \Big|\ln \frac{\epsilon}{|\ln\epsilon|^2}\Big|  \ \  {\rm for}\   h=0,\\[2mm]
 \frac{ \epsilon}{|\ln\epsilon|^2} G_h^\epsilon(\bar{d},\bar{\sigma},\xi)\ \  {\rm for}\   h=1,\cdots,n,
\end{array}
\right.
\end{eqnarray*}
where $j =1,\cdots,k$, and
$G^\epsilon=(G_0^\epsilon,G_h^\epsilon)$ is  given by
\begin{eqnarray*}
\left\{ \arraycolsep=1.0pt
   \begin{array}{lll}
  G_0^\epsilon(\bar{d},\bar{\sigma},\xi)
 =    \alpha_n a_1  d_1 ^{ n-2 }\varphi(\xi)
 + a_3 \sum\limits_{i=1}^{k-1}\Big(\frac{d_{i+1}}{ d_i}\Big)^{\frac{n-2}{2}} g(\sigma_i)
  - a_4\sum\limits_{i=1}^k \frac{2}{2i-1}|\ln  d_i|
 + o(1),   \\[2mm]
 G_h^\epsilon(d_1,\xi)
=\frac{\alpha_n }{2} a_2\partial_{\xi_h}\varphi(\xi)
d_1^{ n-1 }\ \  {\rm for}\   h=1,\cdots,n,
\end{array}
\right.
\end{eqnarray*}
with $G_0^\epsilon:[0,+\infty]\times [0,+\infty]\times \Omega \rightarrow \mathbb{R}\times \mathbb{R}\times\mathbb{R}^n$,
$G_h^\epsilon: [0,+\infty]\times \Omega \rightarrow \mathbb{R}\times\mathbb{R}^n$
and
\begin{align*}
& a_1=(2 ^*-1) \int_{\mathbb{R}^n}  U^{2 ^*-2}(y)  \psi^0(y) dy,\\
& a_2=\int_{\mathbb{R}^n} U^{2 ^*-1} (y)dy,\\
& a_3=\frac{n-2}{2}\alpha_n^{p+1},\\
& a_4=\int_{\mathbb{R}^n}\bigg|
\frac{1}{(1+|y-\sigma_i|^2)^{\frac{n+2}{2}}}
\ln\Big(\frac{1}{(1+|y-\sigma_i|^2)^{\frac{n+2}{2}}}\Big)\psi^0(y)\bigg|dy>0,\\
& g(\sigma)=\int_{\R^n} \frac{  y^{2-n}}{(1+|y-\sigma|^2)^{\frac{n+2}{2}}}dy.
\end{align*}
\end{proposition}

From Propositions \ref{pro21} and \ref{leftside},
we   view that  $V+\phi $ is the solution to problem (\ref{eqa}) if there  are $d_\epsilon>0$, $\sigma_\epsilon>0$ and $\xi_\epsilon\in \Omega$ such that $c_{il}(d_\epsilon,\sigma_\epsilon,\xi_\epsilon)$ are zero when $\epsilon$ small enough.

The sequel of this section is devoted to the proof of the main result.

{\it \textbf{Proof of   Theorem  \ref{kia}}}.
By Proposition \ref{leftside}, equation (\ref{aosd})  is equivalent to
\begin{eqnarray}\label{awk}
\left\{ \arraycolsep=1.0pt
   \begin{array}{lll}
      \alpha_n a_1  d_1 ^{ n-2 }\varphi(\xi)
 + a_3 \sum\limits_{i=1}^{k-1}\Big(\frac{d_{i+1}}{ d_i}\Big)^{\frac{n-2}{2}} g(\sigma_i)
  - a_4\sum\limits_{i=1}^k \frac{2}{2i-1}|\ln  d_i|
 = o(1),   \\[2mm]
  \frac{\alpha_n }{2} a_2\partial_{\xi_h}\varphi(\xi)
d_1^{ n-1 }=0,
\end{array}
\right.
\end{eqnarray}
for $h=1,\cdots,n$.
We note that $G^\epsilon\rightarrow G$ uniformly on compact set of $[0,+\infty]\times[0,+\infty]\times  \Omega$,
and   the vector functional
$G (\bar{d},\bar{\sigma},\xi)=\Big(G_0(\bar{d},\bar{\sigma},\xi),G_h(d_1,\xi)\Big)
$
is the principal part of
defined by
\begin{align*}
& G_0(\bar{d},\bar{\sigma},\xi)
 =    \alpha_n a_1  d_1 ^{ n-2 }\varphi(\xi)
 + a_3 \sum\limits_{i=1}^{k-1}\Big(\frac{d_{i+1}}{ d_i}\Big)^{\frac{n-2}{2}} g(\sigma_i)
  - a_4\sum\limits_{i=1}^k \frac{2}{2i-1}|\ln  d_i|,\nonumber\\
& G_h(d_1,\xi)
=\frac{\alpha_n }{2} a_2\partial_{\xi_h}\varphi(\xi)
d_1^{ n-2 },\quad\mbox{for}\ \  h=1,\cdots,n.
\end{align*}

Let us  set $s_1=d_1$, $s_i=\frac{d_i}{d_{i-1}}$, $i=2,\cdots,k$,
then  in the new variables
$\bar{s}=(s_1,\cdots,s_k)$,
 $G_h(d_1,\xi)$ and $ G_0(\bar{d},\bar{\sigma},\xi)$ can be rewrite as
\begin{align*}
\bar{G}_h(s_1,\xi)
= & \frac{\alpha_n }{2} a_2\partial_{\xi_h}\varphi(\xi)
s_1^{ n-2 },\quad\mbox{for}\quad h=1,\cdots,n,\nonumber\\
\bar{G}_0(\bar{s},\bar{\sigma},\xi)
 =  &   \alpha_n a_1  s_1 ^{ n-2 }\varphi(\xi)
 + a_3 \sum\limits_{i=2}^{k}s_i^{\frac{n-2}{2}} g(\sigma_i)
  - a_4\sum\limits_{i=1}^k \frac{2}{2i-1}|\ln  s_i|.
\end{align*}
We denote $\bar{G}=\Big(\bar{G}_0(\bar{s},\bar{\sigma},\xi),\bar{G}_h(s_1,\xi)\Big)
$.
Let $\xi^0\in\Omega$  be a strict minimum point  of  Robin function $\varphi$,
which is the zero point of function $\bar{G}_h$ for $h=1,\cdots,n$.
Observe that $\sigma_i=0$ is a strict minimum point of $g$.
On the other hand,
when $s_i$ is close to 0, the function $\bar{G}_0$ tends to $-\infty$,
and   $\bar{G}_0>0$  as $s_i>0$ large enough,
thus, by intermediate value theorem,  there exists a $\bar{s}_0$ such that $\bar{G}(\bar{s}_0,0, \xi^0)=0$.
Moreover,  $(\bar{s}_0,0, \xi^0)$  is an isolated zero of $\bar{G}$ whose Brouwer degree is not zero.
Therefore, if $\epsilon$ is small enough,
(\ref{awk}) has a solution $(\bar{s}_\epsilon,\bar{\sigma}_\epsilon,\xi_\epsilon)$
near $(\bar{s}_0,0,\xi_0)$.
We conclude that
the right hand side of (\ref{ccf}) is zero, i.e.,
\[
\sum_{i=1}^k\sum_{l=0}^nc_{il} \Big\langle P\psi^l_{ \mu_i,\xi_i}, P\psi^h_{ \mu_j,\xi_j}\Big\rangle=0.
\]
Moreover,  by Lemma \ref{inerpro}, we conclude that $c_{il} $  are zero.
 We finish the proof of this theorem.
\qed

\section{The  finite dimensional reduction}

In this section, we prove Proposition \ref{pro21}.
Let
$L_{ \mu,\xi}: E^{\bot}_{ \mu,\xi}\rightarrow E^{\bot}_{ \mu,\xi}$
be the linear operator defined by
\begin{equation}\label{linea}
L_{ \mu,\xi}(\phi)=\phi-\Pi^{\bot}_{ \mu,\xi}
\Big( i^*[ f_\epsilon^{'}(V)\phi]\Big),
\end{equation}
where
$V$ is defined in (\ref{solution}).
In the following,  we establish the invertibility of  $L_{ \mu,\xi}$ on $E^{\bot}_{ \mu,\xi}$.

\begin{lemma}\label{inver}
There exist  $\epsilon_0>0$ and $C>0$ such that for any $\xi\in\Omega$, $\bar{d} \in \R^k_+$, $\bar{\sigma} \in (\R_+^{n})^k$ satisfying (\ref{eta}), for   $\epsilon\in(0,\epsilon_0)$, it holds
\begin{equation}\label{lal}
\|L_{ \mu,\xi} (\phi)\|\geq C\|\phi\|,\quad \forall\phi\in E^{\bot}_{ \mu,\xi}.
\end{equation}
\end{lemma}
\begin{proof}
We argue by contradiction.
Assume there exist sequences   $ \epsilon_m\rightarrow0$,
$\xi\in\Omega$, $\bar{\sigma}_m \in (\R_+^{n})^k$  and
$\bar{d}_m =(d_{1m},\cdots,d_{km})\in \R^k_+$
with $\xi_m\rightarrow \xi\in\Omega$, $\sigma_{im}\rightarrow\sigma_i$
and $d_{im}\rightarrow d_i>0$,
$i=1,\cdots,k$,
$\phi_m$, $h_m\in \Lambda^{\bot}_{ \mu_m,\xi_m}$
such that
\begin{equation}\label{assu}
 L_{ \mu_m,\xi_m} (\phi_m) =h_m, \quad
  \|\phi_m\|=1 \quad\mbox{and}\quad  \|h_m\|\rightarrow0.
\end{equation}
From (\ref{linea}), there exists $\omega_m\in E_{ \mu_m,\xi_m}$ such that
\begin{equation}\label{linear}
 \phi_m- i^*[ f_\epsilon^{'}(V_m)\phi_m]=  h_m+\omega_m,
\end{equation}
where $V_m=V(\bar{d}_m,\bar{\sigma}_m,\xi_m)=\sum\limits_{i=1}^kPU_{ \mu_{im},\xi_{im}}$.

\textbf{Step 1.}
We prove that
\begin{equation}\label{lfimi}
\|\omega_m\|\rightarrow0.
\end{equation}
Let
$\omega_m=\sum\limits_{i=1}^k\sum\limits_{l=0}^n c^{il}_m
P\psi^l_{ \mu_m^i,\xi_m^i}$,
we multiply (\ref{linear}) by  $P\psi^h_{ \mu_m^l,\xi_m^l}$,
and integrating in $\Omega$, then
\begin{equation}\label{tonc}
\sum\limits_{i=1}^k\sum\limits_{l=0}^n c^{il}_m
\langle P\psi^l_{ \mu_m^i,\xi_m^i}, P\psi^h_{ \mu_m^j,\xi_m^j}\rangle
= \int_\Omega f_\epsilon^{'}(V_m)\phi_m P\psi^h_{ \mu_m^j,\xi_m^j}dx.
\end{equation}
From Lemma \ref{inerpro},
  we obtain
\begin{align}
&\sum\limits_{i=1}^k\sum\limits_{l=0}^n c^{il}_m
\langle P\psi^l_{ \mu_m^i,\xi_m^i}, P\psi^h_{ \mu_m^j,\xi_m^j}\rangle\nonumber\\
= & c^{jh}_m \Big(c_h(1+o(1))\Big)
+ O(1) \sum\limits_{l=0, l\neq h }^nc^{jl}_m
  +  o\bigg(\Big(\frac{\epsilon}{|\ln\epsilon|^2}\Big)^{\frac{n}{n-2}}\bigg) \sum\limits_{i=1, i\neq j}^k\sum\limits_{l=0}^n c^{il}_m.
\end{align}
On the other hand, by (\ref{gisewr1}), (\ref{subu2}), (\ref{subu3}), (\ref{sumbus}), (\ref{sumbu2}),  (\ref{fepli2}) and the orthogonality condition $ \langle \phi_m,P\psi^h_{ \mu_m^j,\xi_m^j} \rangle =0$, we have
\begin{align}\label{jiay}
 & \int_\Omega f_\epsilon^{'}(V_m)\phi_m P\psi^h_{ \mu_m^j,\xi_m^j}dx\nonumber\\
  = & \int_\Omega \Big(f_\epsilon^{'}(V_m)-f_0^{'}(V_m)\Big)\phi_m
  \Big(P\psi^h_{ \mu_m^j,\xi_m^j}-\psi^h_{ \mu_m^j,\xi_m^j}\Big)dx
   + \int_\Omega\Big( f_\epsilon^{'}(V_m)-f_0^{'}(V_m)\Big)
  \phi_m \psi^h_{ \mu_m^j,\xi_m^j}dx\nonumber\\
  & + \int_\Omega\Big(f_0^{'}(V_m)-\sum\limits_{i=1}^k(-1)^if_0^{'}(PU_{ \mu_m^i,\xi_m^i})\Big)   \phi_m P\psi^h_{ \mu_m^j,\xi_m^j}dx\nonumber\\
    & + \sum\limits_{i=1}^k(-1)^i\int_\Omega\Big(f_0^{'}(PU_{ \mu_m^i,\xi_m^i})-f_0^{'}(U_{ \mu_m^i,\xi_m^i})\Big)   \phi_m P\psi^h_{ \mu_m^j,\xi_m^j}dx\nonumber\\
  \leq  &   \Big|f_\epsilon^{'}(V_m)-f_0^{'}(V_m)\Big|_{\frac{n}{2}}|\phi_m|_{\frac{2n}{n-2}}
  \Big|P\psi^h_{ \mu_m^j,\xi_m^j}-\psi^h_{ \mu_m^j,\xi_m^j}\Big|_{\frac{2n}{n-2}}\nonumber\\
  & +  \Big|f_\epsilon^{'}(V_m)-f_0^{'}(V_m) \Big|_{\frac{n}{2}}|\phi_m|_{\frac{2n}{n-2}}
  | \psi^h_{ \mu_m^j,\xi_m^j}|_{\frac{2n}{n-2}}\nonumber\\
  & +  \Big|f_0^{'}(V_m)-\sum\limits_{i=1}^k(-1)^if_0^{'}(PU_{ \mu_m^i,\xi_m^i})\Big|_{\frac{n}{2}}
   |P\psi^h_{ \mu_m^j,\xi_m^j}|_{\frac{2n}{n-2}}
   |\phi_m|_{\frac{2n}{n-2}}\nonumber\\
  & + \Big|f_0^{'}(PU_{ \mu_m^i,\xi_m^i})-f_0^{'}(U_{ \mu_m^i,\xi_m^i})\Big|_{\frac{n}{2}}
   |P\psi^h_{ \mu_m^j,\xi_m^j}|_{\frac{2n}{n-2}}
   |\phi_m|_{\frac{2n}{n-2}}
  =O \bigg(\epsilon\ln\Big|\ln\frac{\epsilon}{|\ln\epsilon|^2}\Big|\bigg).
\end{align}
Consequently, from (\ref{tonc})-(\ref{jiay}),
we obtain (\ref{lfimi}).

\textbf{Step 2.} We prove that
\begin{equation}\label{limi}
\liminf_{m\rightarrow\infty}\int_\Omega f_\epsilon^{'}(V_m)u_m^2dx=C>0,
\end{equation}
where $u_m$ satisfies
\begin{align}\label{defqa}
\left\{ \arraycolsep=1.5pt
   \begin{array}{lll}
-\Delta u_m=f_\epsilon^{'}(V_m)u_m+ f_\epsilon^{'}(V_m)( h_m+\omega_m)\ \   &{\rm in}\ \Omega, \\[2mm]
u_m= 0 \ \  & {\rm on}\ \partial\Omega,
\end{array}
\right.
\end{align}
with
\begin{equation}\label{kiaa}
u_m=\phi_m- h_m-\omega_m,\quad \ \|u_m\|\rightarrow1.
\end{equation}

We prove that
\begin{equation}\label{lidmi}
\liminf_{m\rightarrow\infty} \|u_m\|=C>0.
\end{equation}
From (\ref{defqa}), there holds
\begin{equation}\label{ezqa}
u_m
=   i^*\Big[f_\epsilon^{'}(V_m)u_m+ f_\epsilon^{'}(V_m)( h_m+\omega_m)\Big].
\end{equation}
Moreover, by (\ref{embed}), (\ref{sumbus}), (\ref{sumbu2}) and  (\ref{fepli2}),   we get
\begin{align}\label{sia}
|u_m|_{\frac{2n}{n+2}}
\leq  & C\Big( |f_\epsilon^{'}(V_m)u_m |_{\frac{2n}{n+2}}+ |f_\epsilon^{'}(V_m)( h_m+\omega_m)| _{\frac{2n}{n+2}}\Big)\nonumber\\
\leq & C |f_\epsilon^{'}(V_m)-f_0^{'}(V_m)|_{\frac{n}{2}} |u_m|_{\frac{2n}{n-2}}
+ C\Big| f_0^{'}(V_m)-\sum\limits_{i=1}^k(-1)^if_0^{'}(PU_{ \mu_m^i,\xi_m^i}) \Big|_{\frac{n}{2}}
|u_m|_{\frac{2n}{n-2}}\nonumber\\
 & + C |f_\epsilon^{'}(V_m)-f_0^{'}(V_m)|_{\frac{n}{2}} | h_m+\omega_m |_{\frac{2n}{n-2}}\nonumber\\
 &+ C\Big| f_0^{'}(V_m)-\sum\limits_{i=1}^k(-1)^if_0^{'}(PU_{ \mu_m^i,\xi_m^i}) \Big|_{\frac{n}{2}}
| h_m+\omega_m |_{\frac{2n}{n-2}}\nonumber\\
 &+ C\Big|\sum\limits_{i=1}^k(-1)^i\Big(f_0^{'}(PU_{ \mu_m^i,\xi_m^i})-f_0^{'}(U_{ \mu_m^i,\xi_m^i})\Big)\Big|_{\frac{n}{2}}
| h_m+\omega_m |_{\frac{2n}{n-2}}
\leq   C \|u_m\|+o(1).
\end{align}
It follows that $|u_m|_{\frac{2n}{n+2}}\rightarrow0 $ provided that $\|u_m\|\rightarrow0 $,
this contradicts with (\ref{kiaa}). Therefore, (\ref{lidmi}) holds.

We multiply (\ref{ezqa}) by $u_m$, that is
\begin{equation}\label{aim}
\|u_m\|^2
= \int_\Omega f_\epsilon^{'}(V_m)u_m^2dx
+ \int_\Omega f_\epsilon^{'}(V_m)
( h_m+\omega_m)u_mdx.
\end{equation}
By (\ref{assu}) and (\ref{lfimi}), one has
\begin{align}\label{laf}
   \int_\Omega f_\epsilon^{'}(V_m)
( h_m+\omega_m)u_m dx
  \leq &  |f_\epsilon^{'}(V_m)|_{\frac{n}{2}}
  | h_m+\omega_m)|_{\frac{2n}{n-2}}| u_m|_{\frac{2n}{n-2}}\nonumber\\
  \leq & \| h_m+\omega_m\| \| u_m\| =o(1).
\end{align}
Therefore (\ref{limi}) follows by (\ref{kiaa}), (\ref{lidmi}), (\ref{aim})  and (\ref{laf}).

\textbf{Step 3.}
Let us prove that a contradiction arises, by showing that
\begin{equation}\label{lismi}
 \int_\Omega f_\epsilon^{'}(V_m)u_m^2dx=o(1).
\end{equation}
In order to deal with this conclusion,
we  decompose   $B(\xi, \rho)$ into the union of non-overlapping  annuli, that is  $B(\xi, \rho)=\bigcup\limits_{i=1}^k\mathcal{A}_i$, where
\begin{equation}\label{annulus}
\mathcal{A}_i=B(\xi,\sqrt{\mu_i\mu_{i-1}})\setminus B(\xi,\sqrt{\mu_i\mu_{i+1}}), \quad i=1,\cdots,k,
\end{equation}
with $\mu_0=\frac{\rho^2}{\mu_1}$ and  $\mu_{k+1}=0$.
We set  a smooth cut-off function $\chi_m^i$ as
\begin{align}\label{eqda}
\chi_m^i(x) =
\left\{ \arraycolsep=1.5pt
   \begin{array}{lll}
1 \ \   & {\rm if}\ \sqrt{\mu_m^i\mu_m^{i+1}}\leq |x-\xi_m|\leq \sqrt{\mu_m^i\mu_m^{i-1}}, \\[2mm]
0 \ \   & {\rm if}\ |x-\xi_m|\leq\frac{\sqrt{\mu_m^i\mu_m^{i+1}}}{2} \ {\rm or } \ |x-\xi_m|\geq 2\sqrt{\mu_m^i\mu_m^{i-1}},
\end{array}
\right.
\end{align}
and
\begin{equation}\label{chi}
|\nabla\chi_m^i(x)|\leq \frac{2}{\sqrt{\mu_m^i\mu_m^{i-1}}}
\quad {\rm and } \quad
|\nabla^2\chi_m^i(x)|\leq \frac{4}{ \mu_m^i\mu_m^{i-1}},\quad\mbox{for \ any}\ i=1,\cdots,k.
\end{equation}
We define
\begin{equation}\label{sarj}
\tilde{u}_m^i(y)
=(\mu_m^i)^{\frac{n-2}{2}}u_m(\mu_m^iy+\xi_m)\chi_m^i(\mu_m^iy+\xi_m).
\end{equation}
First, the following results will be showed  in  Step 4,
\begin{equation}\label{conver}
\tilde{u}_m^i\rightarrow 0 \quad\mbox{weakly\ in}\ D^{1,2}(\R^n),\quad
\tilde{u}_m^i\rightarrow 0 \quad\mbox{strongly\ in}\ L_{loc}^q(\R^n)\ \mbox{for\ any}\ q\in[2,2^*).
\end{equation}
Let us prove (\ref{lismi}).
There holds
\[
\int_\Omega f_\epsilon^{'}(V_m)u_m^2dx
= \int_{\Omega\setminus B(\xi,\rho)} f_\epsilon^{'}(V_m)u_m^2dx
+\sum_{i=1}^k\int_{\mathcal{A}_i} f_\epsilon^{'}(V_m)u_m^2dx,
\]
where
\[
 \int_{\Omega\setminus B(\xi,\rho)} f_\epsilon^{'}(V_m)u_m^2dx
 \leq C \sum_{i=1}^k(\mu_i^n)^2 \int_{\Omega\setminus B(\xi,\rho)} u_m^2dx
 = o(1).
\]
Since $(\frac{1}{1+|x|^2})^2\in L^{\frac{n}{2}}(\R^n)$ and (\ref{conver})  hold, we conclude that $\int_{(\mathcal{A}_m^i-\xi_m)/\mu_m^i}
 (\frac{1}{1+|y-\sigma_{in}|^2})^{\frac{n-2}{2}(p-1)}(\tilde{u}_m^i)^2dy\rightarrow 0$.
On the other hand, we set $x-\xi_m=\mu_m^iy$ and by a fact that,
let $h\in L^1_{rad}(\R^n)$,
performing the proper change of variable: for any $i\neq l$,
\begin{align}\label{eqsa}
\int_{\frac{\mathcal{A}_m^i-\xi_m}{\mu_m^i}}h(|x|)dx
 =
\left\{ \arraycolsep=1.5pt
   \begin{array}{lll}
O\Big((\frac{\mu_l}{\mu_i})^{\frac{n}{2}}\Big) \ \   & {\rm if}\ i\leq l-1<l, \\[2mm]
O\Big((\frac{\mu_i}{\mu_l})^{\frac{n}{2}}\Big) \ \   & {\rm if}\ i\geq l-1>l.
\end{array}
\right.
\end{align}
By (\ref{eqsa}) and the choice of $\mu_i$  in (\ref{delta}),
 we deduce that
\begin{equation}\label{radia}
\mbox{if}\ h\in L^1_{rad}(\R^n),\, i\neq l,\quad
\int_{\frac{\mathcal{A}_m^l-\xi_m}{\mu_m^i}}h(|x|)dx
=O\bigg(\Big(\frac{\epsilon}{|\ln\epsilon|^2}\Big)^{\frac{n}{n-2}}\bigg).
\end{equation}
Then, there holds
\[
\int_{\frac{\mathcal{A}_m^i-\xi_m}{\mu_m^i}}U^{(p-1)\frac{n}{2}}_{ \mu_{im},\xi_{im}} dx
=O\bigg(\int_{\frac{\sqrt{\mu_m^i\mu_m^{i+1}}}{\mu_m^i} \leq |y|\leq \frac{\sqrt{\mu_m^i\mu_m^{i-1}}}{\mu_m^i} }
 \frac{1}{( 1+|y-\sigma_{im}|^2)^n}dy
 \bigg).
\]
Consquently, from (\ref{fepli2}), we have
\begin{align*}
& \int_{\mathcal{A}^i_m} f_\epsilon^{'}(V_m)u_m^2dx\\
 = & \int_{\mathcal{A}^i_m}\Big( f_\epsilon^{'}(V_m)-f_0^{'}(V_m)\Big)u_m^2dx
 +\int_{\mathcal{A}^i_m} f_0^{'}(V_m)u_m^2dx\\
 \leq & C  |f_\epsilon^{'}(V_m)-f_0^{'}(V_m)|_{\frac{n}{2}}
 |u_m|^2_{\frac{n}{n-2}}
  + C\sum_{i=1}^k \int_{\mathcal{A}^i_m}U^{p-1}_{ \mu_{im},\xi_{im}}u_m^2dx\\
 \leq & C  \epsilon\ln\Big|\ln\frac{\epsilon}{|\ln\epsilon|^2}\Big|
 + C (\mu_m^i)^{2-\frac{n-2}{2}(p-1)}\int_{\frac{\mathcal{A}_m^i-\xi_m}{\mu_m^i}}
 \Big(\frac{1}{1+|y-\sigma_{im}|^2}\Big)^{\frac{n-2}{2}(p-1)}(\tilde{u}_m^i)^2dy\\
 & + C\sum_{j=1,j\neq i}^k \Big(\int_{\mathcal{A}_m^i}U^{(p-1)\frac{n}{2}}_{ \mu_{jm},\xi_{jm}} dx \Big)^{\frac{2}{n}}|u_m|_{\frac{2n}{n-2}}^2
 =o(1).
\end{align*}

\textbf{Step 4.} We prove (\ref{conver}).

From the definition of $\tilde{u}_m^i$, $i=1,\cdots,k$, in (\ref{sarj}), when $x-\xi_m=\mu_m^iy$, we get
\begin{equation}\label{nus}
\nabla\tilde{u}_m^i(y)=(\mu_m^i)^{\frac{n}{2}}
\bigg[\Big(\nabla u_m(x)\Big)\chi_m^i(x)+u_m(x)\Big(\nabla\chi_m^i(x)\Big)\bigg], 
\end{equation}
and
\begin{equation}\label{nuss}
\Delta\tilde{u}_m^i(y)=(\mu_m^i)^{\frac{n+2}{2}}
\bigg[\Big(\Delta u_m(x)\Big)\chi_m^i(x)+2 \nabla u_m(x)\nabla\chi_m^i(x)+   u_m(x)\Big(\Delta\chi_m^i(x)\Big)\bigg]. 
\end{equation}
Then, from (\ref{eqda}), (\ref{chi}) and (\ref{nus}),
it holds that $\|\tilde{u}_m^i\|_{D^{1,2}(\R^n)}\leq C$.
It follows that, up to a subsequence,
\[
\tilde{u}_m^i\rightarrow \tilde{u}^i \quad\mbox{weakly\ in}\ D^{1,2}(\R^n),\quad
\tilde{u}_m^i\rightarrow \tilde{u}^i \quad\mbox{strongly\ in}\ L_{loc}^q(\R^n)\ \mbox{for\ any}\ q\in[2,2^*).
\]
Next, we   show that $\tilde{u}^i$ is the solution of the following problem
\begin{equation}\label{equa}
-\Delta \tilde{u}^i=f_0^{'}(U_{1,\sigma_i})\tilde{u}^i\quad\mbox{in}\ \R^n,
\end{equation}
and satisfies the orthogonality conditions
\begin{equation}\label{eqsua}
\int_{ \R^n}\nabla \psi_{1,\sigma_i}^h \nabla \tilde{u}^idx=0,\quad h=0,1,\cdots,n.
\end{equation}
It follows that $\tilde{u}^i=0$.
This
is a contradiction by the result of \cite{bEG}, which concludes the proof.

\textbf{Step 5.} We prove (\ref{equa}) and (\ref{eqsua}).

(1) Let us prove (\ref{equa}).
By (\ref{nus}) and (\ref{nuss}), if $x-\xi_m=\mu_m^iy$, $y\in \Omega_m^i=\frac{ \Omega-\xi_m}{\mu_m^i}$, we have

\begin{align*}
\left\{ \arraycolsep=1.5pt
   \begin{array}{lll}
-\Delta\tilde{u}_m^i(y)
= & (\mu_m^i)^2f_\epsilon^{'}\Big(V_m(x)\Big)\tilde{u}_m^i(y)
+ (\mu_m^i)^{\frac{n+2}{2}}f_\epsilon^{'}\Big(V_m(x)\Big)\Big( h_m(x)+\omega_m(x)\Big)\chi_m^i(x)\\
& + 2  (\mu_m^i)^{\frac{n+2}{2}}\nabla u_m(x)\nabla\chi_m^i(x)
 + 2 (\mu_m^i)^{\frac{n+2}{2}}u_m(x)\Delta\chi_m^i(x), \\[2mm]
\tilde{u}_m^i= 0 \ &  {\rm on}\ \partial\Omega_m^i.
\end{array}
\right.
\end{align*}
Therefore, if $\varpi\in C_0^\infty(\R^n)$, one has
\begin{align}\label{aisd}
 & \int_{\R^n} \nabla\tilde{u}_m^i(y) \nabla\varpi(y)dy  \nonumber\\
  = & \int_{\R^n} (\mu_m^i)^2f_\epsilon^{'}\Big(V_m(\mu_m^iy+\xi_m)\Big)\tilde{u}_m^i(y)\varpi (y)dy
    + \int_{\R^n}(\mu_m^i)^{\frac{n+2}{2}}\nonumber\\
  & \times f_\epsilon^{'}\Big(V_m(\mu_m^iy+\xi_m)\Big)\Big( h_m(\mu_m^iy+\xi_m)+\omega_m(\mu_m^iy+\xi_m)\Big)
  \chi_m^i(\mu_m^iy+\xi_m)\varpi (y)dy\nonumber\\
  & + 2  (\mu_m^i)^{\frac{n+2}{2}} \int_{\R^n}\Big( \nabla u_m(\mu_m^iy+\xi_m)\nabla\chi_m^i(y)
 +  u_m(\mu_m^iy+\xi_m)\Delta\chi_m^i(\mu_m^iy+\xi_m)\Big)\varpi (y)dy.
\end{align}
From (\ref{eqda}) and (\ref{chi}), we deduce that the second and the third term tends to $0$.
For the first term,  if
$
 \frac{\sqrt{\mu_m^i\mu_m^{i+1}}}{2} \leq |\mu_m^iy|\leq 2\sqrt{\mu_m^i\mu_m^{i-1}},
$
there holds
\begin{align}\label{eqsdfsa}
  f_\epsilon^{'}\Big(V_m(\mu_m^iy+\xi_m)\Big)
  = &  f_\epsilon^{'}\Big( PU_{ \mu_{in},\xi_{in}}(\mu_m^iy+\xi_m)+
  \sum\limits_{i=1, j\neq i}^kPU_{ \mu_{jm},\xi_{jm}}(\mu_m^iy+\xi_m)\Big) \nonumber\\
  = &  f_\epsilon^{'}\Big((\mu_m^{i})^{-\frac{n-2}{2}} U (y-\sigma_{in})+
   U_{ \mu_{jm},\xi_{jm}}(\mu_m^iy+\xi_m)+o(1)\Big),
\end{align}
where
\begin{align}\label{eqdfsa}
U_{ \mu_{jm},\xi_{jm}}(\mu_m^iy+\xi_m)
 =
\left\{ \arraycolsep=1.5pt
   \begin{array}{lll}
O\Big((\mu_m^j)^{-\frac{n-2}{2}}\Big) \ \   & {\rm if}\ i>j, \\[2mm]
O\Big(\frac{(\mu_m^j)^{\frac{n-2}{2}}}{(\mu_m^i)^{n-2}}
\Big|y-\frac{\mu_m^j}{\mu_m^i}\sigma_m^j\Big|^{-(n-2)} \Big) \ \   & {\rm if}\ i<j.
\end{array}
\right.
\end{align}
Moreover, by  (\ref{eqsdfsa}), (\ref{eqdfsa}) and Lebesgue's dominated convergence theorem, it holds that
\[
\int_{\R^n} (\mu_m^i)^2f_\epsilon^{'}\Big(V_m(\mu_m^iy+\xi_m)\Big)\tilde{u}_m^i(y)\varpi (y)dy
\rightarrow
\int_{\R^n} f_0^{'}\Big(U(y-\sigma_i)\Big)\tilde{u}^i(y)\varpi (y)dy.
\]
Then,  (\ref{equa}) follows by passing to the limit in (\ref{aisd}).

(2) Let us prove (\ref{eqsua}).
We set $x-\xi_m=\mu_m^iy$, then
\begin{align}\label{eqssua}
\int_{ \R^n}\nabla \psi_{1,\sigma_m^i}^h (y)\nabla \tilde{u}_m^i(y)dy
= & \int_{ \R^n} f_0^{'}\Big(U_{1,\sigma_m^i}(y)\Big) \psi_{1,\sigma_m^i}^h (y)  \tilde{u}_m^i(y)dy\nonumber\\
= &\mu_m^i \int_{\frac{\sqrt{\mu_m^i\mu_m^{i+1}}}{2} \leq |x-\xi|\leq 2\sqrt{\mu_m^i\mu_m^{i-1}}}
f_0^{'}\Big(U_{ \mu^i_m,\xi^i_m}(x)\Big) \psi_{ \mu^i_m,\xi^i_m}^h (x)  u_m(x)\chi_m^i(x)dx\nonumber\\
= &\mu_m^i\bigg(
\int_{\mathcal{A}_m^i}
f_0^{'}\Big(U_{ \mu^i_m,\xi^i_m}(x)\Big) \psi_{ \mu^i_m,\xi^i_m}^h (x)  u_m(x) dx+o(1)\bigg).
\end{align}
Now, we  show that
\begin{equation}\label{sua}
\mu_m^i
\int_{\mathcal{A}_m^i}
f_0^{'}\Big(U_{ \mu^i_m,\xi^i_m}(x)\Big) \psi_{ \mu^i_m,\xi^i_m}^h (x)  u_m(x) dx
=o(1).
\end{equation}
Therefore,  (\ref{eqsua})  follows by (\ref{eqssua}) and (\ref{sua}),
taking into account that $\sigma_m^i\rightarrow \sigma_i$.
From (\ref{lfimi}) and (\ref{kiaa}), one has
\begin{equation}\label{lsua}
\mu_m^i
\int_{\Omega}
\nabla P\psi_{ \mu^i_m,\xi^i_m}^h (x) \nabla u_m(x) dx
=o(1).
\end{equation}
On the other hand,
\begin{align*}
 \int_{\Omega}
\nabla P\psi_{ \mu^i_m,\xi^i_m}^h (x) \nabla u_m(x) dx
= &  \int_{\Omega}f_0^{'}\Big(U_{ \mu^i_m,\xi^i_m}(x)\Big) \psi_{ \mu^i_m,\xi^i_m}^h (x)  u_m(x) dx  \\
= &   \int_{\Omega\setminus B(\xi_m,\rho)}  \cdots dx
+\sum_{l=1,\ l\neq i}^k\int_{\mathcal{A}^l_m}\cdots dx
+\int_{\mathcal{A}^i_m}\cdots dx\\
= & \int_{\mathcal{A}^i_m}\cdots dx  +o\Big(\frac{1}{\mu_m^i}\Big),
\end{align*}
where
\begin{align}\label{deca}
&\int_{\Omega\setminus B(\xi_m,\rho)} \bigg|f_0^{'}\Big(U_{ \mu^i_m,\xi^i_m}(x)\Big) \psi_{ \mu^i_m,\xi^i_m}^h (x)  u_m(x) \bigg|dx\nonumber\\
\leq & C|\psi_{ \mu^i_m,\xi^i_m}^h|_{\frac{2n}{n-2}}|u_m |_{\frac{2n}{n-2}}
\Big(\int_{\Omega\setminus B(\xi_m,\rho)} U_{ \mu^i_m,\xi^i_m}^{\frac{2n}{n-2}} dx\Big)^{\frac{2}{n}}
=   O(\mu^i_m).
\end{align}
If $l\neq i$, by (\ref{radia}), there is a fact that
$\int_{\mathcal{A}^i_m}U_{ \mu^i_m,\xi^i_m}^{\frac{2n}{n-2}} dx=O\bigg(\Big(\frac{\epsilon}{|\ln\epsilon|^2}\Big)^{\frac{n}{n-2}}\bigg)$,
then
\begin{align*}
&\int_{\mathcal{A}^l_m} \bigg|f_0^{'}\Big(U_{ \mu^i_m,\xi^i_m}(x)\Big) \psi_{ \mu^i_m,\xi^i_m}^h (x)  u_m(x) \bigg|dx\\
\leq & C|\psi_{ \mu^i_m,\xi^i_m}^h |_{\frac{2n}{n-2}}|u_m |_{\frac{2n}{n-2}}
\Big(\int_{\mathcal{A}^l_m} U_{ \mu^i_m,\xi^i_m}^{\frac{2n}{n-2}} dx\Big)^{\frac{2}{n}}
=  O\bigg(\Big(\frac{\epsilon}{|\ln\epsilon|^2}\Big)^{\frac{n}{n-2}}\bigg).
\end{align*}
We finish the proof of this lemma.
\end{proof}

Now, by means of the previous result,  we  show the following proof.

{\it \textbf{Proof of Proposition \ref{pro21}}}:
First of all, we point out that $\phi$
solves equation  (\ref{ng2}) if and only if $\phi$ is   a fixed point of the map  $T_{ \mu,\xi}: E^{\bot}_{ \mu,\xi}\rightarrow E^{\bot}_{ \mu,\xi}$ defined by
\begin{align*}
  T_{ \mu,\xi}(\phi)
  = & L_{ \mu,\xi}^{-1} \Pi^{\bot}_{ \mu,\xi}i^*\bigg[
     \Big( f_\epsilon(V+\phi  )-f_\epsilon(V) -f^{'}_\epsilon(V) \phi  \Big)\\
 & + \Big(f_\epsilon^{'}(V  )-\sum_{i=1}^k(-1)^if_0^{'}(PU_{ \mu_i,\xi_i}) \Big)\phi
   + \Big(\sum_{i=1}^k(-1)^if_0^{'}(PU_{ \mu_i,\xi_i}) -\sum_{i=1}^k(-1)^if_0^{'}( U_{ \mu_i,\xi_i}) \Big)\phi\\
 & + \Big(f_\epsilon (V) -\sum_{i=1}^k(-1)^if_0(PU_{ \mu_i,\xi_i}) \Big)
   + \Big(\sum_{i=1}^k(-1)^if_0(PU_{ \mu_i,\xi_i}) -\sum_{i=1}^k(-1)^if_0(U_{ \mu_i,\xi_i}) \Big)\bigg].
\end{align*}
From  Lemma \ref{inver} and  Sobolev  inequality, we have
\begin{align*}
  \|T_{ \mu,\xi}(\phi)\|
  \leq &  C
      \Big|f_\epsilon(V+\phi  )-f_\epsilon(V) -f^{'}_\epsilon(V) \phi  \Big|_{\frac{2n}{n+2}}\\
 & + C \Big|\Big(f_\epsilon^{'}(V  )-\sum_{i=1}^k(-1)^if_0^{'}(PU_{ \mu_i,\xi_i})\Big)\phi \Big|_{\frac{2n}{n+2}}
   + C\Big|\Big(f^{'}_0(V) -\sum_{i=1}^k(-1)^if_0^{'}( U_{ \mu_i,\xi_i}) \Big)\phi\Big|_{\frac{2n}{n+2}}\\
 & + C\Big|f_\epsilon (V) -\sum_{i=1}^k(-1)^if_0(PU_{ \mu_i,\xi_i}) \Big|_{\frac{2n}{n+2}}
   + C\Big|\sum_{i=1}^k(-1)^i\Big(f_0(PU_{ \mu_i,\xi_i}) - f_0(U_{ \mu_i,\xi_i})\Big)  \Big|_{\frac{2n}{n+2}}\\
 = &   H_1+\cdots+H_5.
\end{align*}

\emph{Estimate of  $ H_1$}:
From  the mean value theorem, we choose $t=t(x)\in (0,1)$, then
\begin{align}\label{dkoa}
  H_1 =  \Big|f_\epsilon(V+\phi  )-f_\epsilon(V) -f^{'}_\epsilon(V) \phi )\Big|_{\frac{2n}{n+2}}
   =    \Big|f_\epsilon(V+t\phi  )-f^{'}_\epsilon(V) \phi )\Big|_{\frac{2n}{n+2}}.
\end{align}
When $n<6$,  Lemma \ref{zsj}   follows that
\begin{align*}
  H_1 \leq & C\Big(  |\phi|^p_{\frac{2n}{n+2}}+|U_{ \mu_i,\xi_i}^{p-2}\phi^2|_{\frac{2n}{n+2}} \Big)
   \leq   C \Big( |\phi|^{p-2}_{\frac{2n}{n-2}}+|U_{ \mu_i,\xi_i}|_{p-2}^{p-2} \Big)|\phi|^2_{\frac{2n}{n-2}}
   =   C \Big( |\phi|^{p-2}+1\Big)\|\phi\|^2.
\end{align*}
When $n=6$,  by Sobolev   inequality, one has
\begin{align*}
  H_1 \leq   C\Big( \Big||\phi|^p\Big|_{\frac{2n}{n+2}}+|\phi^2|_{\frac{2n}{n+2}} \Big)
   =   C \bigg( |\phi|^p_{\frac{2n}{n-2}}+\Big(\int\limits_\Omega|\phi|^{\frac{2n}{n-2}} dx\Big)^{\frac{ n+2}{2n}}\bigg)
   =   2C  |\phi|_{p+1}^p \leq 2C  \|\phi\|^2.
\end{align*}
When $n>6$, there holds
\begin{align*}
  H_1 \leq & C\Big( |\phi|^p_{\frac{2n}{n+2}}+\epsilon|U_{ \mu_i,\xi_i}^{p-1}\phi |_{\frac{2n}{n+2}} \Big)
   =  C \bigg( |\phi|^p_{\frac{2n}{n-2}}+\Big(\int_\Omega(U_{ \mu_i,\xi_i}^{p-1}|\phi|)^{\frac{2n}{n+2}} dx\Big)^{\frac{ n+2}{2n}}\bigg)\nonumber\\
   \leq  & C \Big( |\phi|^p_{\frac{2n}{n-2}}+\epsilon|U_{ \mu_i,\xi_i}|^{p-1}_{\frac{2n}{n-2}} |\phi|_{\frac{2n}{n-2}} \Big)
   =    C \Big( |\phi|^{p-1}_{\frac{2n}{n-2}}+\epsilon|U_{ \mu_i,\xi_i}|^{p-1}_{\frac{2n}{n-2}}  \Big)|\phi|_{\frac{2n}{n-2}}\\
   \leq & C ( \|\phi\|^{p-1} +\epsilon  )\|\phi\|.
\end{align*}
Sum up  these estimates, we have
\begin{eqnarray}\label{dt}
H_1 \leq
\left\{ \arraycolsep=1.5pt
  \begin{array}{lll}
C  ( |\phi|^{p-2}+1 )\|\phi\|^2 \ \   &{\rm if}\  3\leq n\leq5, \\[2mm]
C \|\phi\|^2 \ \   &{\rm if}\   n=6, \\[2mm]
C  ( \|\phi\|^{p-1} +\epsilon  )\|\phi\|\ \   &{\rm if}\  n\geq 7. \\[2mm]
\end{array}
\right.
\end{eqnarray}

\emph{Estimate of  $ H_2$}:
From H\"{o}lder's  inequality and (\ref{fepli1}), we get
\begin{align*}
 H_2 = &
  \Big|\Big(f_\epsilon^{'}(V)-\sum_{i=1}^k(-1)^if_0^{'}(PU_{ \mu_i,\xi_i})\Big)\phi\Big|_{\frac{2n}{n+2}} \\
  \leq &  \Big|f_\epsilon^{'}(V)-\sum_{i=1}^k(-1)^if_0^{'}(PU_{ \mu_i,\xi_i})\Big|_{\frac{ n}{ 2}}|\phi |_{\frac{2n}{n-2}}
  \leq  C \epsilon\ln\Big|\ln\frac{\epsilon}{|\ln\epsilon|^2}\Big| \|\phi\|.
\end{align*}

\emph{Estimate of  $ H_3$}:
By H\"{o}lder's  inequality, (\ref{sumbus}) and (\ref{sumbu2}),  there holds
\begin{eqnarray*}
 H_3  & = &
  \Big|\Big(f^{'}_0(V) -\sum_{i=1}^k(-1)^if_0^{'}( U_{ \mu_i,\xi_i}) \Big)\phi\Big|_{\frac{2n}{n+2}}\\
  & = &
  \Big|\Big(f^{'}_0(V) -\sum_{i=1}^k(-1)^if_0^{'}( PU_{ \mu_i,\xi_i}) \Big)\phi\Big|_{\frac{2n}{n+2}}
  +\Big|\sum_{i=1}^k(-1)^i\Big(f^{'}_0(PU_{ \mu_i,\xi_i}) -f_0^{'}( U_{ \mu_i,\xi_i}) \Big)\phi\Big|_{\frac{2n}{n+2}}\\
  &\leq &  \Big|f_0^{'}(V )-\sum_{i=1}^k(-1)^if_0^{'}(U_{ \mu_i,\xi_i})\Big|_{\frac{ n}{ 2}}|\phi |_{\frac{2n}{n-2}}
  +  k\Big|f_0^{'}(PU_{ \mu_i,\xi_i} )-f_0^{'}(U_{ \mu_i,\xi_i})\Big|_{\frac{ n}{ 2}}|\phi |_{\frac{2n}{n-2}}\\
&   \leq &
  \left\{ \arraycolsep=1.5pt
   \begin{array}{lll}
  O\Big(\frac{\epsilon}{|\ln\epsilon|^2}\|\phi \|\Big)\ \   &{\rm if }\   3\leq n\leq 5 , \\[2mm]
  \bigg(\frac{\epsilon}{|\ln\epsilon|^2}
 \Big|\ln\frac{\epsilon}{|\ln\epsilon|^2}\Big|\|\phi \|\bigg)\ \   &{\rm if }\  n=6, \\[2mm]
 O\bigg(\Big(\frac{\epsilon}{|\ln\epsilon|^2}\Big)^{\frac{-n+8}{n-2} }\|\phi \|\bigg)\ \  & {\rm if}\   n\geq 7.
\end{array}
\right.
\end{eqnarray*}

\emph{Estimate of  $H_4$  and $H_5$}:
From (\ref{fepli1}) and (\ref{sumbu3}), one has
\[
|H_4|_{\frac{2n}{n+2}}+|H_5|_{\frac{2n}{n+2}}= O(R_\epsilon),
\]
and  $R_\epsilon$ satisfies
\begin{eqnarray*}
R_\epsilon
=
\left\{ \arraycolsep=1.5pt
   \begin{array}{lll}
  O\bigg( \frac{\epsilon}{|\ln\epsilon|^2} \Big|\ln  \frac{\epsilon}{|\ln\epsilon|^2}\Big| \ln\Big|\ln \frac{\epsilon}{|\ln\epsilon|^2}\Big| \bigg)\ \   &{\rm if }\   3\leq n\leq 6 , \\[2mm]
 O\bigg(\Big(\frac{\epsilon}{|\ln\epsilon|^2}\Big)^{\frac{n+2}{2(n-2)}}\bigg)\ \  & {\rm if}\   n\geq 7.
\end{array}
\right.
\end{eqnarray*}

From   $H_1$-$H_5$,
there is a constant $C^*>0$ and $ \mu_0>0$ such that  for each $ \mu\in(0, \mu_0)$, we obtain
\[
\|T_{ \mu,\xi}(\phi)\|\leq C^*R_\epsilon\quad \mbox{for every }  \phi\in \tilde{B }=\{\phi \in E^{\bot}_{ \mu,\xi}:
\|\phi\|\leq C^*R_\epsilon\}.
\]

Next, we prove that  $T_{ \mu,\xi}$ is a contraction map.
If $\phi_1$, $\phi_2\in \tilde{B }$, then
\begin{align*}
&  \|T_{ \mu,\xi}(\phi_2)-T_{ \mu,\xi}(\phi_1) \|\\
  \leq &  C
     \Big|f_\epsilon(V+\phi_2  )-f_\epsilon(V+\phi_1) -f^{'}_\epsilon(V) (\phi_2-\phi_1)  \Big|_{\frac{2n}{n+2}}\\
 & + C\Big|\Big(f_\epsilon^{'}(V)-\sum_{i=1}^k(-1)^if_0^{'}(PU_{ \mu_i,\xi_i})\Big)(\phi_2-\phi_1) \Big|_{\frac{2n}{n+2}}\\
 &  + C\Big|\Big(\sum_{i=1}^k(-1)^if_0^{'}(PU_{ \mu_i,\xi_i}) -\sum_{i=1}^k(-1)^if_0^{'}( U_{ \mu_i,\xi_i})\Big) (\phi_2-\phi_1) \Big|_{\frac{2n}{n+2}}
   =   K_1+K_2+K_3.
\end{align*}

\emph{Estimate of $  K_1$}:
Similar to the computations of $H_1$-$H_3$.
By  mean value theorem,  we choose $\varrho=\varrho(x)\in (0,1)$ and $\phi_\varrho=(1-\varrho)\phi_1+\varrho\phi_2$, then
\begin{align*}
  K_1
  = & \Big|f_\epsilon(V+\phi_2  )-f_\epsilon(V+\phi_1) -f^{'}_\epsilon(V) (\phi_2-\phi_1 )  \Big|_{\frac{2n}{n+2}}\\
  = & \Big|\Big(f_\epsilon^{'} (V+\phi_\varrho )- f^{'}_\epsilon(V) \Big)(\phi_2-\phi_1)  \Big|_{\frac{2n}{n+2}}.
\end{align*}
When $n<6$, by Lemma \ref{zsj}   and H\"{o}lder's  inequality, we get
\begin{align*}
 K_1\leq &  C\Big(\Big| |\phi_\varrho|^{p-1}(\phi_2-\phi_1) \Big |_{\frac{2n}{n+2}}
  +\Big|\Big( \sum_{i=1}^k U_{ \mu_i,\xi_i}\Big)^{p-2} \phi_\varrho (\phi_2-\phi_1) \Big |_{\frac{2n}{n+2}}\Big)\\
  \leq &  C\bigg(  |\phi_\varrho|^{p-1}_{\frac{2n}{n-2}} | \phi_2-\phi_1|_{\frac{2n}{n-2}}
  +\bigg( \int_\Omega \Big[ \Big(\sum_{i=1}^k U_{ \mu_i,\xi_i}\Big)^{p-2} \phi_\varrho (\phi_2-\phi_1)\Big]^{\frac{2n}{n+2}}dx \bigg)^{\frac{n+2}{2n}}\bigg)\\
  \leq &   C\Big(|\phi_\varrho|^{p-1}_{\frac{2n}{n-2}}
  +\sum_{i=1}^k | U_{ \mu_i,\xi_i}|_{\frac{2n}{n-2}}^{p-2}  |\phi_\varrho|_{\frac{2n}{n-2}}\Big) |\phi_2-\phi_1|_{\frac{2n}{n-2}}.
\end{align*}
When $n=6$, we have
\[
K_1\leq   C|\phi_\varrho|^{p-1}_{\frac{2n}{n-2}} |\phi_2-\phi_1|_{\frac{2n}{n-2}}.
\]
When $n>6$, there holds
\begin{align*}
 K_1\leq &  C\Big(\Big| |\phi_\varrho|^{p-1}(\phi_2-\phi_1) \Big |_{\frac{2n}{n+2}}
  +\epsilon\Big|\Big( \sum_{i=1}^k U_{ \mu_i,\xi_i}\Big)^{p-1} \phi_\varrho(\phi_2-\phi_1) \Big |_{\frac{2n}{n+2}}\Big)\\
  \leq &  C\bigg[ |\phi_\varrho|^{p-1}_{\frac{2n}{n-2}} | \phi_2-\phi_1|_{\frac{2n}{n-2}}
  +\bigg( \epsilon\int_\Omega \Big[ \Big(\sum_{i=1}^k U_{ \mu_i,\xi_i}\Big)^{p-1} \phi_\varrho (\phi_2-\phi_1)\Big]^{\frac{2n}{n+2}}dx \bigg)^{\frac{n+2}{2n}}\bigg]\\
  \leq &   C\Big(|\phi_\varrho|^{p-1}_{\frac{2n}{n-2}} + \epsilon\Big) |\phi_2-\phi_1|_{\frac{2n}{n-2}}.
\end{align*}
Hence, by Sobolev  inequality, we obtain
\[
 K_1\leq    C\Big(|\phi_\varrho|^{p-1} +\max\{\|\phi_\varrho\|,\epsilon\}\Big) \|\phi_2-\phi_1\|.
\]

\emph{Estimate of $  K_2$}:
Similar to the proof of  $H_2$ and $H_3$, from (\ref{fepli2}) and (\ref{sumbu2}), there holds
\begin{align*}
 K_2 = &  \Big|\Big(f_\epsilon^{'}(V)-\sum_{i=1}^k(-1)^if_0^{'}(PU_{ \mu_i,\xi_i})\Big)(\phi_1-\phi_2)\Big|_{\frac{2n}{n+2}}\\
 & \leq   \Big|f^{'}_\epsilon(V) -f^{'}_0(V)\Big|_{\frac{ n}{ 2}}
  |\phi_2-\phi_1 |_{\frac{2n}{n-2}}+   \Big|f^{'}_0(V) -\sum_{i=1}^k(-1)^if_0^{'}( PU_{ \mu_i,\xi_i})\Big|_{\frac{ n}{ 2}}
  |\phi_2-\phi_1 |_{\frac{2n}{n-2}}\\
  & \leq C \Big ( \frac{\epsilon}{|\ln\epsilon|^2}\Big)^{\frac{-n+8}{n-2} }\ln\Big|\ln\frac{\epsilon}{|\ln\epsilon|^2}\Big|  \|\phi_2-\phi_1 \|.
\end{align*}

\emph{Estimate of $  K_3$}:
By (\ref{sumbus}), one has
\begin{eqnarray*}
 K_3 &  = &
  \Big|\Big(\sum_{i=1}^k(-1)^if_0^{'}(PU_{ \mu_i,\xi_i}) -\sum_{i=1}^k(-1)^if_0^{'}( U_{ \mu_i,\xi_i})\Big) (\phi_2-\phi_1) \Big|_{\frac{2n}{n+2}}\\
  &\leq &  k\Big|f_0^{'}(PU_{ \mu_i,\xi_i}) -f_0^{'}( U_{ \mu_i,\xi_i})\Big|_{\frac{ n}{ 2}}
  |\phi_2-\phi_1 |_{\frac{2n}{n-2}}\\
  & \leq &
\left\{ \arraycolsep=1.5pt
   \begin{array}{lll}
  O\Big(\frac{\epsilon}{|\ln\epsilon|^2}\|\phi_2-\phi_1 \| \Big)\ \   &{\rm if }\   3\leq n\leq 5 , \\[2mm]
 O\bigg( \frac{\epsilon}{|\ln\epsilon|^2}
 \Big|\ln\frac{\epsilon}{|\ln\epsilon|^2}\Big|^{\frac{1}{2}}\|\phi_2-\phi_1 \|\bigg)\ \   &{\rm if }\  n=6, \\[2mm]
 O\bigg(\Big(\frac{\epsilon}{|\ln\epsilon|^2}\Big)^{\frac{2}{n-2} }\|\phi_2-\phi_1 \|\bigg)\ \  & {\rm if}\   n\geq 7.
\end{array}
\right.
\end{eqnarray*}

From  $K_1$-$K_3$,
if $\epsilon$ is sufficient small, there exists a constant $L^*\in (0,1)$ such that
\[
\|T_{ \mu,\xi}(\phi_2)-T_{ \mu,\xi}(\phi_1) \|
  \leq L^* \|\phi_2-\phi_1 \|.
\]
It follows that
$T_{ \mu,\xi}$  is a contraction mapping from $\tilde{B }$ to $ \tilde{B }$,
then,
it has a unique fixed point  $\phi\in \tilde{B }$.
This concludes the proof.
\qed

\section{Proof of  Proposition \ref{leftside}}

This section is devoted to prove Proposition \ref{leftside}.

{  \emph{\textbf{ Proof of  Part a}}}.
We consider the following perturbation problem
\begin{eqnarray}\label{lineapro}
\left\{ \arraycolsep=1.5pt
   \begin{array}{lll}
-\Delta (V+\phi) =f_\epsilon(V+\phi )
  +\sum\limits_{i=0}^k\sum\limits_{l=0}^nc_{il}  U_{\mu_{\epsilon i},\xi_{\epsilon i}}^{p-1} P\psi^l_{ \mu_{\epsilon i},\xi_{\epsilon i}} \ \   &{\rm in}\ \Omega, \\[2mm]
\sum\limits_{i=1}^k\int_\Omega   U_{\mu_{\epsilon i},\xi_{\epsilon i}}^{p-1} P\psi^l_{ \mu_{\epsilon i},\xi_{\epsilon i}}\phi dx= 0 \ \  & {\rm for }\ l=0,1,\cdots,n.
\end{array}
\right.
\end{eqnarray}
From (\ref{aosd}), we have
\begin{equation}\label{eifqo}
  \int_\Omega  \Delta (V+\phi) P\psi^h_{ \mu_j,\xi_j}dx +\int_\Omega f_\epsilon(V+\phi ) P\psi^h_{ \mu_j,\xi_j}dx.
  =0,
\end{equation}
Thus, by (\ref{lineapro}) and  (\ref{eifqo}), we obtain
\[
 c_{il} \int_\Omega U_{\mu_{\epsilon i},\xi_{\epsilon i}}^{p-1} P\psi^l_{ \mu_{\epsilon i},\xi_{\epsilon i}} P\psi^h_{ \mu_j,\xi_j}dx
 =0,
\]
which means that $c_{il}=0$ for  $i=1,\cdots,k$ and $l=0,1,\cdots,n$.
Then $V+\phi $ is a solution of problem (\ref{eqa}).

{  \emph{\textbf{ Proof of  Part b}}}.
There holds
\begin{align*}
  & \Big\langle V+\phi -i^*[f_\epsilon(V+\phi )], P\psi^h_{ \mu_j,\xi_j}\Big\rangle\\
  = & \sum_{i=1}^k \langle PU_{ \mu_i,\xi_i}, P\psi^h_{ \mu_j,\xi_j} \rangle
      -\int_\Omega f_\epsilon(V+\phi )P\psi^h_{ \mu_j,\xi_j}dx \\
  = & \sum_{i=1}^k\int_\Omega (-1)^i f_0(U_{ \mu_i,\xi_i})P\psi^h_{ \mu_j,\xi_j}dx -\int_\Omega f_\epsilon(V+\phi )P\psi^h_{ \mu_j,\xi_j}dx \\
  = & \sum_{i=1}^k\int_\Omega (-1)^i\Big[f_0(U_{ \mu_i,\xi_i})-f_0(PU_{ \mu_i,\xi_i}) \Big] \psi^h_{ \mu_j,\xi_j}dx\\
    &  +\sum_{i=1}^k \int_\Omega(-1)^i \Big[f_0(U_{ \mu_i,\xi_i})-f_0(PU_{ \mu_i,\xi_i}) \Big](P\psi^h_{
    \mu_j,\xi_j} -\psi^h_{ \mu_j,\xi_j})dx \\
    & + \int_\Omega \Big[\sum_{i=1}^k(-1)^if_0(PU_{ \mu_i,\xi_i})-f_\epsilon(V) \Big] \psi^h_{ \mu_j,\xi_j}dx\\
    & + \int_\Omega \Big[\sum_{i=1}^k(-1)^if_0(PU_{ \mu_i,\xi_i})-f_\epsilon(V) \Big](P\psi^h_{ \mu_j,\xi_j}-\psi^h_{ \mu_j,\xi_j})dx \\
    & - \int_\Omega \Big[f_\epsilon(V+\phi  )-f_\epsilon(V) -f^{'}_\epsilon(V) \phi  \Big]P\psi^h_{ \mu_j,\xi_j}dx
     - \int_\Omega  [f^{'}_\epsilon(V)-f^{'}_0(V) ] \phi  P\psi^h_{ \mu_j,\xi_j}dx\\
    & - \int_\Omega \Big[f^{'}_0(V)-\sum_{i=1}^k(-1)^if_0^{'}(U_{ \mu_i,\xi_i})\Big] \phi  P\psi^h_{ \mu_j,\xi_j}dx\\
    & - \sum_{i=1}^k \int_\Omega (-1)^if^{'}_0( U_{ \mu_i,\xi_i})\phi (P\psi^h_{ \mu_j,\xi_j}-\psi^h_{ \mu_j,\xi_j})dx
     - \sum_{i=1}^k \int_\Omega (-1)^if^{'}_0( U_{ \mu_i,\xi_i})\phi \psi^h_{ \mu_j,\xi_j}dx\\
    = & P_1+\cdots,P_9.
\end{align*}

\emph{Estimate of $ P_1$}:
It holds
\begin{align*}
 P_1
 = & \sum_{i=1}^k\int_\Omega(-1)^i \Big[f_0(U_{ \mu_i,\xi_i})-f_0(PU_{ \mu_i,\xi_i}) \Big] \psi^h_{ \mu_j,\xi_j}dx \\
 = &  - \sum_{i=1}^k\int_\Omega (-1)^if^{'}_0(U_{ \mu_i,\xi_i})(PU_{ \mu_i,\xi_i}-U_{ \mu_i,\xi_i})  \psi^h_{ \mu_j,\xi_j}dx\\
& - \sum_{i=1}^k\int_\Omega(-1)^i\Big[ f_0(PU_{ \mu_i,\xi_i})- f_0(U_{ \mu_i,\xi_i})-f^{'}_0(U_{ \mu_i,\xi_i})(PU_{ \mu_i,\xi_i}-U_{ \mu_i,\xi_i}) \Big] \psi^h_{ \mu_j,\xi_j}dx.
\end{align*}
If $h=0$, by (\ref{eq1}) and (\ref{psi}), we have
\begin{eqnarray*}
&&
-\int_\Omega f^{'}_0(U_{ \mu_i,\xi_i})(PU_{ \mu_i,\xi_i}-U_{ \mu_i,\xi_i})  \psi^0_{ \mu_j,\xi_j}dx\\
  & = &
     p\alpha_n  \mu_i^{\frac{n-2}{2}}\int_\Omega  U_{ \mu_i,\xi_i}^{p-1}
      H(x,\xi_i) \psi^0_{ \mu_j,\xi_j}dx  \\
  & = &
     p\alpha_n \mu_i^{ n-2 }\Big(H(\xi_i,\xi_i)+ O( \mu_i) \Big)
     \int_{\frac{\Omega-\xi_i}{ \mu_i}}
     \frac{1}{(1+|y|^2)^2} \psi^0\Big(\frac{\mu_iy+\xi_i-\xi_j}{\mu_j}\Big) dy \\
& =&
\left\{ \arraycolsep=1.5pt
   \begin{array}{lll}
  p\alpha_n \mu_i^{ n-2 }\Big( H(\xi_i,\xi_i)+O( \mu_i)\Big) \int_{\mathbb{R}^n}  U^{p-1}(y)\psi^0(y) dy
 & \quad\mbox{if}\ j=i, \\[2mm]
   \frac{(n-2)\alpha_n^2p}{2}\mu_i^{ n-2 }\Big(H(\xi_i,\xi_i)+ O( \mu_i) \Big)
     \int_{\frac{\Omega-\xi_i}{ \mu_i}}
     \frac{1}{(1+|y|^2)^2}
   \frac{|\mu_iy+\xi_i-\xi_j |^2-\mu_j^2}
   {(\mu_j^2+| \mu_iy+\xi_i-\xi_j |^2)^{\frac{n }{2}}}dy
 & \quad\mbox{if}\ j>i, \\[2mm]
 \end{array}
 \right.\\
& = &
 \left\{ \arraycolsep=1.5pt
  \begin{array}{lll}
    \alpha_n a_1 H(\xi,\xi)  \mu_i^{ n-2 }
  +O( \mu_i^{n-1})
 & \quad\mbox{if}\ j=i\ \mbox{and}\ h=0, \\[2mm]
   C H(\xi,\xi) + O( \mu_i^{n-1})
 & \quad\mbox{if}\ j>i\ \mbox{and}\ h=0. \\[2mm]
\end{array}
\right.
\end{eqnarray*}

If $ h=1,\cdots,n$ and $j=i$, we set $\partial_{\xi_i^h}\varphi(\xi)=\frac{\partial \varphi(\xi) }{ \partial \xi_i^h}$,  by (\ref{xy}), one has
\begin{align*}
& -\int_\Omega f^{'}_0(U_{ \mu_i,\xi_i})(PU_{ \mu_i,\xi_i}-U_{ \mu_i,\xi_i})  \psi^h_{ \mu_j,\xi_j}dx\\
  = & p\alpha_n \mu_i^{\frac{n-2}{2}}
\int_\Omega  U_{ \mu_i,\xi_i}^{p-1}H(x,\xi_i)\psi^h_{ \mu_i,\xi_i}dx \\
 = &  \alpha_n  \mu_i^{\frac{n}{2}} \int_\Omega
H(x,\xi_i) \frac{\partial U^p_{ \mu_i,\xi_i} }{ \partial\xi_i^h}dx \\
 = & \alpha_n  \mu_i^{\frac{n}{2}}\Big[\frac{\partial }{ \partial \xi_i^h}
     \int_\Omega U^p_{ \mu_i,\xi_i}  H(x,\xi_i) dx-\int_\Omega U^p_{ \mu_i,\xi_i} \frac{\partial H(x,\xi_i)}{ \partial \xi_i^h} dx\Big]\\
  = &  \alpha_n  \mu_i^{\frac{n}{2}}\Big[ \mu_i^{\frac{n-2}{2}}\frac{\partial }{ \partial \xi_i^h}
     \int_{\frac{\Omega-\xi_i}{ \mu_i}} U^p_{ \mu_i,\xi_i} H( \mu_iy+\xi_i,\xi_i) dy
     - \mu_i^{\frac{n-2}{2}}\int_{\frac{\Omega-\xi_i}{ \mu_i}} U^p_{ \mu_i,\xi_i} \frac{\partial H( \mu_iy+\xi_i,\xi_i)}{ \partial \xi_i^h}dy\Big]\\
  = & \alpha_n a_2  \mu_i^{n-1}\Big[ \frac{\partial ( H(\xi_i,\xi_i)) }{ \partial \xi_i^h}
      -   \frac{\partial H(\xi_i,\xi_i)}{ \partial \xi_i^h} +O( \mu_i)\Big]\\
  = & \alpha_n a_2  \mu_i^{n-1}\Big(\frac{1}{2}\partial_{\xi_h}\varphi(\xi) +O( \mu_i)\Big).
\end{align*}
If $ h=1,\cdots,n$ and $j>i$, by Lemma \ref{exo} and (\ref{psi}), one has
\begin{align*}
& \int_\Omega f^{'}_0(U_{ \mu_i,\xi_i})(PU_{ \mu_i,\xi_i}-U_{ \mu_i,\xi_i})  \psi^h_{ \mu_j,\xi_j}dx\\
  = & p\alpha_n \mu_i^{\frac{n-2}{2}}
\int_\Omega  U_{ \mu_i,\xi_i}^{p-1}H(x,\xi_i)\psi^h_{ \mu_j,\xi_j}dx \\
 = &  (n-2)p \alpha_n^2 \mu_i^{\frac{n-2}{2}}\mu_j^{ \frac{n }{2}}
     \int_\Omega
      \frac{\mu_i^2}{(\mu_i^2+| x-\xi_i|^2)^2}H(x,\xi_i)
\frac{(x-\xi_j)_h}{(\mu_j^2+| x-\xi_j |^2)^{\frac{n }{2}}}dx\\
 = & O(\mu_i^{\frac{n-2}{2}}\mu_j^{\frac{n}{2}})
 =O\bigg(\Big(\frac{\epsilon}{|\ln\epsilon|^2}\Big)^{\frac{n+1}{n-2}}\bigg).
\end{align*}
On the other hand, by (\ref{subu2}), (\ref{subu3}) and (\ref{sumbu1}), we get
 \begin{align*}
 &  \int_\Omega\Big[ f_0(PU_{ \mu_i,\xi_i})- f_0(U_{ \mu_i,\xi_i})-f^{'}_0(U_{ \mu_i,\xi_i})(PU_{ \mu_i,\xi_i}-U_{ \mu_i,\xi_i}) \Big] \psi^h_{ \mu_j,\xi_j}dx\\
   \leq & \Big|f_0(PU_{ \mu_i,\xi_i})- f_0(U_{ \mu_i,\xi_i})-f^{'}_0(U_{ \mu_i,\xi_i})(PU_{ \mu_i,\xi_i}-U_{ \mu_i,\xi_i})\Big|_{\frac{n}{2}}|\psi^h_{ \mu_j,\xi_j}|_{\frac{n}{n-2}}\\
   = & o\bigg(\Big(\frac{\epsilon}{|\ln\epsilon|^2}\Big)^{\frac{n-1}{n-2}}\bigg)
   \quad\mbox{for}\ h=0,\cdots,n.
 \end{align*}
 Therefore, we get
\begin{eqnarray*}
  P_1
 & = &
\left\{ \arraycolsep=1.5pt
   \begin{array}{lll}
  \alpha_n a_1  \sum\limits_{i=1}^k \mu_i^{ n-2 }\Big( H(\xi_i,\xi_i)+O( \mu_i)\Big)+o( \mu_i^{n-2})
& \quad\mbox{if}\ j=i\ \mbox{and}\ h=0, \\[2mm]
  \sum\limits_{i=1}^k O\Big(H(\xi_i,\xi_i)+ O( \mu_i) \Big)+o( \mu_i^{n-2})
 & \quad\mbox{if}\ j>i\ \mbox{and}\ h=0, \\[2mm]
   \alpha_n a_2  \sum\limits_{i=1}^k\mu_i^{n-1}\Big(\frac{1}{2}\partial_{\xi_h}\varphi(x) +O( \mu_i)\Big)+o( \mu_i^{n-1})\ \   &\quad\mbox{if}\ j=i\ \mbox{and}\ h=1,\cdots,n, \\[2mm]
   O(\sum\limits_{i=1}^k\mu_i^{\frac{n-2}{2}}\mu_j^{\frac{n}{2}})+o( \mu_i^{n-1}) \ \   &\quad\mbox{if}\ j>i\ \mbox{and}\ h=1,\cdots,n, \\[2mm]
\end{array}
\right.\\
& = &
\left\{ \arraycolsep=1.5pt
   \begin{array}{lll}
  \alpha_n a_1 H(\xi,\xi)  \frac{\epsilon}{|\ln\epsilon|^2} d_1^{ n-2 }
  +o \Big(\frac{\epsilon}{|\ln\epsilon|^2}\Big)
 & \quad\mbox{if}\ j=i\ \mbox{and}\ h=0, \\[2mm]
   C H(\xi,\xi) + o \Big(\frac{\epsilon}{|\ln\epsilon|^2}\Big)
 & \quad\mbox{if}\ j>i\ \mbox{and}\ h=0, \\[2mm]
   \frac{1}{2}\alpha_n a_2\partial_{\xi_h}\varphi(\xi)
   \Big( \frac{\epsilon}{|\ln\epsilon|^2}\Big)^{\frac{n-1}{n-2}}d_1^{ n-2 }
    + o\bigg(\Big(\frac{\epsilon}{|\ln\epsilon|^2}\Big)^{\frac{n-1}{n-2}}
    \bigg)\ \   &\quad\mbox{if}\ j=i\ \mbox{and}\ h=1,\cdots,n, \\[2mm]
   O\bigg(\Big(\frac{\epsilon}{|\ln\epsilon|^2}\Big)^{\frac{n-1}{n-2}}\bigg)\ \   &\quad\mbox{if}\ j>i\ \mbox{and}\ h=1,\cdots,n. \\[2mm]
\end{array}
\right.
\end{eqnarray*}

\emph{Estimate of $ P_2$}:
By (\ref{sumbu3}) and (\ref{gisewr1}), we deduce
\begin{eqnarray*}
  P_2
 & = & \sum_{i=1}^k \int_\Omega\Big| (-1)^i \Big[f_0(U_{ \mu_i,\xi_i})-f_0(PU_{ \mu_i,\xi_i}) \Big](P\psi^h_{
    \mu_j,\xi_j} -\psi^h_{ \mu_j,\xi_j})\Big|dx  \\
& \leq & C\sum\limits_{i=1}^k \Big|f_0(U_{ \mu_i,\xi_i})-f_0(PU_{ \mu_i,\xi_i})\Big|_{\frac{2n}{n+2}}
 \Big|P\psi^h_{ \mu_j,\xi_j} -\psi^h_{ \mu_j,\xi_j} \Big|_{\frac{2n}{n-2}} \\
& =&
\left\{ \arraycolsep=1.5pt
   \begin{array}{lll}
 o\Big( \frac{\epsilon}{|\ln\epsilon|^2}\Big)\ \   &{\rm if}\ h=0, \\[2mm]
 o\bigg(\Big(\frac{\epsilon}{|\ln\epsilon|^2}\Big)^{\frac{n-1}{n-2}}\bigg)\ \   &{\rm if}\ h=1,\cdots,n. \\[2mm]
\end{array}
\right.
\end{eqnarray*}

\emph{Estimate of $P_3$}:
The main proof of $P_3$ shows in  Lemma \ref{pf1},   the final result is
\begin{eqnarray*}
P_3  & =  &
  \int_\Omega \Big[\sum\limits_{i=1}^k(-1)^if_0(PU_{ \mu_i,\xi_i})-f_\epsilon(V) \Big] \psi^h_{ \mu_j,\xi_j}dx\\
& =& \left\{ \arraycolsep=1.5pt
   \begin{array}{lll}
  \alpha_n a_1  \frac{\epsilon}{|\ln\epsilon|^2}
  H(\xi,\xi)d_1^{n-2}
  +a_3\frac{\epsilon}{|\ln\epsilon|^2}\sum\limits_{i=1}^{k-1}\Big(\frac{d_{i+1}}{ d_i}\Big)^{\frac{n-2}{2}} g(\sigma_i)
  -\frac{2k^2}{(n-2)^2} a_4\epsilon\Big|\ln \frac{\epsilon}{|\ln\epsilon|^2}\Big|\\
\quad   -a_4 \frac{\epsilon}{|\ln\epsilon|^2}\sum\limits_{i=1}^k \frac{2}{2i-1}|\ln  d_i|
+o\Big(\frac{\epsilon}{|\ln\epsilon|^2}\Big)\ \   \   {\rm if}\  h=0, \\[2mm]
   \frac{1}{2}\alpha_n a_2  \Big(\frac{\epsilon}{|\ln\epsilon|^2}\Big)^{\frac{n-1}{n-2}} d_1^{n-1} \partial_{\xi_h}\varphi(\xi)
   + o\bigg(\Big(\frac{\epsilon}{|\ln\epsilon|^2}\Big)^{\frac{n-1}{n-2}}\bigg)\ \  \    {\rm if}\  h=1,\cdots,n,\\[2mm]
\end{array}
\right.
\end{eqnarray*}

\emph{Estimate of $P_4$}:
 From (\ref{fepli1}) and (\ref{gisewr1}), we get
 \begin{eqnarray*}
  P_4
  & = &
  \int_\Omega \Big|\Big[\sum_{i=1}^k(-1)^if_0(PU_{ \mu_i,\xi_i})-f_\epsilon(V) \Big](P\psi^h_{ \mu_j,\xi_j}-\psi^h_{ \mu_j,\xi_j})\Big|dx \\
  & \leq & C \Big|f_\epsilon(V)-\sum_{i=1}^k(-1)^if_0(PU_{ \mu_i,\xi_i})\Big|_{\frac{2n}{n+2}}
  \Big| P\psi^h_{ \mu_j,\xi_j}-\psi^h_{ \mu_j,\xi_j}\Big|_{\frac{2n}{n-2}} \\
 & = &\left\{ \arraycolsep=1.5pt
   \begin{array}{lll}
 o\Big(\frac{\epsilon}{|\ln\epsilon|^2}\Big)\ \   &{\rm if}\ h=0 \ {\rm and}\ 3\leq n \leq 4, \ \\[2mm]
  o\bigg(\Big(\frac{\epsilon}{|\ln\epsilon|^2}\Big)^{\frac{n-1}{n-2}}\bigg)\ \   &{\rm if}\ h=0\  {\rm and}\ n\geq5,\\[2mm]
  o\Big(\frac{\epsilon}{|\ln\epsilon|^2}\Big)\ \ & {\rm if}\ h=1,\cdots,n \  {\rm and}\ n\geq3.
 \end{array}
 \right.
 \end{eqnarray*}

\emph{Estimate of $P_5$}:
From  (\ref{gisewr1}), (\ref{subu2}),  (\ref{subu3}), (\ref{dkoa}) and
(\ref{dt}),   we have
 \begin{eqnarray*}
  P_5 & = &
  \int_\Omega\Big| \Big[f_\epsilon(V+\phi  )-f_\epsilon(V) -f^{'}_\epsilon(V) \phi  \Big]P\psi^h_{ \mu_j,\xi_j}\Big|dx\\
  & = & O\Big(\Big|f_\epsilon(V+\phi  )-f_\epsilon(V) -f^{'}_\epsilon(V) \phi )\Big|_{\frac{2n}{n+2}}
   |P\psi^h_{ \mu_j,\xi_j} |_{\frac{2n}{n-2}}\Big)\\
 & \leq &\left\{ \arraycolsep=1.5pt
   \begin{array}{lll}
  O(1+\|\phi\|^{p-2}|)\|\phi\|^2\ \   &{\rm if}\ 3\leq n\leq5, \\[2mm]
  O( \|\phi\|^2)\ \   &{\rm if}\ n=6,  \\[2mm]
  O(\epsilon+\|\phi\|^{p-1}|)\|\phi\|^2\ \   &{\rm if}\   n\geq 7, \\[2mm]
 \end{array}
 \right.\\
 &  = &
  \left\{ \arraycolsep=1.5pt
   \begin{array}{lll}
 o\Big( \frac{\epsilon}{|\ln\epsilon|^2}\Big)\ \   &{\rm if}\ n=3, \ \\[2mm]
 o\bigg(\Big(\frac{\epsilon}{|\ln\epsilon|^2}\Big)^{\frac{n-1}{n-2}}\bigg)\ \   &{\rm if}\  n\geq4.\\[2mm]
 \end{array}
 \right.
 \end{eqnarray*}

\emph{Estimate of $P_6$}:
By  (\ref{fepli2}) and  (\ref{phi}), it holds
 \begin{eqnarray*}
  P_6 & = &
  \int_\Omega  \Big| [f^{'}_\epsilon(V)-f^{'}_0(V)] \phi  P\psi^h_{ \mu_j,\xi_j}\Big|dx\\
  & = & O\Big( |f^{'}_\epsilon(V)-f^{'}_0(V) |_{\frac{ n}{ 2}}
   |\phi|_{\frac{2 n}{n- 2}}|P\psi^h_{ \mu_j,\xi_j} |_{\frac{2n}{n-2}}\Big)\\
 & = &
    \left\{ \arraycolsep=1.5pt
   \begin{array}{lll}
 o\Big( \frac{\epsilon}{|\ln\epsilon|^2}\Big)\ \   &{\rm if}\ h=0, \ \\[2mm]
 o\bigg(\Big(\frac{\epsilon}{|\ln\epsilon|^2}\Big)^{\frac{n-1}{n-2}}\bigg)\ \   &{\rm if}\  h=1,\cdots,n..\\[2mm]
 \end{array}
 \right.
 \end{eqnarray*}

\emph{Estimate of $P_7$}:
By (\ref{sumbu2}), (\ref{phi}), (\ref{subu3}), (\ref{subu2}), (\ref{gisewr1}), (\ref{sumbus}),  there holds
\begin{eqnarray*}
 P_7  & = & \int_\Omega \Big|\Big[f^{'}_0(V)-\sum\limits_{i=1}^k(-1)^if^{'}_0(U_{ \mu_i,\xi_i})\Big]\phi P\psi^h_{ \mu_j,\xi_j}\Big |dx\\
  & = & O\bigg(\Big|\Big[f^{'}_0(V)-\sum\limits_{i=1}^k(-1)^if^{'}_0(U_{ \mu_i,\xi_i})\Big]  P\psi^h_{ \mu_j,\xi_j}\Big|_{\frac{2n}{n+2}} \|\phi\|\bigg)\\
  & = & O\bigg(\Big|f^{'}_0(V)-\sum\limits_{i=1}^k(-1)^if^{'}_0(PU_{ \mu_i,\xi_i})\Big|_{\frac{n}{2}} |  P\psi^h_{ \mu_j,\xi_j}|_{\frac{2n}{n+2}} \|\phi\|\bigg)\\
&   +& O\bigg(\sum\limits_{i=1}^k(-1)^i\Big| f^{'}_0(PU_{ \mu_i,\xi_i})-f^{'}_0(U_{ \mu_i,\xi_i})\Big|_{\frac{n}{2}}  | P\psi^h_{ \mu_j,\xi_j} |_{\frac{2n}{n+2}} \|\phi\|\bigg)\\
 & = &
    \left\{ \arraycolsep=1.5pt
   \begin{array}{lll}
 o\Big( \frac{\epsilon}{|\ln\epsilon|^2}\Big)\ \   &{\rm if}\ h=0, \ \\[2mm]
 o\bigg(\Big(\frac{\epsilon}{|\ln\epsilon|^2}\Big)^{\frac{n-1}{n-2}}\bigg)\ \   &{\rm if}\  h=1,\cdots,n.\\[2mm]
 \end{array}
 \right.
 \end{eqnarray*}

\emph{Estimate of $P_8$}:
For $h=0$, by (\ref{subu1}), (\ref{phi}) and (\ref{gisewr1}), it follows that
\begin{eqnarray*}
 P_8 &= &
  \sum_{i=1}^k \int_\Omega \Big|(-1)^if^{'}_0( U_{ \mu_i,\xi_i})\phi (P\psi^0_{ \mu_j,\xi_j}-\psi^0_{ \mu_j,\xi_j})\Big|dx\\
 & = & O\bigg(\sum\limits_{i=1}^k\Big|f^{'}_0(U_{ \mu_i,\xi_i})\Big|_{\frac{ n}{ 2}} |\phi|_{\frac{2n}{n-2}} \Big|P\psi^0_{ \mu_j,\xi_j} -\psi^0_{ \mu_j,\xi_j} \Big|_{\frac{2n}{n-2}} \bigg)\\
 & = &  O\Big(\sum\limits_{i=1}^k  \mu_i^{\frac{ n-2}{ 2}}  \|\phi\|\Big)
   =  \left\{ \arraycolsep=1.5pt
   \begin{array}{lll}
 o\Big( \frac{\epsilon}{|\ln\epsilon|^2}\Big)\ \   &{\rm if}\   n=3, \ \\[2mm]
 o\bigg(\Big(\frac{\epsilon}{|\ln\epsilon|^2}\Big)^{\frac{n-1}{n-2}}\bigg)\ \   &{\rm if}\  n\geq 4,\\[2mm]
 \end{array}
 \right.
\end{eqnarray*}
and for $h=1,\cdots,n$,   we  obtain
\begin{eqnarray*}
 P_8
 &= & \sum_{i=1}^k \int_\Omega\Big| (-1)^if^{'}_0( U_{ \mu_i,\xi_i})\phi (P\psi^h_{ \mu_j,\xi_j}-\psi^h_{ \mu_j,\xi_j})\Big|dx\\
 & = & O\bigg(\Big|\sum\limits_{i=1}^kf^{'}_0(U_{ \mu_i,\xi_i})\Big|_{\frac{ n}{ 2}} |\phi|_{\frac{2n}{n-2}}  |P\psi^h_{ \mu_j,\xi_j} -\psi^h_{ \mu_j,\xi_j} |_{\frac{2n}{n-2}} \bigg)\\
 & = &  O\Big(\sum\limits_{i=1}^k  \mu_i^{\frac{ n }{ 2}}  \|\phi\|\Big)
   =  \left\{ \arraycolsep=1.5pt
   \begin{array}{lll}
 o \Big(\frac{\epsilon}{|\ln\epsilon|^2}\Big) \ \   &{\rm if}\  3\leq n\leq 5, \ \\[2mm]
 o\bigg(\Big(\frac{\epsilon}{|\ln\epsilon|^2}\Big)^{\frac{n-1}{n-2}}\bigg)\ \   &{\rm if}\  n\geq 6.\\[2mm]
 \end{array}
 \right.
\end{eqnarray*}

\emph{Estimate of $ P_9$}:
We use $\phi $ to multiply (\ref{linear-equ}) and integral in the $\Omega$,  then
\begin{align*}
 \int_\Omega f^{'}_0( U_{ \mu_i,\xi_i})\phi \psi^h_{ \mu_j,\xi_j}dx=0.
\end{align*}
From $ P_1$-$ P_9$,  we complete  the proof.
\qed

\section{ Appendix}

We collect some well known estimates.

\begin{lemma} \cite{rey}\label{exo}
Let $\xi\in\Omega$,
$\mu>0$ is small, there holds
\begin{equation}\label{eq1}
  PU_{ \mu,\xi}(x)=U_{ \mu,\xi}(x)-\alpha_n \mu^{\frac{n-2}{2}}H(x,\xi)+O( \mu^{\frac{n+2}{2}}),
\end{equation}
\begin{equation}\label{eq2}
  P\psi^0_{ \mu,\xi}(x)=\psi^0_{ \mu,\xi}(x)-\frac{n-2}{2}\alpha_n \mu^{\frac{n-2}{2}}H(x,\xi)+O( \mu^{\frac{n+4}{2}}),
\end{equation}
\begin{equation}\label{eq3}
  P\psi^h_{ \mu,\xi}(x)=\psi^h_{ \mu,\xi}(x)-\alpha_n \mu^{\frac{n }{2}}\partial_{\xi_h}H(x,\xi)+O( \mu^{\frac{n+2}{2}}),
\end{equation}
as $\mu\rightarrow 0$ uniformly with respect to $\xi$ in compact subsets of $\Omega$,
where $h=1,\cdots,n$ and $\alpha_n$ is given in (\ref{bubble}).
Moreover,
\begin{eqnarray}\label{gisewr1}
 |P\psi^h_{ \mu_j,\xi_j}-\psi^h_{ \mu_j,\xi_j}|_{\frac{2n}{n-2}}=
\left\{ \arraycolsep=1.5pt
   \begin{array}{lll}
 O\bigg(\Big(\frac{\epsilon}{|\ln\epsilon|^2}\Big)^{\frac{1}{2}}\bigg)\ \   &{\rm if}\ h=0, \\[2mm]
 O\bigg(\Big(\frac{\epsilon}{|\ln\epsilon|^2}\Big)^{\frac{n}{2(n -2)}}\bigg) \ \  & {\rm if}\ h=1,\cdots,n,
\end{array}
\right.
\end{eqnarray}
 for  $j=1,\cdots,k$.
\end{lemma}

\begin{lemma}\cite{mp}\label{deia}
There holds
\begin{eqnarray}\label{subu1}
 \int_\Omega U_{ \mu,\xi }^q(x)dx=
\left\{ \arraycolsep=1.5pt
   \begin{array}{lll}
 O\bigg(\Big(\frac{\epsilon}{|\ln\epsilon|^2}\Big)^{\frac{q}{2}}\bigg)\ \   &{\rm if }\  0<q<\frac{n}{n-2}, \\[2mm]
O\bigg(\Big(\frac{\epsilon}{|\ln\epsilon|^2}\Big)^{\frac{n}{2(n-2)}}
\Big|\ln\frac{\epsilon}{|\ln\epsilon|^2}\Big|\bigg)\ \   &{\rm if }\  q=\frac{n}{n-2}, \\[2mm]
 O\bigg(\Big(\frac{\epsilon}{|\ln\epsilon|^2}\Big)^{\frac{n}{n-2}-\frac{q}{2}}\bigg)
\ \  & {\rm if}\   \frac{n}{n-2}<q\leq\frac{2n}{n-2},
\end{array}
\right.
\end{eqnarray}
\begin{eqnarray}\label{subu2}
 \int_\Omega |\psi^0_{ \mu,\xi }(x)|^qdx=
\left\{ \arraycolsep=1.5pt
   \begin{array}{lll}
 O\bigg(\Big(\frac{\epsilon}{|\ln\epsilon|^2}\Big)^{\frac{q}{2}}\bigg)\ \   &{\rm if }\  0<q<\frac{n}{n-2}, \\[2mm]
O\bigg(\Big(\frac{\epsilon}{|\ln\epsilon|^2}\Big)^{\frac{n}{2(n-2)}}
\Big|\ln\frac{\epsilon}{|\ln\epsilon|^2}\Big|\bigg)\ \   &{\rm if }\  q=\frac{n}{n-2}, \\[2mm]
 O\bigg(\Big(\frac{\epsilon}{|\ln\epsilon|^2}\Big)^{\frac{n}{n-2}-\frac{q}{2}}\bigg)
\ \  & {\rm if}\   \frac{n}{n-2}<q\leq\frac{2n}{n-2},
\end{array}
\right.
\end{eqnarray}
and
\begin{eqnarray}\label{subu3}
\int_\Omega |\psi^h_{ \mu,\xi }(x)|^qdx=
\left\{ \arraycolsep=1.5pt
  \begin{array}{lll}
 O\bigg(\Big(\frac{\epsilon}{|\ln\epsilon|^2}\Big)
^{\frac{nq}{2(n-2)} }\bigg)\ \   &{\rm if }\  0<q<\frac{n}{n-1}, \\[2mm]
O\bigg( \Big(\frac{\epsilon}{|\ln\epsilon|^2}\Big) ^{\frac{n^2}{2(n-1)(n-2)}}\Big|\ln\frac{\epsilon}{|\ln\epsilon|^2} \Big| \bigg)\ \   &{\rm if }\ q=\frac{n}{n-1}, \\[2mm]
O\bigg(\Big(\frac{\epsilon}{|\ln\epsilon|^2}\Big)^{\frac{n}{n-2}-\frac{q}{2}}\bigg)\ \  & {\rm if}\   \frac{n}{n-1}<q\leq\frac{2n}{n-2},
\end{array}
\right.
\end{eqnarray}
 for  $h=1,\cdots,n$.
 Moreover,
\begin{eqnarray}\label{sumbu3}
 \Big|  f_0( PU_{ \mu,\xi })- f_0(U_{ \mu,\xi }) \Big|_{\frac{2n}{n+2}}
   =
\left\{ \arraycolsep=1.5pt
   \begin{array}{lll}
 O \Big(\frac{\epsilon}{|\ln\epsilon|^2}\Big) \ \   &{\rm if }\  3\leq n\leq5, \\[2mm]
 O\bigg( \frac{\epsilon}{|\ln\epsilon|^2} \Big|\ln  \frac{\epsilon}{|\ln\epsilon|^2}\Big|^{\frac{2}{3}}\bigg)\ \   &{\rm if }\  n=6, \\[2mm]
 O\bigg(\Big(\frac{\epsilon}{|\ln\epsilon|^2}\Big)^{\frac{n+2}{2(n-2)}}\bigg)\ \  & {\rm if}\   n\geq 7,
\end{array}
\right.
\end{eqnarray}

\begin{eqnarray}\label{sumbus}
 \Big|  f^{'}_0( PU_{ \mu,\xi })- f^{'}_0(U_{ \mu,\xi }) \Big|_{\frac{n}{2}}
   =
\left\{ \arraycolsep=1.5pt
   \begin{array}{lll}
 O \Big(\frac{\epsilon}{|\ln\epsilon|^2}\Big) \ \   &{\rm if }\  n=3, \\[2mm]
 O\bigg( \frac{\epsilon}{|\ln\epsilon|^2} \Big|\ln  \frac{\epsilon}{|\ln\epsilon|^2}\Big|^{\frac{1}{2}}\bigg)\ \   &{\rm if }\  n=4, \\[2mm]
 O\bigg(\Big(\frac{\epsilon}{|\ln\epsilon|^2}\Big)^{\frac{2}{n-2}}\bigg)\ \  & {\rm if}\   n\geq 5,
\end{array}
\right.
\end{eqnarray}

\begin{eqnarray}\label{sumbu1}
&& \Big|  f_0( PU_{ \mu,\xi })- f_0(U_{ \mu,\xi })
     -  f^{'}_0(U_{ \mu,\xi }) (PU_{ \mu,\xi }-U_{ \mu,\xi }) \Big|_{\frac{n}{2}}\nonumber\\
    & = &
\left\{ \arraycolsep=1.5pt
   \begin{array}{lll}
 O\bigg( \Big(\frac{\epsilon}{|\ln\epsilon|^2}\Big)^{\frac{5}{2}}\bigg)\ \   &{\rm if }\  n=3, \\[2mm]
 O\bigg( \Big(\frac{\epsilon}{|\ln\epsilon|^2}\Big)^{\frac{3}{2 }} \Big|\ln\frac{\epsilon}{|\ln\epsilon|^2}\Big|^{\frac{1}{2}}\bigg)\ \   &{\rm if }\  n=4, \\[2mm]
 O\bigg( \Big(\frac{\epsilon}{|\ln\epsilon|^2}\Big)^{\frac{n+2}{2(n-2)}}\bigg)\ \  & {\rm if}\   n\geq 5.
\end{array}
\right.
\end{eqnarray}
\end{lemma}

\begin{lemma} \label{inerpro}
It holds
\begin{eqnarray*}
 \Big\langle P\psi^l_{ \mu_i,\xi_i}, P\psi^h_{ \mu_j,\xi_j}\Big\rangle
 =
\left\{ \arraycolsep=1.5pt
   \begin{array}{lll}
 o\bigg(\Big(\frac{\epsilon}{|\ln\epsilon|^2}\Big)^{\frac{n}{n-2}}\bigg)\ \   &{\rm if}\ j>i, \\[2mm]
 O(1) \ \  & {\rm if}\ l\neq h,\\[2mm]
 c_h(1+o(1)) \ \  & {\rm if}\ i=j,\   l=h,
\end{array}
\right.
\end{eqnarray*}
for some positive constants $c_0$  and $c_1=\cdots,c_n$,
where $i$, $j=1,\cdots,k$  and  $h$, $l=0,\cdots,n$,
\end{lemma}

\begin{proof}
We have
\[
 \Big\langle P\psi^l_{ \mu_i,\xi_i}, P\psi^h_{ \mu_j,\xi_j}\Big\rangle
 =\int_\Omega f^{'}_0(U_{ \mu_i,\xi_i})\psi^l_{ \mu_i,\xi_i}\psi^h_{ \mu_j,\xi_j}dx
 +\int_\Omega f^{'}_0(U_{ \mu_i,\xi_i})\psi^l_{ \mu_i,\xi_i}
 ( P\psi^h_{ \mu_j,\xi_j}-\psi^h_{ \mu_j,\xi_j})dx.
\]
From (\ref{subu1}), (\ref{subu3}) and  (\ref{gisewr1}),  there holds
\begin{align*}
&\int_\Omega f^{'}_0(U_{ \mu_i,\xi_i})\psi^l_{ \mu_i,\xi_i}
 ( P\psi^h_{ \mu_j,\xi_j}-\psi^h_{ \mu_j,\xi_j})dx\\
 & \leq |f^{'}_0(U_{ \mu_i,\xi_i})|_{\frac{n}{2}}
 |\psi^l_{ \mu_i,\xi_i}|_{\frac{2n}{n-2}}
 |P\psi^h_{ \mu_j,\xi_j}-\psi^h_{ \mu_j,\xi_j}|_{\frac{2n}{n-2}}
 =o\bigg(\Big(\frac{\epsilon}{|\ln\epsilon|^2}\Big) \bigg).
\end{align*}
On the other hand,
if $l$, $h=1,\cdots,n$, the change of variables $x-\xi_i=\mu_i y$ shows that
\begin{eqnarray*}\label{gi1}
&& \int_\Omega f^{'}_0(U_{ \mu_i,\xi_i})\psi^l_{ \mu_i,\xi_i}\psi^h_{ \mu_j,\xi_j}dx\\
&  = &  (n-2)^2  \alpha_n^{\frac{2n}{n-2}}
  \mu_i^{\frac{n+4}{2}} \mu_j^{\frac{n}{2}}
  \int_\Omega
  \frac{(x-\xi_i)_l}{(\mu_i^2+|x-\xi_i|^2)^{\frac{n+4}{2}}}
  \frac{(x-\xi_j)_h}{(\mu_j^2+|x-\xi_j|^2)^{\frac{n}{2}}}dx \\
&  = &  (n-2)^2  \alpha_n^{\frac{2n}{n-2}}
  \mu_i^{\frac{n-2}{2}} \mu_j^{\frac{n}{2}}
  \int_{\frac{\Omega -\xi_i}{\mu_i}}\frac{y_l}{(1+|y|^2)^{\frac{n+4}{2}}}
  \frac{(\mu_i y+\xi_i- \xi_j)_h}{(\mu_j^2+|\mu_i y+\xi_i- \xi_j|^2)^{\frac{n}{2}}}dy\\
 &  = &
\left\{ \arraycolsep=1.5pt
   \begin{array}{lll}
 O\Big((\frac{\mu_j}{\mu_i})^{\frac{n}{2}}\Big)\ \   &{\rm if}\ j>i, \\[2mm]
 O(1) \ \  & {\rm if}\ i=j, \ l\neq h,\\[2mm]
 c_h(1+o(1)) \ \  & {\rm if}\ i=j, \ l= h.
\end{array}
\right.
\end{eqnarray*}
If $l=1,\cdots,n$ and $h=0$, we have
\begin{eqnarray*}\label{gi1}
&& \int_\Omega f^{'}_0(U_{ \mu_i,\xi_i})\psi^l_{ \mu_i,\xi_i}\psi^0_{ \mu_j,\xi_j}dx\\
&  = & \frac{(n-2)^2 }{2} \alpha_n^{\frac{2n}{n-2}}
  \mu_i^{\frac{n+4}{2}} \mu_j^{\frac{n-2}{2}}
  \int_\Omega
  \frac{(x-\xi_i)^l}{(\mu_i^2+|x-\xi_i|^2)^{\frac{n+4}{2}}}
  \frac{|x-\xi_j|^2-\mu_j^2}{(\mu_j^2+|x-\xi_j|^2)^{\frac{n}{2}}}dx \\
&  = &  (n-2)^2  \alpha_n^{\frac{2n}{n-2}}
  \mu_i^{\frac{n-2}{2}} \mu_j^{\frac{n-2}{2}}
  \int_{\frac{\Omega -\xi_i}{\mu_i}}\frac{y^l}{(1+|y|^2)^{\frac{n+4}{2}}}
  \frac{|\mu_i y+\xi_i- \xi_j|^2-\mu_j^2 }{(\mu_j^2+|\mu_i y+\xi_i- \xi_j|^2)^{\frac{n}{2}}}dy\\
 &  = &
 o\bigg(\Big(\frac{\epsilon}{|\ln\epsilon|^2}\Big) \bigg).
\end{eqnarray*}
Finally,
if $l=0$ and $h=0$, one has
\begin{eqnarray*}\label{gi1}
&& \int_\Omega f^{'}_0(U_{ \mu_i,\xi_i})\psi^0_{ \mu_i,\xi_i}\psi^0_{ \mu_j,\xi_j}dx\\
&  = & \frac{(n-2)^2 }{4} \alpha_n^{\frac{2n}{n-2}}
  \mu_i^{\frac{n-2}{2}} \mu_j^{\frac{n-2}{2}}
  \int_\Omega
  \frac{|x-\xi_i|^2-\mu_i^2}{(\mu_i^2+|x-\xi_i|^2)^{\frac{n}{2}}}
  \frac{|x-\xi_j|^2-\mu_j^2}{(\mu_j^2+|x-\xi_j|^2)^{\frac{n}{2}}}dx \\
&  = &  (n-2)^2  \alpha_n^{\frac{2n}{n-2}}
  \mu_i^{-\frac{n-2}{2}} \mu_j^{\frac{n-2}{2}}
  \int_{\frac{\Omega -\xi_i}{\mu_i}}\frac{(|y-\sigma_i|^2-1)}{(1+|y-\sigma_i|^2)^{\frac{n}{2}}}
  \frac{|\mu_i y+\xi_i- \xi_j|^2-\mu_j^2 }{(\mu_j^2+|\mu_i y+\xi_i- \xi_j|^2)^{\frac{n}{2}}}dy\\
 &  = &
\left\{ \arraycolsep=1.5pt
   \begin{array}{lll}
o\bigg(\Big(\frac{\epsilon}{|\ln\epsilon|^2}\Big) \bigg)\ \   &{\rm if}\ j>i, \\[2mm]
 c_0(1+o(1)) \ \  & {\rm if}\ i=j.
\end{array}
\right.
\end{eqnarray*}
Therefore, this lemma follows from above estimates.

\end{proof}

\begin{lemma}\cite{map}\label{expsi}
 Let $\xi\in\Omega$,   there holds
 \[
 P\psi^0_{\mu,\xi}(x)=\frac{n-2}{2}a_2 \mu^{\frac{n-2}{2}} G(x,\xi)
 +o \Big(\frac{\epsilon}{|\ln\epsilon|^2}\Big),\quad x\in\Omega,
 \]
 and
 \[
 P\psi^h_{\mu,\xi}(x)=a_2\mu^{\frac{n}{2}}\frac{\partial G}{\partial \xi_h}(x,\xi)
 +o\bigg(\Big(\frac{\epsilon}{|\ln\epsilon|^2}\Big)^{\frac{n-1}{n-2}}\bigg)  \ \    {\rm if}\ h=1,\cdots,n, \quad x\in\Omega,
 \]
 as $\epsilon\rightarrow 0$ uniformly on compact sets of $\Omega\setminus \{\xi\}$,
 where  $a_2$ is given in Proposition \ref{leftside}.
\end{lemma}

\begin{lemma}\cite{mp}\label{zsj}
Let $\theta>0$ and $u=\sum\limits_{i=1}^ku_i$, $i=1,\cdots,n$,
if $\epsilon>0$ small enough, for $u, u_i, v\in\mathbb{R}$, $p=2^*-1$,
  it holds that\\
$ (1)$
$ | f_\epsilon( u)-f_0(u)|
   \leq \epsilon |u|^p\ln\ln(e+|u |)$,\\
$ (2)$
$f^{'}_\epsilon(u)\leq C|u |^{p-1}$,\\
$ (3)$
$|f^{'}_\epsilon(u)-f^{'}_0(u) |
\leq \epsilon |u  |^{p-1}
\Big( p\ln\ln(e+|u | )+\frac{1}{\ln (e+ |u |)} \Big)$,\\
$ (4)$
\begin{eqnarray*}
|f^{'}_\epsilon(u+v)-f^{'}_\epsilon(u)|\leq
\left\{ \arraycolsep=1.5pt
   \begin{array}{lll}
 C (|u  |^{p-2}+|v|^{p-2} )|v|\ \   &{\rm if }\  n\leq6, \\[2mm]
 C (|v|^{p-1}+\epsilon |u|^{p-1})\ \  & {\rm if}\  n>6,
\end{array}
\right.
\end{eqnarray*}
$ (5)$
$\ln\ln  (e+\mu^{-\theta} u )=\ln\ln( \mu^{-\theta} )
+\ln\Big(1+\frac{ \ln(e^{1-\theta|\ln\mu|}+u)}
{\theta|\ln \mu|}\Big)$,\\
$ (6)$
$\lim\limits_{ \mu\rightarrow0}\bigg(|\ln\mu|\ln\Big(1+\frac{ \ln(e^{1-\theta|\ln \mu|}+u)}
{\theta|\ln\mu|}\Big)\bigg)=\frac{1}{\theta}\ln u$,\\
where $C$ is a   positive constant.

\end{lemma}

\begin{lemma}\label{yy}
There holds
\begin{eqnarray}\label{sumbu2}
 \Big|  f_0^{'}(V)-\sum\limits_{i=1}^k(-1)^if^{'}_0(PU_{ \mu_i,\xi_i})\Big|_{\frac{n}{2}}
    =
\left\{ \arraycolsep=1.5pt
   \begin{array}{lll}
  O\Big(\frac{\epsilon}{|\ln\epsilon|^2}\Big)\ \   &{\rm if }\   3\leq n\leq 5 , \\[2mm]
 O\bigg(\frac{\epsilon}{|\ln\epsilon|^2}
 \Big|\ln\frac{\epsilon}{|\ln\epsilon|^2}\Big|\bigg)\ \   &{\rm if }\  n=6, \\[2mm]
 O\bigg( \Big ( \frac{\epsilon}{|\ln\epsilon|^2}\Big)^{\frac{-n+8}{n-2} }\bigg)\ \  & {\rm if}\   n\geq 7,
\end{array}
\right.
\end{eqnarray}

\begin{eqnarray}\label{fepli1}
\ \  \Big|f_\epsilon(V)-\sum_{i=1}^k(-1)^if_0(PU_{ \mu_i,\xi_i}) \Big|_{\frac{2n}{n+2}}
   =
\left\{ \arraycolsep=1.5pt
   \begin{array}{lll}
  O\Big(\epsilon\ln\Big|\ln\frac{\epsilon}{|\ln\epsilon|^2}\Big| \Big)\ \   &{\rm if }\   3\leq n\leq 6 , \\[2mm]
 O\bigg(\Big(\frac{\epsilon}{|\ln\epsilon|^2}\Big)^{\frac{n+2}{2(n-2)} }\bigg)\ \  & {\rm if}\   n\geq 7,
\end{array}
\right.
\end{eqnarray}
\begin{equation}\label{fepli2}
 |f^{'}_\epsilon(V)-f^{'}_0(V) |_{\frac{n}{2}}
=O \bigg(\epsilon\ln\Big|\ln\frac{\epsilon}{|\ln\epsilon|^2}\Big|\bigg).
\end{equation}
\end{lemma}

\begin{proof}
Let us estimate (\ref{sumbu2}).
One has
\begin{align*}
 & \int_\Omega\Big|  f^{'}_0(V)-\sum\limits_{i=1}^k (-1)^if^{'}_0(PU_{ \mu_i,\xi_i}) \Big|^{\frac{n}{2}}dx\\
  = & \int_{\Omega\setminus B(\xi,\rho)}\Big|V^{p-1}
  - \sum\limits_{i=1}^k(-1)^i(PU_{ \mu_i,\xi_i})^{p-1} \Big|^{\frac{n}{2}}dx
  +\sum\limits_{l=1}^k\int_{\mathcal{A}_l}\Big|V^{p-1}
  - \sum\limits_{i=1}^k(-1)^i(PU_{ \mu_i,\xi_i})^{p-1} \Big|^{\frac{n}{2}}dx.
\end{align*}
We estimate the first term
\begin{align*}
  &  \int_{\Omega\setminus B(\xi,\rho)}\Big|V^{p-1}
  - \sum\limits_{i=1}^k(-1)^i(PU_{ \mu_i,\xi_i})^{p-1} \Big|^{\frac{n}{2}}dx\\
 \leq  & \sum\limits_{i=1}^k\int_{\Omega\setminus B(\xi,\rho)}
 U_{ \mu_i,\xi_i}^{(p-1)\frac{n}{2}}dx
  \leq C   \sum\limits_{i=1}^k \mu_i^n
  =O \bigg(\Big(\frac{\epsilon}{|\ln\epsilon|^2}\Big)^{\frac{n}{n-2}}\bigg).
\end{align*}
For any $l$, by the mean value theorem,  there exists
$t = t(x)\in[0, 1] $ such that
\begin{align*}
 & \int_{\mathcal{A}_l}\Big|V^{p-1}
  - \sum\limits_{i=1}^k(-1)^i(PU_{ \mu_i,\xi_i})^{p-1} \Big|^{\frac{n}{2}}dx\\
  = & \int_{\mathcal{A}_l}\Big|\Big((-1)^l PU_{ \mu_l,\xi_l}+\sum\limits_{i\neq l}^k (-1)^iPU_{ \mu_i,\xi_i}\Big)^{p-1}
  - (-1)^l(PU_{ \mu_l,\xi_l})^{p-1}-\sum\limits_{i\neq l}^k(-1)^i(PU_{ \mu_i,\xi_i})^{p-1} \Big|^{\frac{n}{2}}dx  \\
 \leq &  C\int_{\mathcal{A}_l}\Big|\Big((-1)^l PU_{ \mu_l,\xi_l}+t\sum\limits_{i\neq l}^k (-1)^iPU_{ \mu_i,\xi_i}\Big)^{p-2}\sum\limits_{i\neq l}^k (-1)^iPU_{ \mu_i,\xi_i}\Big|^{\frac{n}{2}}dx
 +C\sum\limits_{i\neq l}^k\int_{\mathcal{A}_l} |PU_{ \mu_i,\xi_i}|^{(p-1)\frac{n}{2} } dx  \\
 \leq &  C\int_{\mathcal{A}_l}\Big|(-1)^{l+i}( PU_{ \mu_l,\xi_l})^{p-2}\sum\limits_{i\neq l}^k PU_{ \mu_i,\xi_i}\Big|^{\frac{n}{2}}dx
 +C\sum\limits_{i\neq l}^k\int_{\mathcal{A}_l} |PU_{ \mu_i,\xi_i}|^{(p-1)\frac{n}{2} } dx  \\
 \leq &  C\sum\limits_{i\neq l}^k\int_{\mathcal{A}_l}\Big|U_{ \mu_l,\xi_l}^{p-2} U_{ \mu_i,\xi_i}\Big|^{\frac{n}{2}}dx
 +C\sum\limits_{i\neq l}^k\int_{\mathcal{A}_l} |U_{ \mu_i,\xi_i}|^{(p-1)\frac{n}{2} } dx.
\end{align*}
If $i\neq l$, by   (\ref{radia}), let $x-\xi=\mu_i y$, then
\begin{align*}
   \int_{\mathcal{A}_l} |U_{ \mu_i,\xi_i}|^{(p-1)\frac{n}{2} } dx
  \leq & C \int_{\mathcal{A}_l}
\bigg( \frac{ \mu_i^{\frac{n-2}{2}}}{( \mu_i^2 +|x-\xi_i|^2)^{\frac{n-2}{2}}}\bigg)^{(p-1)\frac{n}{2} }dx \\
    =  &  C\mu_i^{n-\frac{n-2}{2}(p-1)\frac{n}{2}} \int_{\frac{\mathcal{A}_l}{\mu_i}}
 \frac{1}{( 1+|y-\sigma_i|^2)^n}dy
 =O\bigg(\Big(\frac{\epsilon}{|\ln\epsilon|^2}\Big)
 ^{\frac{n}{n-2}}\bigg).
\end{align*}
If $n>6$, by  (\ref{radia}),   one has
\begin{eqnarray*}
 && \int_{\mathcal{A}_l}\Big|U_{ \mu_l,\xi_l}^{p-2} U_{ \mu_i,\xi_i}\Big|^{\frac{n}{2}}dx\\
  & \leq & C \int_{\mathcal{A}_l}
\bigg( \frac{ \mu_l^{\frac{-n+6}{2}}}{( \mu_l^2 +|x-\xi_l|^2)^{\frac{-n+6}{2}}}\bigg)^{\frac{ n}{ 2}}
\bigg( \frac{ \mu_i^{\frac{n-2}{2}}}{( \mu_i^2 +|x-\xi_i|^2)^{\frac{n-2}{2}}}\bigg)^{\frac{ n}{ 2}}
dx \\
  &  =  &  C\mu_i^{n-\frac{n-2}{2}\frac{ n}{ 2}} \mu_l^{\frac{-n+6}{2}\frac{ n}{ 2}} \int_{\frac{\mathcal{A}_l}{\mu_i}}
 \frac{1}{( \mu_l^2+|\mu_iy-\mu_l\sigma_l|^2)^{\frac{-n+6}{2}\frac{ n}{ 2}}}\frac{1}{( 1+|y-\sigma_i|^2)^{\frac{n-2}{2}\frac{ n}{ 2}}}dy\\
 &  = &
\left\{ \arraycolsep=1.5pt
   \begin{array}{lll}
 O( \mu_i^{n-\frac{n-2}{2}\frac{ n}{ 2}+(n-6)\frac{ n}{ 2}} \mu_l^{\frac{-n+6}{2}\frac{ n}{ 2}}   )\int_{\frac{\mathcal{A}_l}{\mu_i}}
 \frac{1}{ | y- \frac{\mu_l}{\mu_i}\sigma_l|^{\frac{n(-n+6)}{2}}}
 \frac{1}{( 1+|y-\sigma_i|^2)^{\frac{n-2}{2}\frac{ n}{ 2}}}dy\ \   &{\rm if }\   l>i, \\[2mm]
 O( \mu_i^{n-\frac{n-2}{2}\frac{ n}{ 2}} \mu_l^{-\frac{-n+6}{2}\frac{ n}{ 2}} )\int_{\frac{\mathcal{A}_l}{\mu_i}}
 \frac{1}{( 1+|y-\sigma_i|^2)^{\frac{n-2}{2}\frac{ n}{ 2}}}dy\ \  & {\rm if}\   l<i,
\end{array}
\right.\\
&  = &
\left\{ \arraycolsep=1.5pt
   \begin{array}{lll}
 O( \mu_i^{n-\frac{n-2}{2}\frac{n}{2} } (\frac{\mu_l}{\mu_i})^{\frac{-n+6}{2}\frac{n}{2}} (\frac{\epsilon}{|\ln\epsilon|^2})^{\frac{n}{n-2} } \Big)   \ \   &{\rm if }\   l>i, \\[2mm]
 O\Big( (\frac{\mu_i}{\mu_l})^{\frac{-n+6}{2}\frac{n}{2}}
  (\frac{\epsilon}{|\ln\epsilon|^2})^{\frac{n}{n-2} }\Big)  \ \  & {\rm if}\   l<i,
\end{array}
\right.\\
&  = &
 O\bigg(\Big(\frac{\epsilon}{|\ln\epsilon|^2}\Big)^{\frac{  -n+6 }{n-2}\frac{ n}{2}+\frac{ n}{n-2} }\bigg).
\end{eqnarray*}
If $n<6$, it holds
\begin{eqnarray*}
 && \int_{\mathcal{A}_l}\Big|U_{ \mu_l,\xi_l}^{p-2} U_{ \mu_i,\xi_i}\Big|^{\frac{n}{2}}dx\\
  & \leq & C \int_{\mathcal{A}_l}
\bigg( \frac{ \mu_l^{\frac{-n+6}{2}}}{( \mu_l^2 +|x-\xi_l|^2)^{\frac{-n+6}{2}}}\bigg)^{\frac{ n}{ 2}}
\bigg( \frac{ \mu_i^{\frac{n-2}{2}}}{( \mu_i^2 +|x-\xi_i|^2)^{\frac{n-2}{2}}}\bigg)^{\frac{ n}{ 2}}
dx \\
  &  =  &  C\mu_l^{n-\frac{-n+6}{2}\frac{ n}{ 2}}
  \mu_i^{\frac{n-2}{2}\frac{ n}{ 2}} \int_{\frac{\mathcal{A}_l}{\mu_l}}
 \frac{1}{( 1+|y-\sigma_l|^2)^{\frac{-n+6}{2} \frac{ n}{ 2}}}
 \frac{1}{( \mu_i^2+|\mu_ly-\mu_i\sigma_i|^2)^{\frac{n-2}{2}\frac{ n}{ 2}}}dy\\
 &  = &
\left\{ \arraycolsep=1.5pt
   \begin{array}{lll}
 O(\mu_l^{n-\frac{-n+6}{2}\frac{ n}{ 2}}
  \mu_i^{-\frac{n-2}{2}\frac{ n}{ 2}} )\int_{\frac{\mathcal{A}_l}{\mu_l}}
 \frac{1}{( 1+|y-\sigma_i|^2)^{\frac{-n+6}{2} \frac{ n}{ 2}}}dy \ \   &{\rm if }\   l>i, \\[2mm]
 O( \mu_l^{n-\frac{-n+6}{2}\frac{ n}{ 2}- (n-2)\frac{ n }{ 2}}
  \mu_i^{\frac{n-2}{2}\frac{ n}{ 2}} )\int_{\frac{\mathcal{A}_l}{\mu_l}}
 \frac{1}{ | y- \frac{\mu_i}{\mu_l}\sigma_l|^{\frac{n(n-2)}{2}}}
 \frac{1}{( 1+|y-\sigma_l|^2)^{\frac{-n+6}{2} \frac{ n}{ 2}}}dy\ \  & {\rm if}\   l<i,
\end{array}
\right.\\
&  = &
\left\{ \arraycolsep=1.5pt
   \begin{array}{lll}
 O( (\frac{\mu_l}{\mu_i})^{\frac{n-2}{2} \frac{ n}{ 2}} )   \ \   &{\rm if }\   l>i, \\[2mm]
  O( (\frac{\mu_i}{\mu_l})^{\frac{n-2}{2} \frac{ n}{ 2}} )  \ \  & {\rm if}\   l<i,
\end{array}
\right.\\
&  = &
 O\bigg( \Big (\frac{\epsilon}{|\ln\epsilon|^2}\Big)^{\frac{n}{2} }\bigg).
\end{eqnarray*}
A similar estimate can be obtained for $n=6$, we prove that
\[
\int_{\mathcal{A}_l}\Big|U_{ \mu_l,\xi_l}^{p-2} U_{ \mu_i,\xi_i}\Big|^3dx=
O\bigg( \Big (\frac{\epsilon}{|\ln\epsilon|^2}\Big)^3
\Big|\ln\frac{\epsilon}{|\ln\epsilon|^2}\Big|^3\bigg).
\]
Thus,  (\ref{sumbu2}) holds.

Let us now estimate (\ref{fepli1}), one has
\begin{align}\label{uoa}
&\Big|f_\epsilon(V)-\sum_{i=1}^k(-1)^if_0(PU_{ \mu_i,\xi_i}) \Big|_{\frac{2n}{n+2}}\nonumber\\
 = & \Big|f_\epsilon(V)-f_0(V)\Big|_{\frac{2n}{n+2}}
 +
\Big|f_0(V)-
\sum_{i=1}^k(-1)^if_0(PU_{ \mu_i,\xi_i}) \Big|_{\frac{2n}{n+2}}.
\end{align}
Similar to the proof of (\ref{sumbu2}),
we obtain
\begin{eqnarray}\label{cfla}
 \Big|  f_0(V)-\sum\limits_{i=1}^k(-1)^if_0(PU_{ \mu_i,\xi_i}) \Big|_{\frac{2n}{n+2}}
   =
\left\{ \arraycolsep=1.5pt
   \begin{array}{lll}
  O\Big(\frac{\epsilon}{|\ln\epsilon|^2}\Big)\ \   &{\rm if }\   3\leq n\leq 5 , \\[2mm]
 O\bigg(\frac{\epsilon}{|\ln\epsilon|^2}
 \Big|\ln\frac{\epsilon}{|\ln\epsilon|^2}\Big|\bigg)\ \   &{\rm if }\  n=6, \\[2mm]
 O\bigg( \Big ( \frac{\epsilon}{|\ln\epsilon|^2}\Big)^{\frac{n+2}{2(n-2)} }\bigg)\ \  & {\rm if}\   n\geq 7.
\end{array}
\right.
\end{eqnarray}
On the other hand, by  Lemma \ref{zsj}, there holds
\begin{align}\label{onqan}
  \int_\Omega \Big|f_\epsilon(V)-f_0(V)\Big|^{\frac{2n}{n+2}}dx
  \leq  &  \epsilon \int_\Omega |V^p\ln\ln(e+V)|^{\frac{2n}{n+2}}dx \nonumber\\
  \leq  & \epsilon \int_\Omega\bigg|\Big(\sum\limits_{i=1}^k (-1)^iU_{ \mu_i,\xi_i}\Big)^p
    \ln\ln \Big(e+ \sum\limits_{i=1}^k (-1)^iU_{ \mu_i,\xi_i} \Big)\bigg|^{\frac{2n}{n+2}}dx \nonumber\\
  \leq  &
  \epsilon \int_{\Omega\setminus B(\xi, \rho)} \bigg| \Big(\sum\limits_{i=1}^k (-1)^iU_{ \mu_i,\xi_i}\Big)^p
  \ln\ln \Big(e+\sum\limits_{i=1}^k (-1)^iU_{ \mu_i,\xi_i}\Big)\bigg|^{\frac{2n}{n+2}}dx \nonumber\\
  & + \epsilon \sum_{l=1}^k\int_{\mathcal{A}_l} \bigg| \Big(\sum\limits_{i=1}^k (-1)^iU_{ \mu_i,\xi_i}\Big)^p
  \ln\ln \Big(e+\sum\limits_{i=1}^k (-1)^iU_{ \mu_i,\xi_i}\Big)\bigg|^{\frac{2n}{n+2}}dx.
\end{align}
We now observe that
\begin{align*}
 &  \int_{\Omega\setminus B(\xi, \rho)} \bigg| \Big(\sum\limits_{i=1}^k (-1)^iU_{ \mu_i,\xi_i}\Big)^p
  \ln\ln \Big(e+\sum\limits_{i=1}^k (-1)^iU_{ \mu_i,\xi_i}\Big)\bigg|^{\frac{2n}{n+2}}dx\\
  \leq   &  C \sum\limits_{i=1}^k \int_{\Omega\setminus B(\xi, \rho)}  \bigg| U_{ \mu_i,\xi_i}^p
  \ln\ln \Big(e+\sum\limits_{i=1}^k (-1)^iU_{ \mu_i,\xi_i}\Big)\bigg|^{\frac{2n}{n+2}}dx\\
  \leq &  C\sum\limits_{i=1}^k\mu_i^n
    \bigg|
  \ln\ln \Big(e+\sum\limits_{i=1}^k (-1)^i\mu_i^{\frac{n-2}{2}} \Big)\bigg|^{\frac{2n}{n+2}}dy\\
  \leq &  C\Big(\frac{\epsilon}{|\ln\epsilon|^2}\Big)^{\frac{n}{n-2}}
  \Big(\ln\Big|\ln\frac{\epsilon}{|\ln\epsilon|^2}\Big| \Big)^{\frac{2n}{n+2}}.
\end{align*}
For the second integral in (\ref{onqan}),
from (\ref{annulus}),
and let $x-\xi=\mu_l y$,
then
\begin{align}\label{ksu}
& \int_{\mathcal{A}_l} \bigg|\Big (\sum\limits_{i=1}^k (-1)^iU_{ \mu_i,\xi_i}\Big)^p
  \ln\ln \Big(e+\sum\limits_{i=1}^k (-1)^iU_{ \mu_i,\xi_i}\Big)\bigg|^{\frac{2n}{n+2}}dx\nonumber\\
 = & \int_{\mathcal{A}_l} \bigg| \Big((-1)^l U_{ \mu_l,\xi_l}+\sum_{i\neq l}^k(-1)^iU_{ \mu_i,\xi_i}\Big)^p
  \ln\ln \Big(e+(-1)^lU_{ \mu_l,\xi_l}+\sum_{i\neq l}^k(-1)^iU_{ \mu_i,\xi_i}\Big)\bigg|^{\frac{2n}{n+2}}dx\nonumber\\
 = &  \int_{\mathcal{A}_l} \bigg|\bigg[(-1)^lU_{ \mu_l,\xi_l}^p+ O\Big(\sum_{i\neq l}^k(-1)^iU_{ \mu_i,\xi_i}\Big)    \bigg]\ln\ln \Big(e+(-1)^lU_{ \mu_l,\xi_l}+\sum_{i\neq l}^k(-1)^iU_{ \mu_i,\xi_i}\Big)\bigg|^{\frac{2n}{n+2}}dx\nonumber\\
 = &  \int_{\mathcal{A}_l} \bigg|(-1)^lU_{ \mu_l,\xi_l}^p\ln\ln \Big(e+(-1)^lU_{ \mu_l,\xi_l}+\sum_{i\neq l}^k(-1)^iU_{ \mu_i,\xi_i}\Big)\bigg|^{\frac{2n}{n+2}}dx\nonumber\\
 & + C\sum_{i\neq l}^k\int_{\mathcal{A}_l}U_{ \mu_i,\xi_i} \bigg|\ln\ln \Big(e+(-1)^lU_{ \mu_l,\xi_l}+\sum_{i\neq l}^k(-1)^iU_{ \mu_i,\xi_i}\Big)\bigg|^{\frac{2n}{n+2}}dx.
\end{align}
For $i>l$,  by Lemma \ref{zsj}, we have
\begin{align*}
&   \int_{\mathcal{A}_l} \bigg|U_{ \mu_l,\xi_l}^p\ln\ln \Big(e+(-1)^lU_{ \mu_l,\xi_l}+\sum_{i\neq l}^k(-1)^iU_{ \mu_i,\xi_i}\Big)\bigg|^{\frac{2n}{n+2}}dx\\
= & \int_{\mathcal{A}_l}
\frac{\alpha_n^{\frac{2n}{n-2}}   \mu_l^n}{( \mu_l^2 +|x-\xi_l|^2)^n}
  \bigg| \ln\ln \Big(e+(-1)^l\frac{\alpha_n  \mu_l^{\frac{n-2}{2}}}{( \mu_l^2 +|x-\xi_l|^2)^{\frac{n-2}{2}}}
  +\sum_{i\neq l}^k(-1)^i \frac{\alpha_n\mu_i^{\frac{n-2}{2}}}{( \mu_i^2 +|x-\xi_i|^2)^{\frac{n-2}{2}}}  \Big)\bigg|^{\frac{2n}{n+2}}dx\\
 = & \int_{\frac{\mathcal{A}_l}{\mu_l}}
 \frac{\alpha_n^{\frac{2n}{n-2}}  }{( 1+|y-\sigma_l|^2)^n}
  \bigg| \ln\ln \Big(e+(-1)^l\mu_l^{-\frac{n-2}{2}}\frac{\alpha_n} {( 1+|y-\sigma_l|^2)^{\frac{n-2}{2}}}
  +o(\epsilon)\Big)\bigg|^{\frac{2n}{n+2}} dy\\
 = & \int_{\frac{\mathcal{A}_l}{\mu_l}}
 \frac{\alpha_n^{\frac{2n}{n-2}}  }{( 1+|y-\sigma_l|^2)^n}
  \bigg| \ln\ln \mu_l^{-\frac{n-2}{2}} +   \ln \Big[1+\frac{\ln\Big(e^{1-\frac{n-2}{2}|\ln \mu_l|  }+ \frac{\alpha_n} {1+|y-\sigma_l|^2)^{\frac{n-2}{2}}}\Big)}{\frac{n-2}{2}|\ln \mu_l|  }
\Big]\bigg|^{\frac{2n}{n+2}}dy +o(\epsilon)\\
= & \int_{\frac{\mathcal{A}_l}{\mu_l}}
 \frac{\alpha_n^{\frac{2n}{n-2}}  }{( 1+|y-\sigma_l|^2)^n}
  \bigg| \ln\ln \mu_l^{-\frac{n-2}{2}} +  \frac{1}{ |\ln \mu_l|}\frac{2}{n-2 } \ln \frac{\alpha_n}{( 1+|y-\sigma_l|^2)^{\frac{n-2}{2}}}
\bigg|^{\frac{2n}{n+2}}dy +o(\epsilon)\\
\leq & C (\ln|\ln\mu_1| )^{\frac{2n}{n+2}}
 =O \bigg(\Big(\ln\Big|\ln\frac{\epsilon}{|\ln\epsilon|^2}\Big| \Big)^{\frac{2n}{n+2}}\bigg).
\end{align*}
In a same way, the estimate of second term in (\ref{ksu}) is
\begin{align*}
& \sum_{i\neq l}^k\int_{\mathcal{A}_l}U_{ \mu_i,\xi_i} \bigg|\ln\ln \Big(e+(-1)^lU_{ \mu_l,\xi_l}+\sum_{i\neq l}^k(-1)^iU_{ \mu_i,\xi_i}\Big)\bigg|^{\frac{2n}{n+2}}dx
=O\Big(\frac{\epsilon}{|\ln\epsilon|^2}\Big).
\end{align*}
Thus,
\begin{equation}\label{onran}
\Big|f_\epsilon(V)-f_0(V)\Big|_{\frac{2n}{n+2}}
=O \bigg(\epsilon\ln\Big|\ln\frac{\epsilon}{|\ln\epsilon|^2}\Big|\bigg).
\end{equation}
Combining (\ref{uoa}), (\ref{cfla}) and (\ref{onran}), we obtain  (\ref{fepli1}).
Similar to the proof of (\ref{onqan}), the estimate  (\ref{fepli2}) holds.
\end{proof}

\begin{lemma}\label{pf1}
There holds
\begin{eqnarray*}
P_3  & =  &
  \int_\Omega \Big[\sum\limits_{i=1}^k(-1)^if_0(PU_{ \mu_i,\xi_i})-f_\epsilon(V) \Big] \psi^h_{ \mu_j,\xi_j}dx\\
& =& \left\{ \arraycolsep=1.5pt
   \begin{array}{lll}
  \alpha_n a_1  \frac{\epsilon}{|\ln\epsilon|^2}
  H(\xi,\xi)d_1^{n-2}
  +a_3\frac{\epsilon}{|\ln\epsilon|^2}\sum\limits_{i=1}^{k-1}\Big(\frac{d_{i+1}}{ d_i}\Big)^{\frac{n-2}{2}} g(\sigma_i)
  -\frac{2k^2}{(n-2)^2} a_4\epsilon\Big|\ln \frac{\epsilon}{|\ln\epsilon|^2}\Big|\\
\quad   -a_4 \frac{\epsilon}{|\ln\epsilon|^2}\sum\limits_{i=1}^k \frac{2}{2i-1}|\ln  d_i|
+o\Big(\frac{\epsilon}{|\ln\epsilon|^2}\Big)\ \   \   {\rm if}\  h=0, \\[2mm]
   \frac{1}{2}\alpha_n a_2  \Big(\frac{\epsilon}{|\ln\epsilon|^2}\Big)^{\frac{n-1}{n-2}} d_1^{n-1} \partial_{\xi_h}\varphi(\xi)
   + o\bigg(\Big(\frac{\epsilon}{|\ln\epsilon|^2}\Big)^{\frac{n-1}{n-2}}\bigg)\ \  \    {\rm if}\  h=1,\cdots,n,\\[2mm]
\end{array}
\right.
\end{eqnarray*}
where $a_1$-$a_4$ and $g(\sigma_i)$ are given  in  Proposition \ref{leftside} for $i=1,\cdots,k$.
\end{lemma}

\begin{proof}
We have
\begin{align*}
  P_3
  = & \int_\Omega \Big[\sum_{i=1}^k(-1)^if_0(PU_{ \mu_i,\xi_i})-f_\epsilon(V) \Big] \psi^h_{ \mu_j,\xi_j}dx\\
   = &  \int_\Omega \Big[\sum_{i=1}^k(-1)^if_0(U_{ \mu_i,\xi_i})-f_0(V) \Big] \psi^h_{ \mu_j,\xi_j}dx\\
 &  + \int_\Omega \Big[\sum_{i=1}^k(-1)^if_0(PU_{ \mu_i,\xi_i})-\sum_{i=1}^k (-1)^if_0(U_{ \mu_i,\xi_i}) \Big] \psi^h_{ \mu_j,\xi_j}dx\\
  & + \int_\Omega \Big[f_0(V)-f_\epsilon(V) \Big] \psi^h_{ \mu_j,\xi_j}dx
  = J_1+J_2+J_3.
\end{align*}

\emph{Estimate of $J_1$}:
One has
\begin{align}\label{iak5}
  J_1
  = & \int_\Omega \Big[  \sum_{i=1}^k(-1)^if_0(U_{ \mu_i,\xi_i})-f_0(V) \Big] \psi^h_{ \mu_j,\xi_j}dx\nonumber\\
  = & \sum_{i=1}^k \int_\Omega(-1)^i f_0(U_{ \mu_i,\xi_i})\psi^h_{ \mu_j,\xi_j}dx
  -\int_\Omega f_0(V)  \psi^h_{ \mu_j,\xi_j}dx\nonumber\\
  = & \sum_{i=1}^k (-1)^i\int_\Omega f_0(U_{ \mu_i,\xi_i})P\psi^h_{ \mu_j,\xi_j}dx
  + \sum_{i=1}^k (-1)^i \int_\Omega f_0(U_{ \mu_i,\xi_i})(\psi^h_{ \mu_j,\xi_j}-P\psi^h_{ \mu_j,\xi_j})dx \nonumber\\
 & -\int_\Omega f_0(V)  \psi^h_{ \mu_j,\xi_j}dx.
\end{align}

For the first term in (\ref{iak5}),  if  $h=0$ and $i\neq j$,
from Lemma \ref{expsi},
  we deduce
\begin{align}\label{iak1}
& \int_\Omega f_0(U_{ \mu_i,\xi_i})P\psi^0_{ \mu_j,\xi_j}dx\nonumber\\
= & \frac{n-2}{2}a_2\alpha_n\mu_j^{\frac{n-2}{2}}\int_\Omega \frac{\mu_i^{\frac{n+2}{2}}}{(\mu_i^2+|x-\xi_i|^2)^{\frac{n+2}{2}}}G(x,\xi_j)dx
  +o\Big(\frac{\epsilon}{|\ln\epsilon|^2}\Big)\nonumber\\\
= & \frac{n-2}{2}a_2\alpha_n\mu_i^{ \frac{n-2}{2}}\mu_j^{\frac{n-2}{2}}\int_{\frac{\Omega -\xi_i}{\mu_i}} \frac{1}{(1+|y|^2)^{\frac{n+2}{2}}}  G(\mu_i y+\xi_i,\xi_j)dy
  +o\Big(\frac{\epsilon}{|\ln\epsilon|^2}\Big) \nonumber\\
 = &\frac{n-2}{2}
  a_2^2\mu_j^{\frac{n-2}{2}}\mu_i^{ \frac{n-2}{2}} G(\xi_i,\xi_j)+
  o\Big(\frac{\epsilon}{|\ln\epsilon|^2}\Big).
\end{align}
Let $h=0$ and $i=j$, from (\ref{eq2}) and (\ref{xy}), there holds
 \begin{align}\label{iak1u}
  \int_\Omega f_0(U_{ \mu_i,\xi_i})P\psi^0_{ \mu_i,\xi_i}dx
    = & \int_\Omega U_{ \mu_i,\xi_i}^p\psi^0_{ \mu_i,\xi_i}dx
   -\frac{n-2}{2}\alpha_n \mu_i^{\frac{n-2}{2}}\int_\Omega U_{ \mu_i,\xi_i}^p\Big(H(x,\xi_i)+O( \mu_i^{\frac{n+4}{2}})\Big)dx\nonumber\\
     = & \mu_i^{n-\frac{n+2}{2}+1+\frac{n-2}{2}-(n-2)}
  \int_{\frac{\Omega-\xi_i}{\mu_i}} U^p(y)\Big(\frac{\partial U}{\partial \mu}\Big)_{|_{\mu=1}}dy\nonumber\\
  & -\frac{n-2}{2} \mu_i^{n-2} H(\xi_i,\xi_i)\int_{\frac{\Omega-\xi_i}{\mu_i}} U^p(y) dy+O( \mu_i^{\frac{3n}{2}})\nonumber\\
   = & -\frac{n-2}{2}  a_2\mu_i^{n-2}H(\xi_i,\xi_i)  +O\Big(\frac{\epsilon}{|\ln\epsilon|^2}\Big).
\end{align}
If $h=1,\cdots,n$ and $i\neq j$, from Lemma \ref{expsi}, we get
\begin{align}\label{iak2}
   \int_\Omega f_0(U_{ \mu_i,\xi_i})P\psi^h_{ \mu_j,\xi_j}dx
   = &  \int_\Omega  U_{ \mu_i,\xi_i}^{2^*-1}P\psi^h_{ \mu_j,\xi_j}dx\nonumber\\
  = & a_2\mu_j^{\frac{n}{2}}\mu_i^{ \frac{n-2}{2}}\int_{\frac{\Omega -\xi_i}{\mu_i}} \frac{1}{(1+|y|^2)^{\frac{n+2}{2}}}\frac{\partial G}{\partial \xi_j^h}(\mu_i y+\xi_i,\xi_j)dy
  +o\bigg(\Big(\frac{\epsilon}{|\ln\epsilon|^2}\Big)^{\frac{n-1}{n-2}}\bigg)\nonumber\\
  = &
  a_2^2\mu_j^{\frac{n}{2}}\mu_i^{ \frac{n-2}{2}}\frac{\partial G}{\partial \xi_j^h}(\xi_i,\xi_j)
  +o\bigg(\Big(\frac{\epsilon}{|\ln\epsilon|^2}\Big)^{\frac{n-1}{n-2}}\bigg).
\end{align}
If $h=1,\cdots,n$ and $i=j$, from (\ref{eq3}) and (\ref{xy}),  we have
\begin{align}\label{iak2d}
   \int_\Omega f_0(U_{ \mu_i,\xi_i})P\psi^h_{ \mu_i,\xi_i}dx
   = &  \int_\Omega  U_{ \mu_i,\xi_i}^{2^*-1}P\psi^h_{ \mu_i,\xi_i}dx\nonumber\\
   = &  \int_\Omega  U_{ \mu_i,\xi_i}^{2^*-1} \psi^h_{ \mu,\xi_i}(x)dx
    -\alpha_n \mu_i^{\frac{n }{2}}
   \int_\Omega  U_{ \mu_i,\xi_i}^{2^*-1}\partial_{\xi^h_i}H(x,\xi_i)dx
   +o\bigg(\Big(\frac{\epsilon}{|\ln\epsilon|^2}\Big)^{\frac{n-1}{n-2}}\bigg)\nonumber\\
   = & \mu_i^{n-\frac{n+2}{2}+1-\frac{n-2}{2}} \int_{\frac{\Omega-\xi_i}{\mu_i}}  U^{2^*-1}(y) \psi^h (y)dy\nonumber\\
  &  -\alpha_n \mu_i^{\frac{n }{2}-\frac{n+2 }{2}+n}
   \partial_{\xi^h_i}H(\xi_i,\xi_i)\int_{\frac{\Omega-\xi_i}{\mu_i}}  U^{2^*-1}(y)dy
   +o\bigg(\Big(\frac{\epsilon}{|\ln\epsilon|^2}\Big)^{\frac{n-1}{n-2}}\bigg)\nonumber\\
   = & \mu_i \int_{\frac{\Omega-\xi_i}{\mu_i}}  U^{2^*-1}(y) \psi^h (y)dy\nonumber\\
  &  -\alpha_n \mu_i^{ n-1}
   \partial_{\xi^h_i}H(\xi_i,\xi_i)\int_{\frac{\Omega-\xi_i}{\mu_i}}  U^{2^*-1}(y)dy
   +o\bigg(\Big(\frac{\epsilon}{|\ln\epsilon|^2}\Big)^{\frac{n-1}{n-2}}\bigg)\nonumber\\
  = &   -\alpha_n
  a_2 \mu_i^{n-1} \frac{\partial H}{\partial \xi_i^h}(\xi_i,\xi_i)
  +o\bigg(\Big(\frac{\epsilon}{|\ln\epsilon|^2}\Big)^{\frac{n-1}{n-2}}\bigg).
\end{align}

It remains to estimate the second term in (\ref{iak5}),  if
   $h=0$, by (\ref{eq2}), it holds
\begin{eqnarray}\label{iak3}
 && \int_\Omega f_0(U_{ \mu_i,\xi_i})(\psi^0_{ \mu_j,\xi_j}-P\psi^0_{ \mu_j,\xi_j})dx\nonumber\\
  & = &  \frac{n-2}{2}\alpha_n \mu_i^{\frac{n-2}{2}}\mu_j^{\frac{n-2}{2}}\int_{\frac{\Omega -\xi_i}{\mu_i}} \frac{1}{(1+|y|^2)^{\frac{n+2}{2}}}
  H(\mu_i y+\xi_i,\xi_j)dy+o( \mu_i^{\frac{n-2}{2}}\mu_j^{\frac{n+4}{2}})
  \nonumber\\
  & = & \left\{ \arraycolsep=1.5pt
   \begin{array}{lll}
  \frac{n-2}{2}a_2 \mu_i^{\frac{n-2}{2}}\mu_j^{\frac{n-2}{2}} H(\xi_i,\xi_j)
  + o\bigg(\Big(\frac{\epsilon}{|\ln\epsilon|^2}\Big)^{\frac{n-1}{n-2}}\bigg) \ \   &{\rm if}\ i>j, \\[2mm]
  \frac{n-2}{2}a_2 \mu_i^{n-2} H(\xi_i,\xi_i)
  +o\bigg(\Big(\frac{\epsilon}{|\ln\epsilon|^2}\Big)^{\frac{n-1}{n-2}}\bigg)\ \   &{\rm if}\ i=j.
\end{array}
\right.
\end{eqnarray}
If $h=1,\cdots,n$,  by (\ref{eq3}), we obtain
\begin{eqnarray}\label{iak4}
 &&  \int_\Omega f_0(U_{ \mu_i,\xi_i})(\psi^h_{ \mu_j,\xi_j}-P\psi^h_{ \mu_j,\xi_j})dx\nonumber\\
  &= & \alpha_n  \mu_i^{ \frac{n-2 }{2}}\mu_j^{\frac{n}{2}}\int_{\frac{\Omega -\xi_i}{\mu_i}} \frac{1}{(1+|y|^2)^{\frac{n+2}{2}}}\partial_{\xi_j^h}H(\mu_i y+\xi_i,\xi_j)dy
  +o(\mu_i^{\frac{n-2}{2}}\mu_j^{\frac{n+2}{2}}) \nonumber\\
  & = & \left\{ \arraycolsep=1.5pt
   \begin{array}{lll}
  a_2 \mu_i^{\frac{n-2}{2}}\mu_j^{ \frac{n}{2}} \partial_{\xi_j^h}H(\xi_i,\xi_j)
  + o\bigg(\Big(\frac{\epsilon}{|\ln\epsilon|^2}\Big)^{\frac{n-1}{n-2}}\bigg)  \ \   &{\rm if}\ i> j, \\[2mm]
  a_2 \mu_i^{n-1} \partial_{\xi_i^h}H(\xi_i,\xi_i)
  + o\bigg(\Big(\frac{\epsilon}{|\ln\epsilon|^2}\Big)^{\frac{n-1}{n-2}}\bigg) \ \   &{\rm if}\ i=j.
\end{array}
\right.
\end{eqnarray}

For the last integral in (\ref{iak5}),
\begin{align}\label{ladp}
  \int_\Omega f_0(V)  \psi^h_{ \mu_j,\xi_j}dx
   & =  \int_{\Omega\setminus B(\xi,\rho)} f_0(V)  \psi^h_{ \mu_j,\xi_j}dx
   +\sum_{l=1}^k\int_{\mathcal{A}_l} f_0(V)  \psi^h_{ \mu_j,\xi_j}dx.
\end{align}
If $h=0$, by (\ref{psi}),  a direct computation shows that
\begin{align*}
   \int_{\Omega\setminus B(\xi,\rho)}\Big| f_0(V)  \psi^0_{ \mu_j,\xi_j}\Big|dx
   \leq & C \frac{n-2}{2} \alpha_n\mu_j^{\frac{n-2}{2}} \sum_{i=1}^k\int_{\Omega\setminus B(\xi,\rho)}\bigg| U_{ \mu_i,\xi_i}^p \frac{|x-\xi_j|^2-\mu_j^2}{(\mu_j^2+|x-\xi_j|^2)^{\frac{n}{2}}}\bigg|dx\\
 \leq & C \frac{n-2}{2} \alpha_n\mu_j^{\frac{n-2}{2}} \sum_{i=1}^k\int_{\Omega\setminus B(\xi,\rho)}
 \bigg|\frac{ \mu_i^{\frac{n+2}{2}}}{(\mu_i^2+|x-\xi_i|^2)^{\frac{n+2}{2}}}
 \frac{|x-\xi_j|^2-\mu_j^2}{(\mu_j^2+|x-\xi_j|^2)^{\frac{n}{2}}}\bigg|dx\\
  = &
   o\bigg(\Big(\frac{\epsilon}{|\ln\epsilon|^2}\Big)^{\frac{n-1}{n-2}}\bigg).
\end{align*}
If $h=1,\cdots,n$, we obtain
\begin{align*}
   \int_{\Omega\setminus B(\xi,\rho)} \Big|f_0(V)  \psi^0_{ \mu_j,\xi_j}\Big|dx
  \leq & C  (n-2)  \alpha_n  \mu_j^{\frac{n}{2}}
 \sum_{i=1}^k\int_{\Omega\setminus B(\xi,\rho)} \bigg|(-1)^i  U_{ \mu_i,\xi_i}^p \frac{x^h-\xi_j^h}{(\mu_j^2+|x-\xi_j|^2)^{\frac{n}{2}}}\bigg|dx\\
 \leq & C \frac{n-2}{2} \alpha_n\mu_j^{\frac{n}{2}} \sum_{i=1}^k\int_{\Omega\setminus B(\xi,\rho)}
 \bigg|\frac{ \mu_i^{\frac{n+2}{2}}}{(\mu_i^2+|x-\xi_i|^2)^{\frac{n+2}{2}}}
 \frac{x^h-\xi_j^h}{(\mu_j^2+|x-\xi_j|^2)^{\frac{n}{2}}}\bigg|dx\\
  = & o\bigg(\Big(\frac{\epsilon}{|\ln\epsilon|^2}\Big)^{\frac{n-1}{n-2}}\bigg).
\end{align*}
On the other hand, we estimate
 the second term in (\ref{ladp}),
\begin{align}\label{ora}
& \int_{\mathcal{A}_l} \Big|f_0(V)  \psi^h_{ \mu_j,\xi_j}\Big|dx\nonumber\\
= & \int_{\mathcal{A}_l} \Big|\Big ((-1)^lPU_{ \mu_l,\xi_l}+\sum_{i\neq l}^k(-1)^iPU_{ \mu_i,\xi_i}\Big)^p  \psi^h_{ \mu_j,\xi_j}\Big|dx\nonumber\\
= & \int_{\mathcal{A}_l} \Big|(-1)^l (PU_{ \mu_l,\xi_l})^p  \psi^h_{ \mu_j,\xi_j}\Big|dx
+O\bigg(\int_{\mathcal{A}_l}\Big| \Big (\sum_{i\neq l}^k(-1)^iPU_{ \mu_i,\xi_i}\Big)  \psi^h_{ \mu_j,\xi_j}\Big|dx\bigg)\nonumber\\
= & \int_{\mathcal{A}_l} \Big|(-1)^l\Big( (PU_{ \mu_l,\xi_l})^p-U_{ \mu_l,\xi_l}^p\Big)  \psi^h_{ \mu_j,\xi_j}\Big|dx
+\int_{\mathcal{A}_l} |(-1)^lU_{ \mu_l,\xi_l}^p  \psi^h_{ \mu_j,\xi_j}|dx\nonumber\\
& +O\bigg(\int_{\mathcal{A}_l} \Big| \Big(\sum_{i\neq l}^k(-1)^iPU_{ \mu_i,\xi_i}\Big) \psi^h_{ \mu_j,\xi_j}\Big|dx\bigg).
\end{align}
On the  fixed  annulus $\mathcal{A}_l$,
by (\ref{eq1}), for $h=0$, let $x-\xi=\mu_ly$,
then
\begin{eqnarray*}
 && \int_{\mathcal{A}_l} \Big|\Big( (PU_{ \mu_l,\xi_l})^p-U_{ \mu_l,\xi_l}^p\Big)  \psi^0_{ \mu_j,\xi_j}\Big|dx\\
 & = & \int_{\mathcal{A}_l}\bigg|\bigg[ \Big(   U_{ \mu_l,\xi_l}-\alpha_n \mu_l^{\frac{n-2}{2}}H(x,\xi_l)+O( \mu_l^{\frac{n+2}{2}})\Big)^p-U_{ \mu_l,\xi_l}^p\bigg]  \psi^0_{ \mu_j,\xi_j}\bigg|dx\\
 & = & \frac{(n-2)\alpha_n}{2} \mu_l^{\frac{n-2}{2}+n}\mu_j^{\frac{n-2}{2}}
O\left(-\alpha_n H (\xi_l,\xi_l)+O(\mu_l^2)\right)
 \int_{\frac{\mathcal{A}_l}{\mu_l}}
\left|\frac{ \Big[|y- \sigma_j\frac{\mu_j}{\mu_l}|^2-(\frac{\mu_j}{\mu_l})^2\Big]\mu_l^2}{\Big((\frac{\mu_j}{\mu_l})^2+| y- \sigma_j \frac{\mu_j}{\mu_l}|^2\Big)^{\frac{n}{2}}\mu_l^n} \right|dy\\
 & = & \left\{ \arraycolsep=1.5pt
  \begin{array}{lll}
  O(\mu_j^{\frac{n+2}{2}}\mu_l^{\frac{n-2}{2}}H (\xi_l,\xi_l))
  \ \   &{\rm if}\ j>l, \\[2mm]
  O(\mu_l^nH (\xi_l,\xi_l))
  \ \   &{\rm if}\ j=l,
\end{array}
\right. \\
  & = &  o\bigg(\Big(\frac{\epsilon}{|\ln\epsilon|^2}\Big)^{\frac{n-1}{n-2}}\bigg).
\end{eqnarray*}
Similarly, for $h=1,\cdots,n$,  we get
\begin{eqnarray*}
 && \int_{\mathcal{A}_l} \Big|\Big( (PU_{ \mu_l,\xi_l})^p-U_{ \mu_l,\xi_l}^p\Big)  \psi^h_{ \mu_j,\xi_j}\Big|dx\\
 & = & \int_{\mathcal{A}_l}\bigg|\bigg[ \Big(   U_{ \mu_l,\xi_l}-\alpha_n \mu_l^{\frac{n-2}{2}}H(x,\xi_l)+O( \mu_l^{\frac{n+2}{2}})\Big)^p-U_{ \mu_l,\xi_l}^p\bigg]  \psi^h_{ \mu_j,\xi_j}\bigg|dx\\
 & = & -(n-2)  \alpha_n  \mu_j^{\frac{n}{2}}
 \int_{\mathcal{A}_l}
  \bigg|O\Big(-\alpha_n \mu_l^{\frac{n-2}{2}}H(x,\xi_l)+O( \mu_l^{\frac{n+2}{2}})\Big) \frac{(x -\xi_j)_h}{(\mu_j^2+|x-\xi_j|^2)^{\frac{n}{2}}}\bigg|dx\\
 & = & \left\{ \arraycolsep=1.5pt
  \begin{array}{lll}
  O\Big(\mu_j^{\frac{n }{2}}\mu_l^{\frac{n}{2}}H (\xi_l,\xi_l)\Big)
  \ \   &{\rm if}\ j>l, \\[2mm]
  O\Big(\mu_l^nH (\xi_l,\xi_l)\Big)
 \ \   &{\rm if}\ j=l,
\end{array}
\right.\\
 & = &    o\bigg(\Big(\frac{\epsilon}{|\ln\epsilon|^2}\Big)^{\frac{n-1}{n-2}}\bigg).
\end{eqnarray*}
The second term in (\ref{ora}).
By (\ref{eq1}), for $h=0$, if $ j>l$,
let $x-\xi=\mu_ly$,  then
\begin{align*}
 & \int_{\mathcal{A}_l} |U_{ \mu_l,\xi_l}^p  \psi^0_{ \mu_j,\xi_j} |dx\\
  = & \frac{n-2}{2}\alpha_n^{p+1}\mu_j^{\frac{n-2}{2}}
 \int_{\frac{\mathcal{A}_l}{\mu_l}}
 \bigg|\frac{\mu_l^{-\frac{n+2}{2}}}{(1+|y-\sigma_l|^2)^{\frac{n+2}{2}}}
  \frac{ \Big(|y- \sigma_j\frac{\mu_j}{\mu_l}|^2-(\frac{\mu_j}{\mu_l})^2\Big)\mu_l^2}{\Big((\frac{\mu_j}{\mu_l})^2+| y- \sigma_j \frac{\mu_j}{\mu_l}|^2\Big)^{\frac{n}{2}}\mu_l^n} \mu_l^n\bigg|dy\\
   = & \frac{n-2}{2}\alpha_n^{p+1} \Big(\frac{\mu_j}{ \mu_l}\Big)^{\frac{n-2}{2}}
  \bigg( \int_{\R^n} \bigg|\frac{  y^{2-n}}{(1+|y-\sigma_l|^2)^{\frac{n+2}{2}}}\bigg|dy
  +o\Big(\frac{\epsilon}{|\ln\epsilon|^2}\Big)\bigg)\\
  = & \frac{n-2}{2}\alpha_n^{p+1} \bigg(\frac{\mu_{l+1}}{ \mu_l}\bigg)^{\frac{n-2}{2}}
  \bigg( \int_{\R^n} \bigg|\frac{  y^{2-n}}{(1+|y-\sigma_l|^2)^{\frac{n+2}{2}}}\bigg|dy
  +  o\Big(\frac{\epsilon}{|\ln\epsilon|^2}\Big)
  \bigg).
\end{align*}
For $h=1,\cdots,n$,  if $ j>l$,  let $x-\xi=\mu_ly$,   there holds
\begin{align*}
\int_{\mathcal{A}_l}  |U_{ \mu_l,\xi_l}^p  \psi^h_{ \mu_j,\xi_j} |dx
  = &  (n-2)  \alpha_n^{p+1}  \mu_j^{\frac{n}{2}}
 \int_{\frac{\mathcal{A}_l}{\mu_l}}
 \bigg|\frac{\mu_l^{-\frac{n+2}{2}}}{(1+|y-\sigma_l|^2)^{\frac{n+2}{2}}}
 \frac{  (y_h- \sigma_j\frac{\mu_j}{\mu_l})\mu_l }{\Big((\frac{\mu_j}{\mu_l})^2+| y- \sigma_j \frac{\mu_j}{\mu_l}|^2\Big)^{\frac{n}{2}}\mu_l^n} \mu_l^n\bigg|dy\\
   = &
 o\bigg(\Big(\frac{\epsilon}{|\ln\epsilon|^2}\Big)^{\frac{n-1}{n-2}}\bigg).
\end{align*}
The  last term in (\ref{ora}).
For $h=0$ and $ j>i$,
\begin{align*}
& \int_{\mathcal{A}_l} \Big| \Big(\sum_{i\neq l}^k(-1)^iPU_{ \mu_i,\xi_i}\Big) \psi^0_{ \mu_j,\xi_j}\Big|dx\\
  \leq & \sum_{i\neq l}^k \int_{\mathcal{A}_l}\big| U_{ \mu_i,\xi_i} \psi^0_{ \mu_j,\xi_j}\big|dx\\
   = & \frac{n-2}{2}\alpha_n^2
\sum_{i\neq l}^k \mu_i^{-\frac{n-2}{2}}\mu_j^{\frac{n-2}{2}}
   \int_{\frac{\mathcal{A}_l}{\mu_i}}\bigg|\frac{  1}{(1+|y-\sigma_i|^2)^{\frac{n-2}{2}}}
    \frac{ \Big(|y- \sigma_j\frac{\mu_j}{\mu_i}|^2-(\frac{\mu_j}{\mu_i})^2\Big)
    \mu_i^2}{\Big((\frac{\mu_j}{\mu_i})^2+| y- \sigma_j \frac{\mu_j}{\mu_i}|^2\Big)^{\frac{n}{2}} } \bigg|dy\\
   = &o\bigg(\Big(\frac{\epsilon}{|\ln\epsilon|^2}\Big)^{\frac{n-1}{n-2}}\bigg).
\end{align*}
For $h=1,\cdots,n$ and  $ j>i$,
\begin{align*}
 & \int_{\mathcal{A}_l}  \Big|\Big(\sum_{i\neq l}^k(-1)^iPU_{ \mu_i,\xi_i}\Big) \psi^h_{ \mu_j,\xi_j}\Big|dx\\
  \leq & \sum_{i\neq l}^k \int_{\mathcal{A}_l} |U_{ \mu_i,\xi_i} \psi^h_{ \mu_j,\xi_j}|dx\\
  = & (n-2)  \alpha_n^2 \sum_{i\neq l}^k \mu_j^{\frac{n}{2}}
 \int_{\frac{\mathcal{A}_l}{\mu_l}}
 \bigg|\frac{\mu_i^{-\frac{n-2}{2}}}{(1+|y-\sigma_i|^2)^{\frac{n-2}{2}}}
 \frac{  (y^h- \sigma_j\frac{\mu_j}{\mu_l})\mu_i }{\Big((\frac{\mu_j}{\mu_i})^2+| y- \sigma_j \frac{\mu_j}{\mu_i}|^2\Big)^{\frac{n}{2}}\mu_i^n} \bigg|\mu_i^ndy\\
 =& o\bigg(\Big(\frac{\epsilon}{|\ln\epsilon|^2}\Big)^{\frac{n-1}{n-2}}\bigg).
\end{align*}
Therefore,  we have
\begin{eqnarray}\label{lap}
&&  \int_\Omega f_0(V)  \psi^h_{ \mu_j,\xi_j}dx\nonumber\\
 & = & \left\{ \arraycolsep=1.5pt
  \begin{array}{lll}
   a_3\Big(\frac{\mu_{l+1}}{ \mu_l}\Big)^{\frac{n-2}{2}} g(\sigma_l)
   +o\bigg(\Big(\frac{\epsilon}{|\ln\epsilon|^2}\Big)^{\frac{n-1}{n-2}}\bigg)
  \ \   &{\rm if}\ l=1,\cdots,k-1,\ h=0, \\[2mm]
  o\bigg(\Big(\frac{\epsilon}{|\ln\epsilon|^2}\Big)^{\frac{n-1}{n-2}}\bigg) \ \   &{\rm if}\ l=k,
\end{array}
\right.
\end{eqnarray}
where $a_3$ and $g(\sigma_l)$ are defined in Proposition \ref{leftside}.
From (\ref{iak5})-(\ref{iak4}) and (\ref{lap}),
 we obtain
 \begin{eqnarray}\label{laps}
 \quad J_1
 =
\left\{ \arraycolsep=1.5pt
  \begin{array}{lll}
  -a_3\frac{\epsilon}{|\ln\epsilon|^2\sum\limits_{l=1}^{k-1}\Big(\frac{d_{l+1}}{ d_l}\Big)^{\frac{n-2}{2}}} g(\sigma_l)
  + o\bigg(\Big(\frac{\epsilon}{|\ln\epsilon|^2}\Big)^{\frac{n-1}{n-2}}\bigg)
  \ \   &{\rm if}\ l=1,\cdots,k-1,\ h=0, \\[2mm]
  o\bigg(\Big(\frac{\epsilon}{|\ln\epsilon|^2}\Big)^{\frac{n-1}{n-2}}\bigg) \ \   &{\rm if}\ l=k,\ h=1,\cdots,n.
\end{array}
\right.
\end{eqnarray}

\emph{Estimate of $J_2$}:
By Lemma \ref{exo}, there holds
\begin{align*}
J_2 = & \sum_{i=1}^k\int_\Omega (-1)^i\Big[f_0(PU_{ \mu_i,\xi_i})- f_0(U_{ \mu_i,\xi_i}) \Big] \psi^h_{ \mu_j,\xi_j}dx\\
 = &\sum_{i=1}^k\int_\Omega (-1)^i f^{'}_0(U_{ \mu_i,\xi_i})(PU_{ \mu_i,\xi_i}-U_{ \mu_i,\xi_i})\psi^h_{ \mu_j,\xi_j}dx\\
 & + \sum_{i=1}^k\int_\Omega(-1)^i \Big[
   f_0(PU_{ \mu_i,\xi_i})-f_0(U_{ \mu_i,\xi_i})-
   f^{'}_0(U_{ \mu_i,\xi_i})(PU_{ \mu_i,\xi_i}-U_{ \mu_i,\xi_i})\Big] \psi^h_{ \mu_j,\xi_j}dx\\
 = & (2^*-1)\sum_{i=1}^k \mu_i^{\frac{n-2 }{2}} \int_\Omega U_{ \mu_i,\xi_i}^{2^*-2}
     \Big(- \alpha_nH(x,\xi_i)+O( \mu_i^2)\Big)\psi^h_{ \mu_j,\xi_j}dx \\
 & - \int_\Omega (-1)^i \Big[f_0(PU_{ \mu_i,\xi_i})- f_0(U_{ \mu_i,\xi_i})
     -  f^{'}_0(U_{ \mu_i,\xi_i}) (PU_{ \mu_i,\xi_i}-U_{ \mu_i,\xi_i})\Big]  \psi^h_{ \mu_j,\xi_j}dx.
\end{align*}
Moreover, for  $ l=0,\cdots,n$, by  H\"{o}lder inequality, (\ref{subu2}),  (\ref{subu3}) and (\ref{sumbu1}), one has
\begin{eqnarray*}
 && \int_\Omega \Big|(-1)^i \Big[f_0(PU_{ \mu_i,\xi_i})- f_0(U_{ \mu_i,\xi_i})
     -  f^{'}_0(U_{ \mu_i,\xi_i}) (PU_{ \mu_i,\xi_i}-U_{ \mu_i,\xi_i})\Big]  \psi^h_{ \mu_j,\xi_j}\Big|dx\\
 & \leq & \Big|  f_0(PU_{ \mu_i,\xi_i})- f_0(U_{ \mu_i,\xi_i})
     -  f^{'}_0(U_{ \mu_i,\xi_i}) (PU_{ \mu_i,\xi_i}-U_{ \mu_i,\xi_i}) \Big|_{\frac{n}{2}} | \psi^h_{ \mu_j,\xi_j}|_{\frac{n}{n-2}}
     = o \bigg(\Big(\frac{\epsilon}{|\ln\epsilon|^2}\Big)^{\frac{n-1}{n-2}}\bigg).
\end{eqnarray*}
If $h=0$ and $j=i$, by (\ref{psi}), we get
\begin{align*}
 & (2^*-1)\alpha_n \sum_{i=1}^k \mu_i^{\frac{n-2}{2}}\int_\Omega  U_{\mu_i,\xi_i}^{2^*-2}
      H(x,\xi_i)  \psi^0_{ \mu_i,\xi_i}dx \\
  = & (2^*-1)\alpha_n  \sum_{i=1}^k \mu_i ^{\frac{3(n-2)}{2}} \mu_i ^{-\frac{n-2}{2}}
     \int_{\frac{\Omega-\xi_i}{ \mu_i }}  U^{2^*-2}(y)H( \mu_i y+\xi_i,\xi_i) \psi^0\Big(\frac{x-\xi_i}{\mu_i}\Big)dy \\
    = &
 \alpha_n a_1  \sum\limits_{i=1}^k \mu_i ^{ n-2 }
  H(\xi_i,\xi_i)+
  o\Big(\frac{\epsilon}{|\ln\epsilon|^2}\Big),
\end{align*}
where $  a_1 $ is given in Proposition \ref{leftside}.
If $h=1,\cdots,n$ and $j=i$, by (\ref{xy}),  one has
\begin{align*}
  &  (2^*-1) \alpha_n\sum_{i=1}^k\mu_i ^{\frac{n-2}{2}} \int_\Omega U_{\mu_i,\xi_i}^{2^*-2}(x)
H(x,\xi_i) \psi^h_{ \mu_i,\xi_i}dx \\
  = &  \alpha_n \sum_{i=1}^k \mu_i ^{\frac{n}{2 }} \int_\Omega
H(x,\xi_i) \frac{\partial }{ \partial \xi_i^h} U^{2^*-1}_{ \mu_i,\xi_i} (x) dx \\
  = &  \alpha_n \sum_{i=1}^k \mu_i ^{\frac{n}{2 }}\Big( \mu_i ^{\frac{n-2}{2}}\frac{\partial }{ \partial \xi^h_i}
     \int_{\frac{\Omega-\xi_i}{ \mu_i }} U^{2^*-1}_{ \mu_i,\xi_i}(y)  H( \mu_i y+\xi_i,\xi_i) dy
     - \mu_i ^{\frac{n-2}{2}}\int_{\frac{\Omega-\xi_i}{ \mu_i }} U^{2^*-1}_{ \mu_i,\xi_i} \frac{\partial H( \mu_i y+\xi_i,\xi_i)}{ \partial \xi^h_i}dy\Big)\\
  = & \alpha_n a_2  \sum_{i=1}^k\mu_i^{n-1}\Big( \frac{\partial ( H(\xi_i,\xi_i)) }{ \partial \xi^h_i}
      -   \frac{\partial H(\xi_i,\xi_i)}{ \partial \xi^h_i} + O( \mu_i )\Big)\\
  = & \frac{1}{2}\alpha_n a_2  \sum_{i=1}^k\mu_i ^{n-1} \partial_{\xi^h_i}\rho(\xi_i)
  + o \bigg(\Big(\frac{\epsilon}{|\ln\epsilon|^2}\Big)^{\frac{n-1}{n-2}}\bigg).
\end{align*}
In a same way,  if   $h=1,\cdots,n$ and $j\neq i$, we have
\begin{align*}
  (2^*-1) \alpha_n\sum_{i=1}^k\mu_i ^{\frac{n-2}{2}} \int_\Omega U_{\mu_i,\xi_i}^{2^*-2}(x)
H(x,\xi_i) \psi^h_{ \mu_j,\xi_j}dx
  =  o\Big(\frac{\epsilon}{|\ln\epsilon|^2}\Big).
\end{align*}
If $h=0$ and $j\neq i$,
\begin{align*}
  (2^*-1)\alpha_n \sum_{i=1}^k \mu_i^{\frac{n-2}{2}}\int_\Omega  U_{\mu_i,\xi_i}^{2^*-2}
      H(x,\xi_i)  \psi^0_{ \mu_j,\xi_j}dx
  =o\Big(\frac{\epsilon}{|\ln\epsilon|^2}\Big).
\end{align*}
As a consequence, there holds
\begin{eqnarray*}
 J_2
   =
\left\{ \arraycolsep=1.5pt
   \begin{array}{lll}
   \alpha_n a_1d_1  \frac{\epsilon}{|\ln\epsilon|^2}
  H(\xi,\xi)
  + o\Big(\frac{\epsilon}{|\ln\epsilon|^2}\Big)\ \   &{\rm if}\ h=0, \\[2mm]
   \frac{1}{2}\alpha_n a_2  \Big(\frac{\epsilon}{|\ln\epsilon|^2}\Big)^{\frac{n-1}{n-2}} \partial_{\xi_h}\rho(\xi)
   + o\bigg(\Big(\frac{\epsilon}{|\ln\epsilon|^2}\Big)^{\frac{n-1}{n-2}}\bigg)\ \   &{\rm if}\ h=1,\cdots,n. \\[2mm]
\end{array}
\right.
\end{eqnarray*}

\emph{Estimate of $J_3$}:
By Taylor  expansion with respect to $\epsilon$, we have
\begin{align}\label{mrfa}
 J_3 =  & \int_\Omega\Big[f_0(V)-f_\epsilon(V) \Big] \psi^h_{ \mu_j,\xi_j}dx\nonumber\\
  = & \epsilon\int_\Omega  V^p\ln\ln(e+V)\psi^h_{ \mu_j,\xi_j}dx
  -\epsilon^2\int_\Omega V^p
   \Big(\ln\ln(e+V)\Big)^2
  \psi^h_{ \mu_j,\xi_j} dx.
\end{align}
For the second term in (\ref{mrfa}),
from Lemma \ref{zsj}  and the annulus given in (\ref{annulus}), it  holds
\begin{align}\label{ma1}
  &\int_\Omega \bigg| V^p
   \Big(\ln\ln(e+V)\Big)^2
  \psi^h_{ \mu_j,\xi_j} \bigg|dx\nonumber\\
  \leq & \int_\Omega \Big|\Big(\sum\limits_{i=1}^k (-1)^iU_{\mu_i,\xi_i}\Big)^p\Big[\ln\ln\Big(e+\sum\limits_{i=1}^k (-1)^iU_{ \mu_i,\xi_i}\Big)\Big]^2
  \psi^h_{ \mu_j,\xi_j}\Big|dx\nonumber\\
 = & \int_{B(\xi,\rho)} \bigg|\Big(\sum\limits_{i=1}^k (-1)^iU_{ \mu_i,\xi_i}\Big)^p
 \Big[\ln\ln\Big(e+\sum\limits_{i=1}^k (-1)^iU_{ \mu_i,\xi_i}\Big)\Big]^2
  \psi^h_{ \mu_j,\xi_j}\bigg|dx
  + o\Big(\frac{\epsilon}{|\ln\epsilon|^2}\Big)\nonumber\\
  = & \sum_{i=1}^k\int_{\mathcal{A}_i}\bigg|\Big((-1)^i U_{ \mu_i,\xi_i}+(-1)^j\sum\limits_{j\neq i}^k U_{ \mu_j,\xi_j}\Big)^p\nonumber\\
&\quad \times \Big[\ln\ln\Big(e+(-1)^i U_{ \mu_i,\xi_i}+(-1)^j\sum\limits_{j\neq i}^k U_{ \mu_j,\xi_j}\Big)\Big]^2
  \psi^h_{ \mu_j,\xi_j}\bigg|dx
  +
  o\Big(\frac{\epsilon}{|\ln\epsilon|^2}\Big),
\end{align}
and on each annulus  $\mathcal{A}_i$,  we change variable setting $\mu_iy=x-\xi$,
for $h=0$, by (\ref{psi}),  then
\begin{align*}
   &  \int_{\mathcal{A}_i}\bigg|\Big((-1)^i U_{ \mu_i,\xi_i}+(-1)^j\sum\limits_{j\neq i}^k U_{ \mu_j,\xi_j}\Big)^p
  \Big[\ln\ln\Big(e+(-1)^i U_{ \mu_i,\xi_i}+(-1)^j\sum\limits_{j\neq i}^k U_{ \mu_j,\xi_j}\Big)\Big]^2
  |\psi^0_{ \mu_j,\xi_j}|\bigg|dx\nonumber\\
   = & \frac{(n-2)\alpha_n^{p+1}}{2}   \mu_j^{\frac{n-2}{2}}\int_{\frac{\mathcal{A}_i}{\mu_i}}
  \bigg|\frac{1}{(1+|y-\sigma_i|^2)^{\frac{n-2}{2}}}
  + \mu_i^{\frac{n-2}{2}}\sum\limits_{j\neq i}^k
  \frac{\mu_j^{\frac{n-2}{2}} }{(\mu_j^2+|\mu_iy-\mu_j\sigma_j|^2)^{\frac{n-2}{2}}}
   \bigg|^p\nonumber\\
  & \times  \bigg|\ln\ln\bigg[\alpha_n\mu_i^{-\frac{n-2}{2}}
  \bigg(e+\frac{1}{(1+|y-\sigma_i|^2)^{\frac{n-2}{2}}}
  +\mu_i^{\frac{n-2}{2}} \sum\limits_{j\neq i}^k
  \frac{\mu_j^{\frac{n-2}{2}} }{(\mu_j^2+|\mu_iy-\mu_j\sigma_j|^2)^{\frac{n-2}{2}}}\bigg)\bigg]\bigg|^2 \nonumber\\
  &   \times   \bigg| \frac{|\mu_iy-\mu_j\sigma_j|^2-\mu_j^2}{(\mu_j^2+|\mu_iy-\mu_j\sigma_j|^2)^{\frac{n}{2}}}\bigg|dy\nonumber\\
 = & O\bigg( \Big|\ln|\ln\mu_i|\Big|^2+\Big(\frac{\mu_i}{\mu_{i-1}}\Big)^{\frac{n}{2}}+\Big(\frac{\mu_{i+1}}{\mu_i}\Big)^{\frac{n}{2}}\bigg)
 = O( \ln|\ln\mu_i|)
 =O\Big( \ln\Big|\ln\frac{\epsilon}{|\ln\epsilon|^2}\Big|\Big).
\end{align*}
Similarly, for $h=1,\cdots,n$, we have
\begin{align*}
 &  \int_{\mathcal{A}_i}\Big((-1)^iU_{ \mu_i,\xi_i}+(-1)^j\sum\limits_{j\neq i}^k U_{ \mu_j,\xi_j}\Big)^p
  \Big[\ln\ln\Big(e+(-1)^iU_{ \mu_i,\xi_i}+(-1)^j\sum\limits_{j\neq i}^k U_{ \mu_j,\xi_j}\Big)\Big]^2
  |\psi^h_{ \mu_j,\xi_j}|dx\nonumber\\
   =  &  O\bigg( \ln|\ln\mu_i|+\Big(\frac{\mu_i}{\mu_{i-1}}\Big)^{\frac{n+2}{2}}+\Big(\frac{\mu_{i+1}}{\mu_i}\Big)^{\frac{n+2}{2}}\bigg)
 = O( \ln|\ln\mu_i|)
 =O\Big( \ln\Big|\ln\frac{\epsilon}{|\ln\epsilon|^2}\Big|\Big).
\end{align*}
Thus the second term in (\ref{mrfa}) becomes
\begin{align*}
\int_\Omega \Big|V^p
   \Big(\ln\ln(e+V)\Big)^2
  \psi^h_{ \mu_j,\xi_j} \Big|dx
 = O( (\ln|\ln\mu_i| )^2)
 = O\bigg(  \Big( \ln\Big|\ln\frac{\epsilon}{|\ln\epsilon|^2}\Big|\Big)^2\bigg).
\end{align*}
Consequently,
\begin{align}\label{ma}
 J_3 =  & \int_\Omega\Big(f_0(V)-f_\epsilon(V) \Big) \psi^h_{ \mu_j,\xi_j}dx\nonumber\\
   = & \epsilon\int_\Omega V^p\Big(\ln\ln(e+V)\Big)\psi^h_{ \mu_j,\xi_j}dx
   +O\bigg( \epsilon^2 \Big( \ln\Big|\ln\frac{\epsilon}{|\ln\epsilon|^2}\Big|\Big)^2\bigg).
\end{align}
Moreover,
\begin{align}\label{iam}
  &   \int_\Omega V^p\Big(\ln\ln(e+V)\Big)
      \psi^h_{ \mu_j,\xi_j}dx\nonumber\\
  = & \int_\Omega \Big(\sum\limits_{i=1}^k(-1)^iU_{ \mu_i,\xi_i}\Big)^p
  \Big[\ln\ln\Big(e+\sum\limits_{i=1}^k(-1)^i U_{ \mu_i,\xi_i}\Big)\Big]\psi^h_{ \mu_j,\xi_j}dx\nonumber\\
    &   -\bigg[
    \int_\Omega  \Big(\sum\limits_{i=1}^k (-1)^iU_{ \mu_i,\xi_i}\Big)^p  \Big[\ln\ln\Big(e+ \sum\limits_{i=1}^k (-1)^iU_{ \mu_i,\xi_i}\Big) \Big]
    -  V^p  \ln\ln(e+ V) \bigg]
         \psi^h_{ \mu_j,\xi_j} dx.
\end{align}
Let us set $h(u)=u^p\ln\ln(e+u)$,  by the mean value theorem,  one has
\[
0\leq h(u)-h(v)\leq Cu^{p-1}\Big(\ln\ln(e+u)+1\Big)(u-v)\quad\mbox{for}\ 0\leq v\leq u.
\]
Then
\begin{align*}
  &  \int_\Omega  \Big(\sum\limits_{i=1}^k (-1)^iU_{ \mu_i,\xi_i}\Big)^{p-1} \bigg(\ln\ln\Big(e+ \sum\limits_{i=1}^k (-1)^iU_{ \mu_i,\xi_i}  \Big)+1\bigg)
       \bigg(\sum\limits_{i=1}^k (-1)^iU_{ \mu_i,\xi_i}- V\bigg) \psi^h_{ \mu_j,\xi_j}dx\\
  = & \sum\limits_{i=1}^k \int_{\mathcal{A}_i}
     \Big(\sum\limits_{i=1}^k (-1)^iU_{ \mu_i,\xi_i}\Big)^{p-1}
     \bigg(\ln\ln\Big(e+ \sum\limits_{i=1}^k (-1)^iU_{ \mu_i,\xi_i}  \Big)+1\bigg)
       \sum\limits_{i=1}^k  (U_{ \mu_i,\xi_i}-   PU_{ \mu_i,\xi_i}) \psi^h_{ \mu_j,\xi_j} dx \\
      & + o\Big(\frac{\epsilon}{|\ln\epsilon|^2}\Big).
\end{align*}
Moreover,
on each annulus $\mathcal{A}_i$, $i=1,\cdots,k$,  if $h=0$, by Lemma \ref{exo} ,  (\ref{psi}) and the Lemma \ref{zsj},
we change variable setting $\mu_iy=x-\xi$,
then
\begin{eqnarray*}
  &&   \int_{\mathcal{A}_i}\bigg|
     \Big( (-1)^iU_{ \mu_i,\xi_i}+ (-1)^j\sum_{j\neq i}^kU_{ \mu_j,\xi_j}\Big)^{p-1}
     \bigg(\ln\ln\Big(e+  (-1)^iU_{ \mu_i,\xi_i}+ (-1)^j\sum_{j\neq i}^kU_{ \mu_j,\xi_j}\Big)+1\bigg)\\
    &&\times\sum\limits_{i=1}^k  (-1)^i(U_{ \mu_i,\xi_i}-   PU_{ \mu_i,\xi_i}) \psi^0_{ \mu_j,\xi_j} \bigg|dx \\
 & = & \alpha_n \sum\limits_{i=1}^k\mu_i^{\frac{n-2}{2}}
   \int_{\mathcal{A}_i}\bigg|  (-1)^i U_{ \mu_i,\xi_i}^{p-1}
     \Big(\ln\ln(e+  (-1)^iU_{ \mu_i,\xi_i})+1\Big)
      (H(x,\xi_i)+O( \mu_i^2)
      ) \psi^0_{ \mu_j,\xi_j} \bigg|dx\\
  &&    + o\Big(\frac{\epsilon}{|\ln\epsilon|^2}\Big)\\
 & = & \alpha_n \sum\limits_{i=1}^k\mu_i^{\frac{n-2}{2}}
  \int_{\frac{\mathcal{A}_i}{\mu_i}}\bigg| \frac{\mu_i^{-2}}{(1+|y-\sigma_i|^2)^2}
  \Big(\ln\ln \Big (e+ \frac{\mu_i^{-\frac{n-2}{2}}}{(1+|y-\sigma_i|^2)^{\frac{n-2}{2}}} \Big  )+1\Big) \\
 & \times& \Big(H(\xi_i+\mu_i y-\mu_i\sigma_i,\xi_i)+O( \mu_i^2)\Big)
 \frac{ |\mu_iy- \mu_j\sigma_j|^2-\mu_j^2 }{\Big( \mu_j^2+| \mu_iy-\mu_j \sigma_j |^2\Big)^{\frac{n}{2}} } \bigg|dy
 +o\Big(\frac{\epsilon}{|\ln\epsilon|^2}\Big)\\
 & \leq & \left\{ \arraycolsep=1.5pt
   \begin{array}{lll}
    \alpha_n \sum\limits_{i=1}^k\mu_i^{\frac{n-2}{2}}
     \Big(H(\xi_i,\xi_i)+O( \mu_i^2)\Big) \mu_j^{-(n-2)}(\ln|\ln \mu_i|+1)\\
\quad  \times \int_{\frac{\mathcal{A}_i}{\mu_i}} \frac{1}{(1+|y-\sigma_i|^2)^2}\frac{1}{|y|^{n-2}}dy
  +o\Big(\frac{\epsilon}{|\ln\epsilon|^2}\Big)\ \   {\rm if}\ j<i, \\[2mm]
  \alpha_n \sum\limits_{i=1}^k\mu_i^{-\frac{n-2}{2}}(\ln|\ln \mu_i|+1)
  \Big(H(\xi_i,\xi_i)+O( \mu_i^2)\Big)   \\
\quad \times \int_{\frac{\mathcal{A}_i}{\mu_i}} \frac{1}{(1+|y-\sigma_i|^2)^2}\frac{1}{|y-\frac{\mu_j}{\mu_i}\sigma_j|^{n-2}}dy
  +o\Big(\frac{\epsilon}{|\ln\epsilon|^2}\Big)\ \   {\rm if}\ j>i, \\[2mm]
  \end{array}
\right.\\
 & = &
\left\{ \arraycolsep=1.5pt
   \begin{array}{lll}
   O\bigg((\frac{\mu_i}{\mu_j}\Big)^{ \frac{n-2}{2}}\mu_j^{ -\frac{n-2}{2}}\ln |\ln \mu_i |\bigg)\ \   &{\rm if}\ h=0, \\[2mm]
  O\bigg(  \mu_i^{ -\frac{n-2}{2}}\ln |\ln \mu_i |\bigg)\ \   &{\rm if}\ h=1,\cdots,n, \\[2mm]
\end{array}
\right.\\
& = &
  O\bigg(\frac{\epsilon}{|\ln\epsilon|^2}
   \ln\Big|\ln\frac{\epsilon}{|\ln\epsilon|^2}\Big|\bigg).
\end{eqnarray*}

By the same argument, for $h=1,\cdots,n$, we have
\begin{align*}
 &  \int_{\mathcal{A}_i}\bigg|
     \Big( (-1)^iU_{ \mu_i,\xi_i}+(-1)^j\sum_{j\neq i}^kU_{ \mu_j,\xi_j} \Big)^{p-1}
     \bigg(\ln\ln\Big(e+ (-1)^iU_{ \mu_i,\xi_i}+(-1)^j\sum_{j\neq i}^kU_{ \mu_j,\xi_j}\Big )+1\bigg)\\
 &\quad \times \sum\limits_{i=1}^k (-1)^i (U_{ \mu_i,\xi_i}-   PU_{ \mu_i,\xi_i})\psi^h_{ \mu_j,\xi_j}\bigg|dx\\
 =  & O\Big (\sum\limits_{i=1}^k \mu_i^{\frac{n-2}{2}}\ln |\ln \mu_i | \Big)
  = O\bigg(\Big(\frac{\epsilon}{|\ln\epsilon|^2}\Big)
   ^{\frac{1}{2}}\ln\Big|\ln\frac{\epsilon}{|\ln\epsilon|^2}\Big|\bigg).
\end{align*}
Therefore,  the second term in (\ref{iam}) becomes
\begin{align*}
 &   \int_\Omega \bigg| \Big(\sum\limits_{i=1}^k (-1)^iU_{ \mu_i,\xi_i}\Big)^p\bigg(\ln\ln\Big(e+ \sum\limits_{i=1}^k (-1)^iU_{ \mu_i,\xi_i}\Big)  \bigg)
          \psi^h_{ \mu_j,\xi_j} \bigg|dx
 - \int_\Omega  \Big|V^p \ln\ln(e+ V)
         \psi^h_{ \mu_j,\xi_j}\Big| dx\\
   = & O\Big(\sum\limits_{i=1}^k \mu_i^{\frac{n-2}{2}}\ln |\ln \mu_i |\Big)
   =  O\bigg(\Big(\frac{\epsilon}{|\ln\epsilon|^2}\Big)
   ^{\frac{1}{2}}\ln\Big|\ln\frac{\epsilon}{|\ln\epsilon|^2}\Big|\bigg).
\end{align*}

For the first term in (\ref{iam}),  if $j=i$, from Lemma \ref{zsj}, (\ref{psi}),
let $x-\xi=\mu_iy$,
then\\
\begin{align}\label{maa}
  &   \int_\Omega \bigg| \Big(\sum\limits_{i=1}^k(-1)^iU_{ \mu_i,\xi_i} \Big)^p\Big[\ln\ln \Big(e+\sum\limits_{i=1}^k
        (-1)^iU_{ \mu_i,\xi_i} \Big)\Big]\psi^h_{ \mu_i,\xi_i}\bigg|dx\nonumber\\
   = & \sum\limits_{i=1}^k \int_{\mathcal{A}_i}
    \bigg|\Big((-1)^iU_{ \mu_i,\xi_i}+(-1)^j\sum\limits_{j\neq i}^k U_{ \mu_j,\xi_j} \Big)^p\nonumber\\
 &\times  \Big[\ln\ln\Big(e+ (-1)^i U_{ \mu_i,\xi_i}+(-1)^j\sum\limits_{j\neq i}^k U_{ \mu_j,\xi_j}  \Big)\Big]\psi^h_{ \mu_i,\xi_i}\bigg|dx
   +  o\Big(\frac{\epsilon}{|\ln\epsilon|^2}\Big)\nonumber\\
  = & \mu_1^{n-\frac{n+2}{2}-\frac{n-2}{2}}
  \int_{\frac{\mathcal{A}_i}{\mu_i}} \frac{1}{(1+|y-\sigma_i|^2)^{\frac{n+2}{2}}}
 \bigg|\ln\ln\bigg(e+(-1)^i\mu_i^{ -\frac{n-2}{2}} \frac{1}{(1+|y-\sigma_i|^2)^{\frac{n-2}{2}}}\nonumber\\
 & +(-1)^j\sum\limits_{j\neq i}^k  \frac{\mu_j^{\frac{n-2}{2}}}{\Big(  (\frac{\mu_j}{\mu_i} )^2 +|y-\frac{\sigma_i}{\mu_i}|^2\Big)^{\frac{n+2}{2}}\mu_i^{n-2}} \bigg)
  \psi^h(y)\bigg|dy
  +  o\Big(\frac{\epsilon}{|\ln\epsilon|^2}\Big)\nonumber\\
  = &\sum\limits_{i=1}^k \ln\Big|\ln\mu_i^{ -\frac{n-2}{2}}\Big|
  \int_{\frac{\mathcal{A}_i}{\mu_i}} \frac{1}{(1+|y-\sigma_i|^2)^{\frac{n+2}{2}}}\psi^h(y)dy
   + \sum\limits_{i=1}^k \frac{1}{|\ln \mu_i|}   \int_{\frac{\mathcal{A}_i}{\mu_i}}
  \frac{1}{(1+|y-\sigma_i|^2)^{\frac{n+2}{2}}}\nonumber\\
  & \times \bigg[|\ln  \mu_i|\ln
     \bigg(1+\frac{ \ln \Big[e^{1-\frac{n-2}{2}|\ln \mu_i|}+\frac{1}{(1+|y-\sigma_i|^2)^{\frac{n+2}{2}}}\Big]}{\frac{n-2}{2}
     |\ln \mu_i|}\bigg)\bigg]\psi^h(y)dy
     + o\Big(\frac{\epsilon}{|\ln\epsilon|^2}\Big).
\end{align}
Moreover, let
$\Lambda(y)=\frac{1}{(1+|y-\sigma_i|^2)^{\frac{n+2}{2}}} |\ln\mu_i|\ln
     \bigg(1+\frac{ \ln \Big(e^{1-\frac{n-2}{2}|\ln  \mu_i|}+\frac{1}{(1+|y-\sigma_i|^2)^{\frac{n+2}{2}}}\Big)}{\frac{n-2}{2}
     |\ln  \mu_i|}\bigg)\psi^h(y)$ for
$h=1,\cdots,n$.
Since $\psi(y)$ is a odd function,
we deduce  $\int_{\mathbb{R}^n}\Lambda(y)dy=0$.
Further,  by Lemma \ref{zsj}, there holds
\begin{align*}
 &  \int_{\mathbb{R}^n}\Lambda(y)dy-\int_{\frac{\mathcal{A}_i}{\mu_i}}\Lambda(y)dy
 =\int_{\mathbb{R}^n \backslash {\frac{\mathcal{A}_i}{\mu_i}}}\Lambda(y)dy\\
  \leq & C  |\ln \mu_i |\ln
     \bigg(1+\frac{ \ln \Big(e^{1-\frac{n-2}{2}|\ln \mu_i|}+\alpha_n  \Big)}{\frac{n-2}{2}|\mu_i|}\bigg)
     \int_{\mathbb{R}^n \backslash {\frac{\mathcal{A}_i}{\mu_i}}}
     \frac{1}{(1+|y-\sigma_i|^2)^{\frac{n+2}{2}}} |\psi^h(y)|dy\\
   \leq & C\Big(\frac{2}{n-2}\ln \alpha_n+o(1)\Big)\mu_i^{ n+1}
   =O(\mu_i^{ n+1})
   = o\bigg(\Big(\frac{\epsilon}{|\ln\epsilon|^2}\Big)
   ^{\frac{n-1}{n-2}}\bigg).
\end{align*}
Hence, for $ \mu_i$ small enough, we conclude that
\begin{align}\label{ma2}
  &  \int_\Omega \Big (\sum\limits_{i=1}^k(-1)^iU_{ \mu_i,\xi_i}\Big )^p\ln\ln \Big(e+\sum\limits_{i=1}^k(-1)^i
        U_{ \mu_i,\xi_i} \Big)\psi^h_{ \mu_j,\xi_j}dx\nonumber\\
  = & O\Big(\sum\limits_{i=1}^k \mu_i^{ n+1} \ln\Big|\ln \mu_i^{ -\frac{n-2}{2}}\Big|\Big)
  =o\bigg(\Big(\frac{\epsilon}{|\ln\epsilon|^2}\Big)^{\frac{n-1}{n-2}}\bigg),\quad
   \mbox{for}\
h=1,\cdots,n.
\end{align}

When  $h=0$, since $U_{ \mu_i,\xi_i}$ is the unique positive solution of problem  (\ref{limiequ}) and also, $\psi_{ \mu_i,\xi_i}^0$ solves (\ref{linear-equ}), then
\begin{equation*}
 \int_{\mathbb{R}^n} U_{ \mu_i,\xi_i}^p \psi_{ \mu_i,\xi_i}^0 dx
 =\int_{\mathbb{R}^n} \nabla U_{ \mu_i,\xi_i} \nabla \psi_{ \mu_i,\xi_i}^0 dx=p\int_{\mathbb{R}^n} U_{ \mu_i,\xi_i}^p\psi_{ \mu_i,\xi_i}^0 dx,
\end{equation*}
which follows that
\begin{equation*}
   \langle U_{ \mu_i,\xi_i},\psi_{ \mu_i,\xi_i}^0 \rangle=\int_{\mathbb{R}^n} U_{ \mu_i,\xi_i}^p\psi_{ \mu_i,\xi_i}^0dx=0.
\end{equation*}
Thus,   $\int_{\mathbb{R}^n} \frac{1}{(1+|y-\sigma_i|^2)^{\frac{n+2}{2}}}\psi^0(y)dx=0$ holds,
also,
 $
\int_{\mathbb{R}^n \backslash {\frac{\mathcal{A}_i}{\mu_i}}}
\frac{1}{(1+|y-\sigma_i|^2)^{\frac{n+2}{2}}}\psi^0dy=0.
$
From (\ref{maa}) and
Lemma \ref{zsj}, there holds
 \begin{align*}
  &   \int_\Omega \bigg|\Big(\sum\limits_{i=1}^k(-1)^iU_{ \mu_i,\xi_i}\Big)^p\ln\ln\Big(e+\sum\limits_{i=1}^k(-1)^i
        U_{ \mu_i,\xi_i}\Big)\psi^0_{ \mu_i,\xi_i}\bigg|dx\nonumber\\
  = & \sum\limits_{i=1}^k \frac{1}{|\ln\mu_i|}\frac{2}{n-2}\int_{\mathbb{R}^n}\bigg|
     \frac{1}{(1+|y-\sigma_i|^2)^{\frac{n+2}{2}}}
     \ln\Big(\frac{1}{(1+|y-\sigma_i|^2)^{\frac{n+2}{2}}}\Big)\psi^0(y)\bigg|dy +o\Big(\frac{\epsilon}{|\ln\epsilon|^2}\Big).
 \end{align*}
On the other hand,
\begin{eqnarray*}
&&\int_\Omega\bigg|  \Big(\sum\limits_{i=1}^k (-1)^iU_{ \mu_i,\xi_i} \Big)^p\ln\ln \Big(e+\sum\limits_{i=1}^k(-1)^i
        U_{ \mu_i,\xi_i} \Big)\psi^h_{ \mu_i,\xi_i}\bigg|dx\\
        &  = &
\left\{ \arraycolsep=1.5pt
   \begin{array}{lll}
   \frac{2}{n-2}a_4\sum\limits_{i=1}^k\frac{\epsilon}{ |\ln \mu_i |} + O\Big(\frac{1}{|\ln\mu_1|}  \Big)\ \   &{\rm if}\ h=0, \\[2mm]
    o\bigg(\Big(\frac{\epsilon}{|\ln\epsilon|^2}\Big)^{\frac{n-1}{n-2}}\bigg)\ \   &{\rm if}\ h=1,\cdots,n. \\[2mm]
\end{array}
\right.\\
& = &
\left\{ \arraycolsep=1.5pt
   \begin{array}{lll}
      \frac{2}{n-2}a_4\frac{\epsilon}{|\ln\epsilon|^2}
  + \frac{2}{n-2}a_4\sum\limits_{i=1}^k |\ln d_i|
   + o\Big(\frac{\epsilon}{|\ln\epsilon|^2}\Big)\ \   &{\rm if}\ h=0, \\[2mm]
   o\bigg(\Big(\frac{\epsilon}{|\ln\epsilon|^2}\Big)^{\frac{n-1}{n-2}}\bigg)\ \   &{\rm if}\ h=1,\cdots,n. \\[2mm]
\end{array}
\right.\\
\end{eqnarray*}
By the same argument,  if $j\neq i$, one has
\begin{equation*}
\int_\Omega\bigg| \Big (\sum\limits_{i=1}^k(-1)^iU_{ \mu_i,\xi_i} \Big)^p\ln\ln \Big(e+\sum\limits_{i=1}^k(-1)^i
        U_{ \mu_i,\xi_i} \Big)\psi^h_{ \mu_i,\xi_i}\bigg|dx=
    o\bigg(\Big(\frac{\epsilon}{|\ln\epsilon|^2}\Big)^{\frac{n-1}{n-2}}\bigg).
\end{equation*}
Consequently,
\begin{eqnarray*}
 J_3  & = &   \int_\Omega\Big(f_0(V)-f_\epsilon(V) \Big) \psi^h_{ \mu_j,\xi_j}dx\\
 &  = &
\left\{ \arraycolsep=1.5pt
   \begin{array}{lll}
  - \frac{2}{n-2}a_4\frac{\epsilon}{|\ln\epsilon|^2}
  - \frac{2}{n-2}a_4\sum\limits_{i=1}^k |\ln d_i|
  + o\Big(\frac{\epsilon}{|\ln\epsilon|^2} \Big)\ \   \ & {\rm if}\  h=0, \\[2mm]
   O\bigg (\Big(\frac{\epsilon}{|\ln\epsilon|^2}\Big)
   ^{\frac{1}{2}}\ln\Big|\ln\frac{\epsilon}{|\ln\epsilon|^2}\Big|\bigg)\ \  \  & {\rm if}\  h=1,\cdots,n. \\[2mm]
\end{array}
\right.
\end{eqnarray*}
Combining $J_1$-$J_3$,  the proof of this lemma   is completed.

Finally, let us state that $a_4$ is a positive constant.
From (\ref{psi0}),  polar coordinates, integrating
by parts, changing variables ($s = r^2$, $dr = \frac{1}{2}s^{-\frac{1}{2}}ds$), and a fact that $|\partial B_1(0) |=\frac{2\pi^{\frac{n}{2}}}{\Gamma(\frac{n}{2})}$,  we obtain
\begin{align*}
a_4 = & - \int_{\mathbb{R}^n}
     \frac{1}{(1+|y-\sigma_i|^2)^{\frac{n+2}{2}}}
     \ln\Big(\frac{1}{(1+|y-\sigma_i|^2)^{\frac{n+2}{2}}}\Big)\psi^0(y)dy\\
  = &  -  \int_{B_1(0)}
     \frac{1}{(1+|y-\sigma_i|^2)^{\frac{n+2}{2}}}
     \ln\Big(\frac{1}{(1+|y-\sigma_i|^2)^{\frac{n+2}{2}}}\Big)\psi^0(y)dy \\
 = &  -  \frac{(n-2)\alpha_n^{2^*}}{2}|\partial B_1(0) |  \int_0^\infty
     \frac{r^{n-1}}{(1+r^2)^{\frac{n+2}{2}}}
     \ln\Big(\frac{1}{(1+r^2)^{\frac{n+2}{2}}}\Big)
\frac{r^2-1}{(1+r^2)^{\frac{n }{2}}} dr \\
 = &  \frac{(n-2)^2\alpha_n^{2^*}}{4}\frac{2\pi^{\frac{n}{2}}}{\Gamma(\frac{n}{2})}
 \int_0^\infty
     \frac{r^{n-1}(r^2-1)}{(1+r^2)^{n+1}}\ln (1+r^2)  dr \\
 = &  \frac{(n-2)^2\alpha_n^{2^*}}{2}\frac{ \pi^{\frac{n}{2}}}{\Gamma(\frac{n}{2})}
 \int_0^\infty\frac{1}{n}\Big(\frac{r}{1+r^2}\Big)^n\frac{2r}{1+r^2} dr\\
 = &  \frac{(n-2)^2\alpha_n^{2^*}}{2n}\frac{ \pi^{\frac{n}{2}}}{\Gamma(\frac{n}{2})}
 \int_0^\infty \frac{s^{\frac{n}{2}}}{(1+s)^{n+1}}ds
 =   \frac{(n-2)^2\alpha_n^{2^*}}{2n}\frac{ \pi^{\frac{n}{2}}}{\Gamma(\frac{n}{2})}
 B(\frac{n}{2}+1,\frac{n}{2})ds\\
 =  &  \frac{(n-2)^2\alpha_n^{2^*}}{2n}\frac{ \pi^{\frac{n}{2}}}{\Gamma(\frac{n}{2})}
 \frac{\Gamma(\frac{n}{2}+1)\Gamma(\frac{n}{2})}{\Gamma( n+1)}
 =  \frac{\Gamma(\frac{n}{2} )\pi^{\frac{n}{2}}}{4\Gamma( n+1)} n^\frac{n}{2}  (n-2)^{\frac{n+4}{2}}.
\end{align*}
\end{proof}

\subsection*{Acknowledgments}
The authors were supported by National Natural Science Foundation of China 11971392.

\end{document}